\documentclass[a4paper,11pt]{article}
\usepackage[utf8]{inputenc}
\usepackage[T1]{fontenc}
\usepackage[english]{babel}
\usepackage[margin=2.8cm]{geometry}
\usepackage{mathpazo}
\usepackage[scaled=.95]{helvet}
\usepackage{courier}
\usepackage{amsmath,amsfonts,amssymb,amsthm}
\usepackage{graphicx}
\usepackage{float}
\usepackage{natbib}
\usepackage{color}
\usepackage[onehalfspacing]{setspace}

\newtheorem{prop}{Proposition}[section]

\newtheorem{rmq}{Remark}[section]  
\newtheorem{theo}{Theorem}[section]  
\newtheorem{lem}{Lemma}[section]

\newtheorem{cor}{Corollary}[section]
\DeclareMathOperator{\argmin}{argmin}

\DeclareUnicodeCharacter{B0}{\textdegree}

\usepackage{hyperref}
\hypersetup{
pdfpagemode=UseOutlines,      % UseOutlines, UseThumbs, None, FullScreen : agencement au démarrage
pdfstartview=Fit,             % Fit, FitH, FitB, FitBH : vue de la page au départ (pleine largeur...)
pdffitwindow=true,            % bool: Maximiser
pdfpagelayout=TwoColumnsRight,% SinglePage, TwoColumnsRight/Left, OneColumn : affichage des pages
pdftoolbar=true,              % bool: Affichage de la barre d'outils
pdfmenubar=true,              % bool: Affichage de la barre de menus
bookmarksopen=false,          % bool: Dépliement des signets
bookmarksnumbered=true,       % bool: Numérotation des signets
colorlinks=true,              % bool: Liens colorés
pdfauthor={Ton nom},     % Auteur
pdftitle={Titre PDF},    % Titre
pdfcreator=PDFLaTeX,          % 
pdfproducer=PDFLaTeX,         %
linkcolor=blue,               % Couleur des liens
urlcolor=blue,                %             url
anchorcolor=blue,            %         du texte
citecolor=blue,              % Couleur de citation 
frenchlinks=true,             % bool: Utiliser des petites majuscules pour les liens, plutôt que de la couleur
pdfborder={0 0 0}             % Ne pas encadrer les liens
}
	\usepackage{xr}

\usepackage[usenames,dvipsnames]{xcolor}
\usepackage[colorinlistoftodos,textwidth=2.3cm]{todonotes}

\title{Non asymptotic analysis of Adaptive stochastic gradient algorithms and applications}
\author{Antoine Godichon-Baggioni $^*$ \and Pierre Tarrago \thanks {Laboratoire de Probabilités, Statistique et Modélisation, Sorbonne Université, France \newline
antoine.godichon\_baggioni@sorbonne-universite.fr, pierre.tarrago@sorbonne-universite.fr \newline This project is supported by the Agence Nationale de la Recherche funding CORTIPOM ANR-21-CE40-0019.}}

\begin{document}

\date{}

\maketitle

\begin{abstract}
In stochastic optimization, a common tool to deal sequentially with large sample is to consider the well-known stochastic gradient algorithm. Nevertheless, since the stepsequence is the same  for each direction, this can lead to bad results in practice in case of ill-conditionned problem. To overcome this, adaptive gradient algorithms such that Adagrad or Stochastic Newton algorithms should be prefered. This paper is devoted to the  non asymptotic  analyis of these adaptive gradient algorithms for strongly convex objective. All the theoretical results will be adapted to linear regression and regularized generalized linear model for both Adagrad and Stochastic Newton algorithms.
\end{abstract}


\section{Introduction}
A usual problem in stochastic optimization is to estimate the minimizer $\theta$ of a convex functional $G : \mathbb{R}^{d} \longrightarrow \mathbb{R}$ of the form 
\[
G(h) =  \mathbb{E}\left[ g(X ,   h ) \right]
\]
where $g : \mathcal{X}\times \mathbb{R}^{d} \longrightarrow \mathbb{R}$, and  $X$ is a random variable lying in $\mathcal{X}$. Indeed, this is the case for usual regressions such that the linear and logistic ones \citep{bach2014adaptivity} or the estimation of the geometric median and quantiles \citep{HC,CCG2015,godichon2015} to name a few. Several techniques have been developed to estimate the solution of the problem, which can be split into two main branches: iterative and recursive methods. Iterative methods consist in considering the empirical function generated by  the sample and to approximate its minimizer with the help of usual convex optimization methods \citep{boyd2004convex} or considering some refinements such that mini-batch algorithms \citep{konevcny2015mini}. Although these methods are known to be very competitive, they can encounter computational problems to deal with large samples. In addition, they are not suitable for dealing with data arriving sequentially, and one can so focus on recursive methods.

One of the most famous and studied recursive method is unquestionably the stochastic gradient algorithm \citep{robbins1951} and its averaged version \citep{ruppert1988efficient,PolyakJud92}. Considering data $X_{1} , \ldots ,X_{n}, X_{n+1},\ldots$ arriving sequentially, it is defined recursively for all $n \geq 0$ by
\begin{align*}
& \theta_{n+1} = \theta_{n} - \gamma_{n+1} \nabla_{h}g \left( X_{n+1} , \theta_{n} \right),
& \overline{\theta}_{n+1} = \overline{\theta}_{n} + \frac{1}{n+2}\left( \theta_{n+1} - \overline{\theta}_{n} \right)
\end{align*}
where $\left( \gamma_{n} \right)$ is a positive step sequence converging to $0$. These estimates are studied for a while: one can refer to \citep{pelletier1998almost,Pel00} for some asymptotic results while one can refer to more recent literature for non asymptotic results such that convergence in quadratic mean of the estimates \citep{bach2013non,gadat2017optimal,gower2019sgd,GB2021}. The averaged estimates are known to be asymptotically efficient and achieve the Cramer-Rao bound (up to rest terms) under some regularity assumptions.

Nevertheless, the step sequence $\left( \gamma_{n} \right)$ cannot be adapted to each direction of the gradient which can lead to bad results in practice for ill-conditioned problems. In order to alleviate this, one can more focus on adaptive stochastic gradient algorithms of the form
\[
\theta_{n+1} = \theta_{n} - \gamma_{n+1} A_{n} \nabla_{h} g \left( X_{n+1} , \theta_{n} \right)
\]
where $\left( A_{n} \right)$ is a sequence of (random) matrices which enables to be adapted to each coordinate. One of the most famous adaptive algorithm is Adagrad \citep{duchi2011adaptive}, which can be seen as a way to standardize the gradient $\nabla_{h} g \left( X_{n+1} , \theta \right)$. In recent works, \cite{BGBP2019} and \cite{BGB2020} consider $\left( A_{n} \right)$ as a sequence of estimates of the inverse of the Hessian, leading to a Stochastic Newton algorithm. This last method is of particular interest in the case where the Hessian of the function we would like to minimize has eigenvalues at different scales for instance.

Remark that several asymptotic results exist on adaptive method and one can focus on the recent works  of \cite{leluc2020asymptotic} or \cite{gadat2020asymptotic} among others, while non asymptotic results are less usual. Nevertheless, in a recent work, \cite{Joseph} give bounds with high probabilities in the special case of Kalman recursion for logistic regression, while  \cite{defossez2020simple} focus on the $L^{2}$ rates of convergence for Adagrad and Adam. Furthermore, \cite{bercu2021stochastic} obtain the rate of convergence in quadratic mean of Stochastic Gauss-Newton algorithms for optimal transport. Note however that in all these cases, the gradient of $g$ is supposed to be uniformly bounded. 

In this paper, we focus on non asymptotic rates of convergence for strongly convex functions (and so, with unbounded gradient). More precisely, we propose a first rate of convergence of Adaptive estimates in the case where the sequence $A_{n}$ possibly diverges, but with a control on this possible divergence. Supposing in addition that $A_{n}$ admits an uniform fourth order moment, we establish that $\mathbb{E}\left[ G\left( \theta_{n} \right) - G(\theta) \right]$ converges at the usual rate of convergence. Finally, we establish a  non constraining general framework for obtaining the rate of convergence of Stochastic Newton and Adagrad algorithms. These results will be applied for linear regression and ridge generalized linear model. 

The paper is organized as follows: Section \ref{sec::framework}, the general framework is introduced. The algorithms and theoretical results of convergence are given in Section \ref{sec::adaptive} while applications consisting in the linear regression and the generalized linear model are respectively given in Sections \ref{sec::lm} and \ref{sec::glm}. The proofs are postponed in Section \ref{sec::proofs} and in Appendix.

\section{Framework}\label{sec::framework}
In what follows, we consider a random variable $X$ taking values in a measurable space $\mathcal{X}$ and fix $d\geq 2$. We focus on the estimation of the minimizer $ \theta$ of the convex function $G : \mathbb{R}^d \longrightarrow \mathbb{R}$ defined for all $h \in \mathbb{R}^d$ by
\[
G(h) := \mathbb{E}\left[ g \left( X, h \right) \right]
\]
with $g: \mathcal{X} \times \mathbb{R}^d \longrightarrow \mathbb{R}$. Let us suppose from now that the following assumptions are fulfilled:
\begin{itemize}
\item[\textbf{(A1)}] For almost every $x \in  \mathcal{X}$, the functional $g (x,. )$ is differentiable on $\mathbb{R}^d$  and there exists $p\geq 2$ and non-negative constants $C^{(p)}_1,\,C^{(p)}_2$ such that for all $h\in \mathbb{R}^d$,
$$\mathbb{E}\left[ \left\| \nabla_{h} g \left( X , h \right) \right\|^{2p} \right] \leq C_{1}^{(p)} + C_{2}^{(p)}  \left\| h - \theta \right\|^{2p}.$$

\item[\textbf{(A2)}] The functional $G$ is twice continuously differentiable.
\item[\textbf{(A3)}] The Hessian of $G$ is uniformly bounded on $\mathbb{R}^d$, i.e there is a positive constant $L_{\nabla G}$ such that for all $h \in \mathbb{R}^d$,
\[
\left\| \nabla^{2}G(h) \right\|_{op} \leq L_{\nabla G }
\]
where $\| . \|_{op}$ is the usual spectral norm for matrices.
\item[\textbf{(A4)}] The functional $G$ is $\mu$ quasi-strongly convex: for all $h \in \mathbb{R}^d$,
\[
\left\langle \nabla G \left( h \right) , h - \theta \right\rangle \geq \mu \left\| h-   \theta \right\|^{2} .
\]
\end{itemize}
Remark that in a particular case, Assumption \textbf{(A3)} ensures that the gradient of $G$ is $L_{\nabla G}$-Lipschitz. Note that these assumptions are usual for obtaining the $L^{2}$ rates of convergence of the stochastic gradient algorithms and their averaged versions \citep{bach2013non,gower2019sgd}.
\section{Adaptive stochastic gradient algorithms} \label{sec::adaptive}
\subsection{The algorithms}
Let $X_{1} , \ldots ,X_{n},X_{n+1}, \ldots$ be i.i.d copies of $X$. Then, an adaptive stochastic gradient algorithm is defined recursively for all $n \geq 0$ by
\[
\theta_{n+1} = \theta_{n} - \gamma_{n+1} A_{n} \nabla_{h} g \left( X_{n+1} , \theta_{n} \right) ,
\]
where $\theta_{0}$ is arbitrarily chosen, $\gamma_{n} = c_{\gamma}n^{-\gamma}$ with $c_{\gamma}> 0$, $\gamma \in (0,1)$ and $A_{n}$ is a sequence of symmetric and positive matrices such that there is a filtration $\left( \mathcal{F}_{n} \right)_{n \geq 0}$ satisfying:
\begin{itemize}
\item For all $n \geq 0$, $A_{n}$ is $\mathcal{F}_{n}$-measurable.
\item $X_{n+1}$ is independent of $\mathcal{F}_{n}$.
\end{itemize} 
Typically, one can consider $A_{n}$ only depending on $X_{1},\ldots ,X_{n}, \theta_{0} , \ldots , \theta_{n}$ and consider the filtration generated by the sample, i.e $\mathcal{F}_{n} = \sigma \left( X_{1} , \ldots , X_{n} \right) $.  Considering $A_{n}$ diagonal with $\left( A_{n}\right)_{k,k} = \left( \frac{1}{n+1} \left( a_{k} + \sum_{i=1}^{n} \nabla_{h}g \left( X_{i} , \theta_{i-1} \right)_{i,i}^{2} \right)\right)^{-1/2}$ leads to Adagrad algorithm \citep{duchi2011adaptive}. Furthermore, the case where $A_{n}$ is a recursive estimate of the inverse of the Hessian corresponds to the Stochastic Newton algorithm \citep{BGBP2019,BGB2020} while the case where $A_{n} = \frac{1}{n+1}\left( A_{0} + \sum_{i=1}^{n} \nabla_{h}g \left( X_{i}, \theta_{i-1} \right)\nabla_{h} g \left( X_{i} , \theta_{i-1} \right)^{T} \right)$ corresponds to the stochastic Gauss-newton algorithm \citep{CGBP2020,bercu2021stochastic}.

\subsection{Convergence results}
\subsubsection{A first convergence result}
In order to obtain a first rate of convergence  of the estimates, let us now introduce some assumptions on the sequence of random matrices $\left( A_{n} \right)_{n\geq 0}$:
\begin{itemize}
\item[\textbf{(H1 )}] One can control the smallest and largest eigenvalues of $A_{n}$: 
\begin{itemize}
\item[\textbf{(H1a)}]  There exists $(v_n)_{n\geq 0},\lambda_0>0$ and $\delta,q\geq 0$ such that 
\[
\mathbb{P}\left[ \lambda_{\min} \left( A_n \right) \leq \lambda_0t \right] \leq v_{n+1}t^q(n+1)^{-\delta},
\]
for $0<t\leq 1$, with $(v_{n+1}(n+1)^{-\delta})_{n\geq 0}$ decreasing.

If $\gamma\leq 1/2$, one also assumes the stronger hypothesis of the existence of $\lambda_n'=\lambda_0'(n+1)^{-\lambda'}$ with $\lambda_0'>0,\, \lambda'<\gamma$ such that for all $n\geq 0$,
$$\lambda_{\min}(A_n)\geq \lambda_n'.$$
\item[\textbf{(H1b)}] There exists $\beta_{n} = c_{\beta}n^{\beta}$ with $c_{\beta} \geq 0$ and $0<\beta<\frac{\gamma}{2}$ if $\gamma \leq 1/2$ or $0 < \beta < \gamma -1/2$ if $\gamma > 1/2$ such that for all $n \geq 0$,
\[
\left\| A_{n} \right\|_{op} \leq \beta_{n+1}.
\]
\end{itemize}
\end{itemize}

Remark that the case $\delta=0$ is allowed in \textbf{(H1a)} and that one can always choose $\beta$ in the allowed range of \textbf{(H1b)}. In most cases and especially for Adagrad and Stochastic Newton algorithm, \textbf{(H1a)} is easily verified. The presence of the decreasing term $v_n$ in \textbf{(H1a)} takes into account a general phenomenon (usually implied by Rosenthal inequality) that error contributions from higher moments of $X$, albeit dominant for small $n$, fade as $n$ goes to infinity. Concerning \textbf{(H1b)}, some counter-examples showing that the estimates possibly diverge in the case where this last assumption is not fulfilled  are given in Appendix \ref{sec::counter}, meaning that this assumption is unfortunately crucial. Anyway, an easy way to corroborate it is to replace the random matrices $A_{n}$ by
\[
\tilde{A}_{n} = \frac{\min \left\lbrace \left\| A_{n} \right\|_{op} , \beta_{n+1} \right\rbrace }{\left\| A_{n} \right\|_{op}} A_{n}
\]
and one can directly check that $\left\| \tilde{A}_{n} \right\|_{op} \leq \beta_{n+1}$.  Similar adjustment can be used to ensure \textbf{(H1a)} in the case $\gamma\leq 1/2$.

 Let us consider the case of Newton's method, and especially the case where the estimates of the Hessian are of the form $H_{n} = \frac{1}{n+1} \left( H_{0} + \sum_{k=1}^{n} a_{k}\Phi_{k}\Phi_{k}^{T} \right)$  and which can be so recursively invert with the help of Riccati/Shermann-Morrisson's formula (see \cite{BGBP2019,BGB2020,WEI}). In order to verify \textbf{(H1b)}, one can consider the following version of the estimate of the Hessian
\[
\tilde{H}_{n} = H_{n} + \frac{1}{n+1}\sum_{k=1}^{n} \frac{\tilde{c}_{\beta}}{k^{ {\beta}}}e_{k}e_{k}^{T}
\]
where $e_{k}$ is the $k$-th (modulo $d$) canonical vector  (see \cite{bercu2021stochastic,WEI}).
 We can now obtain a first rate of convergence of the estimates. For the sake of simplicity, let us now denote the risk error by $V_{n} := G \left( \theta_{n} \right) - G(\theta)$. Note that since $G$ is $\mu$ quasi-strongly convex, one has $\left\| \theta_{n} - \theta \right\|^{2} \leq \frac{2}{\mu} V_{n}$.

\begin{theo}\label{theo1}
Suppose Assumptions \textbf{(A1)} to \textbf{(A3)} and \textbf{(H1)} hold. Then,  for all $n \geq 1$ and for any $\lambda<\min\left\lbrace \gamma-2\beta,1-\gamma \right\rbrace$, 
\begin{align*}
\mathbb{E}\left[ V_{n} \right]   \leq \exp \left( -c_{\gamma} \mu \lambda_{0} n^{1-(\lambda+\gamma)} (1-\varepsilon(n) )\right)&\left( K_1^{(1)}+K_{1'}^{(1)}\max_{1 \leq k \leq n+1} k^{\gamma-2\beta-\delta/2-(q/2+1)\lambda}\sqrt{v_{k}}\right)\\
&\hspace{1,5cm}+K_2^{(1)}n^{-(\gamma-2\beta -\lambda)}+K_3^{(1)}\sqrt{v_{\lfloor n/2\rfloor}}n^{-(\delta+q\lambda)/2},
\end{align*}
with $\varepsilon(n)=o(1)$ given in \eqref{eq:espilon(n)} and $K_1^{(1)},\, K_{1'}^{(1)},\,K_2^{(1)},\,  K_3^{(1)}$ respectively given in \eqref{eq:3.1_first_constant} and \eqref{eq:3.1_second_constants}.
%$V^{2} =  \exp \left( C_{V^{2}} \right) \left(  \mathbb{E}\left[ V_{0}^{2} \right] + C_{V^{2}}' \right)$,  and
%\begin{align*}
%& C_{V^{2}} :=    \left( \frac{4C_{2}'}{\mu^{2}} + \frac{L_{\nabla G}}{2} \right) c_{\gamma}^{2}c_{\beta}^{2} \frac{2\gamma -2 \beta}{2\gamma -2 \beta-1} + \frac{4L_{\nabla G}C_{2}'}{\mu^{2}} c_{\gamma}^{3}c_{\beta}^{3} \frac{3\gamma - 3\beta}{3\gamma - 3 \beta -1}+ \frac{L_{\nabla G}^{2}C_{2}'}{\mu^{2}}c_{\gamma}^{4}c_{\beta}^{4}\frac{4\gamma - 4 \beta}{4\gamma - 4 \beta -1} \\
%& C_{V^{2}}' :=    \left( \frac{4C_{1}'}{\mu^{2}} + \frac{L_{\nabla G}}{2} \right) c_{\gamma}^{2}c_{\beta}^{2} \frac{2\gamma -2 \beta}{2\gamma -2 \beta-1} + L_{\nabla G}C_{1}' c_{\gamma}^{3}c_{\beta}^{3} \frac{3\gamma - 3\beta}{3\gamma - 3 \beta -1}+ \frac{L_{\nabla G}^{2}C_{1}'}{2}c_{\gamma}^{4}c_{\beta}^{4}\frac{4\gamma - 4 \beta}{4\gamma - 4 \beta -1} 
%\end{align*}

\end{theo}

In the particular case where $\delta/2\geq \gamma-2\beta$ (which happens as soon as $\delta\geq 1$), one can simply set $\lambda=0$ in the above formula : we will see that it is the case for the generalized linear model with the stochastic Newton algorithm. However, for Adagrad algorithms, one can not avoid using first $\lambda>0$, since $A_n$ depends on $\nabla g(X , \cdot)$ rather than $\nabla^2g(X, \cdot)$ (while the expectation of latter is bounded on $\mathbb{R}^d$, the one of the former is generally unbounded). To get rid of this weaker statement, we need the following equivalent of Theorem \ref{theo1} for higher moments. 
\begin{prop}\label{prop::ordrep'}
Suppose Assumptions \textbf{(A1)} with $p>2$, \textbf{(A2)}, \textbf{(A3)} and \textbf{(H1)} hold. Then for any $p'<p$ and any $\lambda<\min \lbrace\gamma-2\beta,1-\gamma \rbrace$,
\begin{align*}
\mathbb{E}\left[ V_{n}^{p'} \right]   \leq \exp \left( -c_{\gamma} \mu \lambda_{0} n^{1-(\lambda+\gamma)} (1-\varepsilon'(n)\right)&\left( K_1^{(1')}+K_{1'}^{(1')}\max_{1 \leq k \leq n+1} k^{\gamma-2\beta-\lambda-\frac{p-p'}{p}(\delta+q\lambda)}v_{k}^{\frac{p-p'}{p}}\right)\\
&\hspace{0,5cm}+K_2^{(1')}n^{-p'(\gamma-2\beta -\lambda)}+K_3^{(1')}v_{\lfloor n/2\rfloor}^{\frac{p-p'}{p}}(n+1)^{-\frac{p-p'}{p}(\delta+q\lambda)},
\end{align*}
with $\epsilon'(n)$,  $K^{(1')}_1$, $K^{(1')}_{1'}$, $K^{(1')}_2$ and $K^{(1')}_3$ respectively given in \eqref{eq:espilon(n)'}, \eqref{eq:3.1p_first_constant} and \eqref{eq:3.1p_second_constants}.
\end{prop}
\subsubsection{Convergence when $A_{n}$ has bounded moments}
In order to get a better rate of convergence, let us now introduce some new assumptions on the sequence of random matrices $\left( A_{n} \right)$:
\begin{itemize}
\item[\textbf{(H2a)}] The random matrices $A_{n}$ admit uniformly bounded second order moments: there is $C_{S}$ such that for all $n \geq 0$: 
\[
\mathbb{E}\left[ \left\| A_{n} \right\|^{2} \right] \leq C_{S}^{2}.
\]
\item[\textbf{(H2b)}] The random matrices $A_{n}$ admit uniformly bounded fourth order moments: there is $C_{S}$ such that for all $n \geq 0$: 
\[
\mathbb{E}\left[ \left\| A_{n} \right\|^{4} \right] \leq C_{S}^{4}.
\]
\end{itemize}
For a simpler statement, we assume here  and in the next paragraph that $q>0$ in \textbf{(H1a)}, although similar bound would hold in full generality.
\begin{theo}\label{theo2}
Suppose Assumptions \textbf{(A1)} to \textbf{(A3)} for some $p> 2$, \textbf{(H1)} and \textbf{(H2a)} hold with $\delta>0$. Then, for all $n \geq 0$,
\begin{align*}
\mathbb{E}\left[ V_{n} \right]  & \leq \exp \left( -  c_{\gamma}  \mu \lambda_{0} n^{1-\gamma} \left(1-\varepsilon(n)\right)\right)\cdot\left(K^{(2)}_1+K^{(2)}_{1'}\max_{1\leq k\leq n+1}v_k^{\frac{p-1}{p}}k^{\gamma - 2 \beta-\frac{p-1}{p}\delta}\right)\\
&\hspace{5cm}+ K^{(2)}_2v_{\lfloor n/2\rfloor}^{\frac{p-1}{p}}n^{-\frac{(p-1)}{p}\delta}+K^{(2)}_3n^{ - \gamma},
\end{align*}
where $\varepsilon(n)=o(1)$ is given in \eqref{eq:epsion(n)theo2}  and $K^{(2)}_1,K^{(2)}_{1'},K^{(2)}_2,K^{(2)}_3$ are respectively given in  \eqref{eq:3.2_2_constant}, \eqref{eq:3.2_3_constant} and \eqref{eq:3.2_4_constant}.
\end{theo}
Finally, in order to get the rate of convergence in quadratic mean of  Stochastic Newton estimates, we now give the $L^{2}$ rate of convergence of $G\left( \theta_{n} \right)$ when $\gamma>1/2$.
\begin{prop}\label{prop::ordre4}
Suppose Assumptions \textbf{(A1)} to \textbf{(A3)} for some $p>2$, \textbf{(H1)} and \textbf{(H2b)} hold with $\gamma>1/2,\delta > 0$ and $\beta<\gamma-1/2$. Then
\begin{align*}
\mathbb{E}\left[ V_{n}^{2} \right]  \leq &\exp \left( - \frac{3}{2}c_\gamma\lambda_{0}\mu n^{1-\gamma} \right)\left(K^{(2')}_1+K^{(2')}_{1'}\max_{1\leq k\leq n+1}v_k^{\frac{p-2}{p}}k^{\gamma-\frac{p-2}{p}\delta}\right) \\
 &\hspace{6cm}+K^{(2')}_2n^{-2\gamma} + K^{(2')}_3v_{\lfloor n/2\rfloor}^{(p-2)/p}n^{-\delta(p-2)/p} =: M_n.
 \end{align*}
with  $K^{(2')}_1$, $K^{(2')}_{1'}$, $K^{(2')}_2$, $K^{(2')}_3$ respectively given in \eqref{eq:prop31_constant1}, \eqref{eq:prop31_constant1'} and \eqref{eq:prop31_constant2}.
\end{prop}
Remark that one has $
M_{n} = O \left( n^{-\min\left\lbrace 2\gamma,\frac{\delta(p-2)}{p}\right\rbrace} \right) $.
Hence, for $\delta$ large enough (namely $\delta>\frac{2p}{p-2}\gamma$), the main contribution comes from the second term of the latter bound.
Then, for any $0\leq \gamma' \leq \min\left\lbrace 2\gamma,\frac{\delta(p-2)}{p}\right\rbrace $, only depending on $v_{n}$ and $\gamma$, we have
\begin{equation}
\label{eq::sup_Mn}w_\infty(\gamma'):=\sup_{n \geq 1} M_{n}n^{\gamma '}< + \infty.
\end{equation}
The function $w_\infty:\left[0,\min\left\lbrace 2\gamma,\frac{\delta(p-2)}{p}\right\rbrace\right]\rightarrow \mathbb{R}$ can be computed numerically, but in any case note that $
w_\infty(\gamma')\leq K^{(2')}_1 \sup_{t\geq 1}\left\lbrace t^{\gamma'}\exp \left( - \frac{1}{2}\lambda_{0}\mu t^{1-\gamma} \right)\right\rbrace +K^{(2')}_2+ K^{(2')}_3$, so that a function analysis yields, for $\gamma'\in\left[ 0,\min\left\lbrace 2\gamma,\frac{\delta(p-2)}{p}\right\rbrace \right]$, 
\begin{equation}\label{eq:bound_w_infty}
w_\infty(\gamma')\leq K^{(2')}_1\left(\frac{2\gamma'}{\lambda_0\mu e(1-\gamma)}\right)^{\frac{\gamma'}{1-\gamma}}+K^{(2')}_2+ K^{(2')}_3.
\end{equation}
We will see in most applications that under suitable assumptions, $\gamma' $ can be equal to $2 \gamma$ (namely when $\delta\geq\frac{2p}{p-2}\gamma$). %For $\gamma$ close to $1$, it may be wise to choose $\gamma'$ of order $1-\gamma$, so that $w_{\infty}(\gamma')$ does not explode.

\subsubsection{Convergence results for stochastic Newton algorithms}
Let us now focus on the rate of convergence of Stochastic Newton algorithm. In this aim, let us denote $H := \nabla^{2}G(\theta) $ and let us suppose from now that the following assumptions are fulfilled too:
\begin{itemize}
\item[\textbf{(A1')}] There is $L_{\nabla g}$ such that for all $h \in \mathbb{R}^d$,
\begin{equation}
\mathbb{E}\left[ \left\|  \nabla_{h}g \left( X , h \right) - \nabla_{h}g \left( X, \theta \right)   \right\|^{2} \right] \leq L_{\nabla g} \left\| h - \theta \right\|^{2} 
\end{equation}
\item[\textbf{(A5)}] There is a non negative constant $L_{\delta}$ such that for all $h \in \mathbb{R}^d$,
\[
\left\| \nabla G(h) - \nabla^{2}G (\theta ) \left( \theta - h \right) \right\| \leq L_{\delta} \left\| h - \theta \right\|^{2}
\]\item[\textbf{(H3)}] The estimate $A_{n}$ converges to $H^{-1}$: there is a decreasing positive sequence $\left( v_{A,n} \right)$ such that for al $n \geq 0$, 
\[
\mathbb{E}\left[ \left\| A_{n} - H^{-1} \right\|^{2} \right] \leq v_{A,n}.
\]
\end{itemize}
Observe that assumption \textbf{(A1')} is often called expected smoothness in the literature \citep{bach2013non} and is satisfied in most of examples such that linear and logistic regression \citep{bach2013non,bach2014adaptivity} or the estimation of geometric quantiles and medians \citep{HC} among others. Concerning \textbf{(A5)}, under \textbf{(A3)}, it is satisfied as soon as the Hessian is Lipschitz on a neighborhood of $\theta$. For instance, in the case of the linear regression, $L_{\delta} = 0$. Finally, Assumption \textbf{(H3)} is satisfied if having a first rate of convergence of the estimates of $\theta$ (thanks to Theorem \ref{theo2} or Proposition \ref{prop::ordre4} for instance) leads to have a first rate of convergence of $A_{n}$, which is often verified in practice (see \cite{BGB2020} for  instance).

\begin{theo}\label{theo3}
Suppose Assumptions \textbf{(A1)} to \textbf{(A5)}, and \textbf{(H1)} to \textbf{(H3)} hold with $\gamma>1/2$, $\delta>0$ and $\beta<\gamma-1/2$. Then, 
\begin{align*}
&\mathbb{E} \left[ \left\| \theta_{n} - \theta \right\|^{2} \right]  \leq  e^{ - \frac{1}{2}c_{\gamma}n^{1-\gamma} }\left(K^{(3)}_1+K^{(3)}_{1'}\max_{0\leq k \leq n} (k+1)^{\gamma}d_k\right) \\
&\hspace{2cm} + n^{-\gamma}\left(2^{3 + \gamma} c_{\gamma}\text{Tr} \left(H^{-1}\Sigma H^{-1} \right)+\frac{K_2^{(3)}}{n^{\gamma}}+ K_{2'}^{(3)}v_{A,n/2}\right)+ d_{\lfloor n/2\rfloor}.
\end{align*}
where $K_{i}^{(3)}, \,i=1, 1',2,2'$ are defined in \eqref{eq:theo3_constant_1}, \eqref{eq:theo3_constant_2} and \eqref{eq:theo3_constant_3}, and $d_k$ only depending on $M_k$ and $v_{A,k}$ is given in \eqref{eq:theo3_constant_2}. 
\end{theo}
Remark from \eqref{eq:theo3_constant_2} that $d_{k}\leq C(v_{A,k}+M_k)$ for some constant $C>0$. The latter results can be further simplified if we also assume a sufficiently large exponent $\delta$ in \textbf{(H1a)}.
\begin{cor}\label{cor:simplified_theo3}
Suppose Assumptions \textbf{(A1)} to \textbf{(A4)}, and \textbf{(H1)} to \textbf{(H3)} hold with $\gamma>1/2$, $\delta>\frac{2\gamma p}{p-2}$ and $\beta<\gamma-1/2$. Then, 
\begin{align*}
\mathbb{E} \left[ \left\| \theta_{n} - \theta \right\|^{2} \right]  \leq& n^{-\gamma}\left(2^{3 + \gamma} c_{\gamma}\text{Tr} \left(H^{-1}\Sigma H^{-1} \right)+\frac{K_2^{(3')}}{n^{\gamma}}+ K_{2'}^{(3')}v_{A,n/2}+ K_{2''}^{(3')}\sqrt{v_{A,n/2}}\right)\\
&\hspace{9cm} +K^{(3')}_{1}e^{ - \frac{1}{2}c_{\gamma}n^{1-\gamma} },
\end{align*}
with $K_{i}^{(3')},\, i=1...2''$ given in  \eqref{eq:cor_constant_expo} and  \eqref{eq:cor_constant_main}.
\end{cor}
Then, if $v_{A,n}$ converges to $0$, we obtain the usual rate of convergence $\frac{1}{n^{\gamma}}$. 

\subsubsection{Convergence results for adaptive gradient (Adagrad)}
Recall that the Adagrad algorithm amounts to specify $d$ initial parameters $a_1,\ldots,a_d\in\mathbb{R_{+}}$ choose $\overline{A_{n}}$ diagonal with 
\begin{equation}\label{def::an::ada}
(\overline{A_{n}})_{kk'}=\delta_{kk'}\frac{1}{\sqrt{\frac{1}{n+1}\left(a_{k}+\sum_{i=0}^{n-1}n\left(\nabla_hg(X_{i+1},\theta_i)_k\right)^2\right)}}.
\end{equation}
The original Adagrad algorithm would then amount to take $\gamma=1/2$. To guarantee non-degeneracy of the matrices $(\overline{A_{n}})_{n\geq 0}$, we assume some minimal fluctuation of the gradient at the minimizer $\theta$.
\begin{itemize}
\item[\textbf{(A6)}] There is $\alpha>0$ such that for all $1\leq i\leq d$,
\begin{equation}
\mathbb{E}\left[\left(\nabla_hg(X,\theta)\right)_i^2\right]>\alpha  .
\end{equation}
\item[\textbf{(A6')}] There is $\alpha>0$ such that for all $h\in \mathbb{R}^d$ and $1\leq i\leq d$, 
\begin{equation}
\mathbb{E}\left[\left(\nabla_hg(X,h)\right)_i^2\right]>\alpha  .
\end{equation}
\end{itemize}
Remark that \textbf{(A6')} is much stronger as \textbf{(A6)}. However, the former is often satisfied, as it is the case for the linear regression with noise. Anyway, one can consider the following transformation of $\overline{A_{n}}$:
\begin{equation}\label{eq:modification_An_adagrad}
\left(A_n\right)_{kk'}=\left\lbrace\begin{aligned}&\min \left\lbrace c_{\beta} n^{\beta}, (\overline{A_n})_{kk'}\right\rbrace, \quad \text{ if } \gamma>1/2\\
&\max \left\lbrace \min \left\lbrace c_{\beta} n^{\beta}, (\overline{A_n})_{kk'} \right\rbrace,\lambda_0'n^{-\lambda'} \right\rbrace,\quad  \text{ if }  \gamma\leq 1/2
\end{aligned}\right.
\end{equation}
where $\beta_n=c_\beta n^{\beta}$ with $\beta<\min\lbrace\gamma/2,1/4\rbrace$ ($\lambda',\lambda'_0$ and $c_{\beta}>0$ are chosen arbitrarily).

\begin{theo}
\label{theo:adagrad}
Suppose Assumptions \textbf{(A1)} to \textbf{(A4)} and \textbf{(A6)} hold with $\beta< \min\left\lbrace \frac{ (1-\gamma)\gamma(\gamma-2\beta)p}{4(2-\gamma)}  , 1/4 \right\rbrace$. Then, 
\begin{align*}
\mathbb{E}\left[ \Vert \theta_n-\theta\Vert^2 \right]  & \leq \tilde{K}_{1}^{(4)}\exp \left( -  c_{\gamma}  \mu \tilde{\lambda}_{0} n^{1-\gamma} \left(1-\tilde{\varepsilon}(n)\right)\right)+ \tilde{K}^{(4)}_2\log (n+1)^{\frac{p-1}{p}}n^{-\frac{(p-1)}{p}\min \left\lbrace \frac{2(1-\gamma)\gamma(\gamma-2\beta)p}{2-\gamma} ,  1 \right\rbrace } \\
& +\tilde{K}^{(4)}_3n^{ - \gamma},
\end{align*}
with $\tilde{\varepsilon}(n)$ given in \eqref{eq:espilon(n):adagrad}, $v_n=v_0\log(n+1)$, with $v_0$, $C_{S}^{4}$ and $\tilde{\lambda}_{0}$ given in \eqref{eq:Adagrad_v0}, \eqref{eq:CS4_adagrad} and \eqref{eq:Adagrad_def_lambda0} with $p' = \frac{2(1-\gamma)}{2-\gamma}p$. In addition, $K^{(4)}_1$, $K^{(4)}_{2}$ and $K^{(4)}_3$   given in \eqref{eq:Adagrad_constant}.
If \textbf{(A6')} is satisfied, same conclusion holds for $\beta<1/4$ with $C_{S}$ given in \eqref{eq:CS4_adagrad_bis} taking $p' = \frac{2(1-\gamma)}{2-\gamma}p$. 
\end{theo}
In the special case where $\gamma=1/2$, which corresponds to the usual Adagrad algorithm, we get \begin{align*}
\mathbb{E}\left[ \Vert \theta_n-\theta\Vert^2 \right]  & \leq K_{1}^{(4)}\exp \left( -  c_{\gamma}  \mu \lambda_{0} \sqrt{n} \left(1-\varepsilon(n)\right)\right) \\
& + \frac{1}{\sqrt{n}}\left(K^{(4)}_2 \log(n+1)n^{1/2-\frac{(1-4\beta)(p-1)}{6}}+K^{(4)}_3\right),
\end{align*}
and we so achieve the usual rate of convergence $\frac{1}{\sqrt{n}}$ as soon as $1/2-\frac{(1-4\beta)(p-1)}{6} < 0$, i.e as soon as $p> 4 \frac{1 -\beta}{1-4\beta}$.
%\begin{cor}
%\label{cor:adagrad}
%Suppose Assumptions \textbf{(A1)} to \textbf{(A4)} and \textbf{(A6)} hold with $\beta<1/4$ and $p>9\frac{1+4\beta}{1-4\beta}$. Then, 
%\begin{align*}
%\mathbb{E}\left[ \Vert \theta_n-\theta\Vert^2 \right]  & \leq K_{1}^{(4)}\exp \left( -  c_{\gamma}  \mu \lambda_{0} \sqrt{n} \left(1-\varepsilon(n)\right)\right)+ \frac{1}{\sqrt{n}}\left(K^{(4)}_2v_{\lfloor n/2\rfloor}n^{1/2-\frac{p-1}{p}}+K^{(4)}_3\right),
%\end{align*}
%with $\varepsilon(n)$ given in \eqref{eq:espilon(n)} and $K^{(4)}_1$, $K^{(4)}_{2}$ and $K^{(4)}_3$ given in Theorem \ref{theo:adagrad} with $\gamma=1/2$.

%If \textbf{(A6')} is satisfied, same conclusion holds for $\beta<1%/4$ with $C_{S}$ given in \eqref{eq:CS4_adagrad_bis}. 
%\end{cor}
\section{Application to linear model}\label{sec::lm}

Let us now consider the linear model $Y = X^{T} \theta + \epsilon $ where $X \in \mathbb{R}^{d}$ and $\epsilon$ is a centered random real variable independent from $X$. Let us suppose from now that $\mathbb{E}\left[ XX^{T} \right]$ is positive. Then, $\theta$ is the unique minimizer of the functional $G : \mathbb{R}^{d} \longrightarrow \mathbb{R}$ defined for all $h \in \mathbb{R}^{d}$ by
\[
G(h) = \frac{1}{2}\mathbb{E} \left[ \left( Y - X^{T}h \right)^{2} \right] .
\] 
If $X$ admits a second order moment, the function $G$ is twice continuously differentiable with $\nabla G(h) = - \mathbb{E}\left[ \left( Y- X^{T}h \right) X \right]$ and $\nabla^{2}G(h) =\mathbb{E}\left[ XX^{T} \right]$. 

\subsection{Stochastic Newton algorithm}
The Stochastic Newton algorithm is defined recursively for all $n \geq 0$ by \citep{BGB2020}
\[
\theta_{n+1} = \theta_{n} + \gamma_{n+1} \overline{S}_{n}^{-1} \left( Y_{n+1} - X_{n+1}^{T}\theta_{n} \right) X_{n+1}
\]
with $\tilde{S}_{n} = \frac{1}{n+1}\left( S_{0} + \sum_{i=1}^{n} X_{i}X_{i}^{T} \right)$, with $S_{0}$ positive, and
\[
\overline{S}_{n}^{-1} = \frac{\min \left( \left\| \tilde{S}_{n}^{-1} \right\|_{op} , \beta_{n+1} \right)}{\left\| \tilde{S}_{n}^{-1} \right\|_{op}} \tilde{S}_{n}^{-1}
\]
with $\beta_{n} = c_{\beta}n^{-\beta}$. Remark that $\tilde{S}_{n+1}^{-1}$ can be easily updated with only $O \left( d^{2} \right)$ operations using Sherman Morrison (or Ricatti's) formula. More precisely, considering $S_{n} = (n+1)\tilde{S}_{n}$, one has
\[
S_{n+1}^{-1} = S_{n}^{-1} - \left( 1+ X_{n+1}^{T}S_{n}^{-1}X_{n} \right)^{-1}S_{n}^{-1}X_{n+1}X_{n+1}^{T}S_{n}^{-1}. 
\]
Then, one can easily update $\tilde{S}_{n}$ and $\overline{S}_{n}$. 
We can now rewrite Theorem \ref{theo3} as follows:

%\textcolor{red}{A supprimer par la suite: on a $L_{\nabla g} = \frac{\mathbb{E}\left[ \| X \|^{4} \right]}{\lambda_{\min}^{2}}$, $C_{A} = 4c_{\gamma} \frac{\mathbb{E}\left[ \| X \|^{4} \right]}{ \lambda_{\min}^{2}} \geq 4c_{\gamma}$, $C_{1} = \sigma_{(2)},C_{1}' = \sigma_{(4)},C_{2}=C_{(2)},C_{2}' = C_{(4)},L_{\nabla G} = \lambda_{\max}, \mu = \lambda_{\min}, \lambda_{0} = \frac{1}{2 \mathbb{E}\left[ \left\| X \right\|^{2} \right]}$ j'ai utilisé le fait que $L_{\nabla G}/\frac{\mathbb{E}\left[ \| X \|^{4} \right]}{ \lambda_{\min}^{2}} \leq \lambda_{\min} $, $C_{2} / \frac{\mathbb{E}\left[ \| X \|^{4} \right]}{ \lambda_{\min}^{2}} \leq \lambda_{\min}^{2}/2 $}

\begin{theo}\label{theo::LM}
Suppose that there is $p> 4$ such that $X,\epsilon$ admits a moment of orders $2p$ and $p$.  Suppose also that there is a positive constant $L_{MK}$ such that for any $h \in \mathbb{S}^{d-1}$, $\sqrt{\mathbb{E}\left[ hXX^{T}h \right]} \leq L_{MK} \mathbb{E}\left[ \left| X^{T}h \right| \right]$.    Then, for $\gamma>1/2$, we have
\begin{align*}
&\mathbb{E} \left[ \left\| \theta_{n} - \theta \right\|^{2} \right]  \leq  e^{ - \frac{1}{2}c_{\gamma}n^{1-\gamma} }\left(K^{(3)}_{1,\text{lin}}+K^{(3)}_{1',\text{lin}}\max_{0\leq k \leq n} d_k(k+1)^{\gamma}\right) \\
& + n^{-\gamma}\left(2^{3 + \gamma} c_{\gamma} \mathbb{E} \left[ \epsilon^{2} \right]\text{Tr} \left(H^{-1} \right)+\frac{K_{2,\text{lin}}^{(3)}}{n^{\gamma}}+ K_{2',\text{lin}}^{(3)}v_{H,n/2}\right)+ d_{\lfloor n/2\rfloor}.
\end{align*}
where $K_{2,\text{lin}},K_{2',\text{lin}}^{(3)},K^{(3)}_{1,\text{lin}},K^{(3)}_{1',\text{lin}},d_{n} $  are given by \eqref{def::const::lm} while $v_{H,n}$ is defined in \eqref{def::vhn}.
\end{theo}
Observe that $d_{n} = O \left( \frac{1	}{n^{ \max \left\lbrace \frac{p-2}{2}, 2\gamma \right\rbrace}}  \right)$ and $v_{H,n} = O\left(n^{-1} \right)$, and since $p>4$, these terms are both negligible.
%Remark that for the main term, one has a rate of order $\text{Tr}\left( H^{-1} \right) \leq \sqrt{d}\lambda_{\min}^{-1}$ while for the stochastic gradient, one has an upper bound, for the main term, of the form $ \mathbb{E}\left[ \left\| X \right\|^{2} \right] \lambda_{\min}^{-1} \leq \lambda_{\max} \lambda_{\min}^{-1}$. This seems to confirm the fact that Newton algorithm should be preferred if the largest and the smallest eigenvalues of the Hessian are at different scales. Nevertheless and without surprise, it seems to be much more sensitive to the dimension.

\subsection{Adagrad algorithm}
For linear model, Adagrad algorithm is defined for all $n\geq 0$ by 
$$\theta_{n+1}=\theta_n+\gamma_{n+1}\bar{D}_n^{-1}\left( Y_{n+1} - X_{n+1}^{T}\theta_{n} \right) X_{n+1},$$
with $\bar{D}_n$ diagonal with, for $\gamma \leq 1/2$, 
$$(\bar{D}_n)_{kk}=\min\left\lbrace\max\left\lbrace\frac{n^{-\beta}}{c_{\beta}},\sqrt{\frac{1}{n+1}\left(a_k+\sum_{i=0}^{n-1} \left(\left( Y_{i+1} - X_{i+1}^{T}\theta_{i} \right) (X_{i+1})_{k}\right)^2\right)}\right\rbrace,\frac{n^{\lambda ' }}{\lambda_0 '}\right\rbrace .$$
where $0<\beta<(\gamma-\lambda ')/2$ for some $a_k>0$ and if $\gamma>1/2$,
$$(\bar{D}_n)_{kk}=\max\left\lbrace\frac{n^{-\beta}}{c_{\beta}},\sqrt{\frac{1}{n+1}\left(a_k+\sum_{i=0}^{n-1}\left(\left( Y_{i+1} - X_{i+1}^{T}\theta_{i} \right) (X_{i+1})_{k}\right)^2\right)}\right\rbrace,$$
for some $0<\beta<\gamma-1/2$. The usual Adagrad algorithm is done with $\gamma=1/2$, which yields for us 
$$(\theta_{n+1})_{k}=(\theta_n)_{k}+\frac{\left( Y_{n+1} - X_{n+1}^{T}\theta_{n} \right) (X_{n+1})_{k}}{\min\left\lbrace\max\left\lbrace\frac{n^{-\beta+1/2}}{c_{\beta}},\sqrt{a_k+\sum_{i=0}^{n-1}\left(\left( Y_{i+1} - X_{i+1}^{T}\theta_{i} \right) (X_{i+1})_{k}\right)^2}\right\rbrace,\frac{n^{\lambda '+1/2}}{\lambda_0 '}\right\rbrace},$$
and almost surely there exists $n_0\geq n$ such that for $n\geq n_0$,
$$(\theta_{n+1})_{k}=(\theta_n)_{k}+\frac{\left( Y_{n+1} - X_{n+1}^{T}\theta_{n} \right) (X_{n+1})_{k}}{\sqrt{a_k+\sum_{i=0}^{n-1}\left(\left( Y_{i+1} - X_{i+1}^{T}\theta_{i} \right) (X_{i+1})_{k}\right)^2}},$$
which is the usual Adagrad algorithm. We can then rewrite Theorem \ref{theo:adagrad} as follows (we simply give it for $\gamma=1/2$, the reader can easily adapt it to the case $\gamma>1/2$).
\begin{theo}\label{theo::LM_ada}
Suppose that there is $p> 2$ such that $X,\epsilon$ admits a moment of orders $2p$. Then, for $\gamma \leq 1/2$, we have
\begin{align*}
\mathbb{E}\left[ \Vert \theta_n-\theta\Vert^2 \right]  & \leq  {K}_{1,lin}^{ada}\exp \left( -  c_{\gamma}  \lambda_{\min} \lambda_{0,lin}^{ada} n^{1-\gamma} \left(1- {\varepsilon}_{n,lin}^{ada}\right)\right) \\
& +  {K}^{ada}_{2,lin}\log (n+1)^{\frac{p-1}{p}}n^{-\frac{(p-1)}{p}\min\left\lbrace  \frac{2(1-\gamma)\gamma(\gamma-2\beta)p}{2-\gamma} , 1 \right\rbrace}  + {K}^{ada}_{3,lin}n^{ - \gamma},
\end{align*}
where $\varepsilon_{n,lin}^{ada}=o(1)$ is given in \eqref{def::epsilon::lm::ada} and $K^{ada}_{1,\text{lin}},\,K^{ada}_{2,\text{lin}}\,K^{ada}_{3,\text{lin}}$  are given by \eqref{def::constant_1::lm::ada}, \eqref{def::constant_2::lm::ada} and \eqref{def::constant_3::lm::ada}.
\end{theo}
Remark that similar statements hold for $\gamma>1/2$. Observe that in the case where $\gamma = 1/2$, the $\frac{1}{\sqrt{n}}$ rate of convergence is achieved as soon as $(p-1)(1-4\beta) /3 \geq 1/2$, i.e as soone as $p > \frac{5-4\beta}{2(1-4\beta)} $. 
\section{Application to generalized linear models}\label{sec::glm}

The framework of the linear regression can be easily generalized to the more general setting of finite dimensional linear models. Let $\ell:\mathcal{Y}\times \mathcal{Y}\rightarrow \mathbb{R}$ a cost function for some domain $\mathcal{Y}\subset \mathbb{R}$. The general learning problem is to solve the minimization problem
$$\argmin_{f\in \mathcal{F}} \mathbb{E} \left[\ell(Y,f(X))\right],$$
with $(X,Y)\sim \mathbb{P}$ and $\mathcal{F}$ is a given class of measurable function from $\mathcal{X}$ to $\mathcal{Y}$, where $\mathcal{X}$ is a measurable space. In the case of finite dimensional linear models, $\mathcal{Y}=\mathbb{R}$ and $\mathcal{F}=\left\{h^T\Phi(\cdot), h\in\mathbb{R}^m\right\}$, with $\Phi:\mathcal{X}\rightarrow \mathbb{R}^m$ a \textit{known} design function (remark that the setting can be easily generalized to the case $\mathcal{Y}=\mathbb{R}^p$ and $\Phi:\mathcal{X}\rightarrow \mathbb{R}^m$ and $h\in M_{m,p}(\mathbb{R})$). Then, assuming that $\ell$ is convex and adding a regularization term on $\theta$, the minimization problem turns into the framework of this paper with 
$$G(h)=\mathbb{E}\left[ g \left(\tilde{Z},h\right)\right],$$
with $\tilde{Z}=\left(Y,\Phi(X)\right):=(Y,\tilde{X})$ and for all $h \in M_{m,p}(\mathbb{R})$,  $g(\tilde{Z},h)=\ell(Y,h^T\tilde{X})$. In what follows, let us suppose from now that the cost function $l$ is twice differentiable for the second variable and that there is a positive constant $L_{\nabla l}$ such that for all $h$
\begin{equation}\label{upperbound_glm}
\left| \nabla_{h}^{2} \ell \left( Y , h^{T}\tilde{X} \right) \right|  \leq L_{\nabla l} .
\end{equation}
where $\nabla_{h}^{2}\ell(.,.)$ is the second order derivative with respect to the second variable.  Remark that such a bound is generally assumed if we require that for all $h$, $\Vert \nabla^2G (h) \Vert_{op} \leq L_{\nabla G}<\infty$ 
This is for example satisfied when $\ell(y,y')=f(y-y')$ with $\sup_{y} \left|f''(y)\right|<+\infty$.
For example, in the simplest case of the logistic regression, we consider a couple of random variables  $\left( X ,Y \right)$ lying in $ \mathbb{R}^{d} \times\left\lbrace -1,1 \right\rbrace$, $\Phi=I_{d}$ and $\ell(y,y')=\log(1+\exp(-yy'))$, and we indeed have for all $h$ and $Y\in \{-1,1\}$
$$\nabla_{h}^{2}\ell(Y,h^{T}X)=\frac{1}{1+\exp(h^{T}X)}\cdot \frac{1}{1+\exp(-h^{T}X)}\leq 1. $$ 
There are then two main cases to deal with the convexity of the minimization problem : either assume strong convexity or use a regularization. The first consists in assuming that the functional $h \longmapsto  \mathbb{E}\left[ \ell \left( Y , h^{T}X \right) \right]$ is strongly convex, which is in particular verified when there exists $\alpha>0$ such that
\begin{equation}\label{lowerbound_glm}
\inf_{y'\in\mathbb{R}}\nabla_{h}^{2}\ell(y,y')>\alpha.
\end{equation}
and $\mathbb{E}\left[ XX^{T} \right]$ is positive. This case is called the elliptic case in the sequel and the results are very analogous to the ones for the linear regression and are thus not repeated. We will then focus on the regularized case. Without uniform lower bound on $\nabla_{h}^{2} \ell(y,y')$, one needs a regularization term, yielding the following regularized minimization problem
\begin{equation}\label{eq:regularized_equation}
\argmin_{\theta \in \mathbb{R}^{m}} \mathbb{E} \left[\ell(Y,\langle \theta, \theta^{T}X\rangle)\right]+\frac{\sigma}{2} \Vert \theta\Vert^2
\end{equation}
for some $\sigma> 0$. In what follows,  we suppose that the minimizer exists and we denote it by $\theta_{\sigma}$.

\subsection{Stochastic Newton algorithm} 
The Stochastic Newton algorithm is defined recursively for all $n \geq 0$ by
\[
\theta_{n+1} = \theta_{n} - \gamma_{n+1} \overline{S}_{n}^{-1} \left( \nabla_{h}l \left( Y_{n+1} , \theta_{n}^{T}X_{n+1} \right) X_{n+1} + \sigma \theta_{n} \right),
\] 
where, using the tricked introduced in \cite{bercu2021stochastic} and developed in \cite{WEI},  $\overline{S}_{n}$ is the natural recursive estimate of the Hessian given by
\begin{equation}\label{def::snbar::glm}
\overline{S}_{n}=\frac{1}{n+1}\sum_{i=0}^{n-1}\nabla_{h}^{2}\ell(Y_{i+1},\langle \theta_i,X_{i+1}\rangle)X_{i+1}X_{i+1}^T+\frac{\sigma d}{n+1}\sum_{i=1}^ne_{i[d] +1}e_{i[d]+1}^T,
\end{equation}
with $i[d]$ denoting $i$ modulo $d$.  Remark that one can easily update the inverse using the Riccati's formula used twice, i.e considering $S_{n}=(n+1) \overline{S}_{n}$ and 
\begin{align*}
{S}_{n+ \frac{1}{2}}^{-1} &  = {S}_{n}^{-1} - \nabla_{h}^{2}\ell(Y_{n+1},\langle \theta_n,X_{n+1}\rangle)\left( 1+ \nabla_{h}^{2}\ell(Y_{n+1},\langle \theta_n,X_{n+1}\rangle) X_{n+1}^{T}{S}_{n}^{-1} X_{n+1} \right)^{-1} {S}_{n}^{-1}X_{n+1}X_{n+1}^{T}{S}_{n}^{-1} \\
{S}_{n+1} & = {S}_{n+\frac{1}{2}}^{-1} - \sigma d\left( 1+ \sigma d e_{(n+1)[d]+1}^{T}{S}_{n + \frac{1}{2}}^{-1} e_{(n+1)[d]+1} \right)^{-1} {S}_{n + \frac{1}{2}}^{-1}e_{(n+1)[d]+1}e_{(n+1)[d]+1}^{T}{S}_{n + \frac{1}{2}}^{-1},
\end{align*}
one has $\overline{S}_{n+1}^{-1} = (n+2)S_{n+1}^{-1}$.
In what follows, let us suppose that the following assumptions hold:
\begin{enumerate}
\item[\textbf{(GLM1)}] There is $L_{\nabla^{2}L}\geq 0$ such that the function $h \longmapsto \mathbb{E}\left[ \nabla_{h}^{2}\ell \left( Y , h^{T}X \right) XX^{T} \right]$ is $L_{\nabla^{2}L}$-Lispchitz with respect to the spectral norm.
\item[\textbf{(GLM2)}] There is $p> 2$ such that $X$ admits a moment of order $2p$ and such that there is a positive constant $L_{\sigma}$ satisfying for all $0 \leq a \leq 2p$
\[
 \mathbb{E}\left[ \left\| \nabla_{h}\ell \left( Y , X^{T}\theta_{\sigma} \right)X + \sigma \theta_{\sigma} \right\|^{a}  \right] \leq L_{\sigma}^{a} .
\]
\end{enumerate} 
Remark that Assumption \textbf{(GLM1)} is verified when for all $y$, $\nabla_{h}^{2}\ell(y,.)$ is Lipschitz and $X$ admits a third order moment, which can be easily verified for the logistic regression for instance. Assumption \textbf{(GLM2)} is verified when the random variable $\nabla_{h}\ell \left( Y , X^{T}\theta_{\sigma} \right)X$ admits a moment of order $2p$.

\begin{theo}\label{theo::glm}
Suppose Assumptions \textbf{(GLM1)} and \textbf{(GLM2)} hold. Then,  
\begin{align*}
&\mathbb{E} \left[ \left\| \theta_{n} - \theta_{\sigma} \right\|^{2} \right]  \leq  e^{ - \frac{1}{2}c_{\gamma}n^{1-\gamma} }\left(K^{(3)}_{1,\text{GLM}}+K^{(3)}_{1',\text{GLM}}\max_{0\leq k \leq n} (k+1)^{\gamma}d_{k,\text{GLM}}\right) \\
&\hspace{2cm} + n^{-\gamma}\left(2^{3 + \gamma} c_{\gamma}\text{Tr} \left(H_{\sigma}^{-1}\Sigma_{\sigma} H_{\sigma}^{-1} \right)+\frac{K_{2,\text{GLM}}^{(3)}}{n^{\gamma}}+ K_{2',\text{GLM}}^{(3)}v_{l,n/2}\right)+ d_{\lfloor n/2\rfloor,\text{GLM}},
\end{align*}
where  $H_{\sigma} = \mathbb{E}\left[ \nabla_{h}^{2}\ell \left( Y ,X^{T}\theta_{\sigma} \right)XX^{T} \right] + \sigma I_{d}$, $\Sigma_{\sigma} = \mathbb{E}\left[ \left( \nabla_{h} \ell \left( Y , X^{T}\theta_{\sigma} \right) X + \sigma \theta_{\sigma} \right)\left( \nabla_{h} \ell \left( Y , X^{T}\theta_{\sigma} \right) X + \sigma \theta_{\sigma} \right)^{T} \right]$, 
$K^{(3)}_{1,\text{GLM}},K^{(3)}_{1',\text{GLM}},K_{2,\text{GLM}}^{(3)},K_{2',\text{GLM}}^{(3)}, d_{n,\text{GLM}}$ are defined in equations \eqref{eq:theo3_constant_1_GLM}, \eqref{eq:theo3_constant_2_GLM} and \eqref{eq:theo3_constant_3_GLM}, and $v_{l,n}$ is defined in Proposition \ref{lem::GLM::H2}.
\end{theo}
\subsection{Adagrad algorithm}
 For generalized linear model, Adagrad algorithm is defined for all $n\geq 0$ by 
$$\theta_{n+1}=\theta_n -\gamma_{n+1}\bar{D}_n^{-1}\nabla_{h}l \left( Y_{n+1} , \theta_{n}^{T}X_{n+1} \right) X_{n+1},$$
where $\bar{D}_n$ is diagonal and for  $\gamma>1/2$,
$$(\bar{D}_n)_{kk}=\max\left\lbrace\frac{n^{-\beta}}{c_{\beta}},\sqrt{\frac{1}{n+1}\left(a_k+\sum_{i=0}^{n-1}\left(\nabla_{h}l \left( Y_{i+1} , \theta_{i}^{T}X_{i+1} \right) (X_{i+1})_{k} + \sigma (\theta_{i})_k \right)^2\right)}\right\rbrace,$$
for some $0<\beta<\gamma-1/2$, and for $\gamma \leq 1/2$,
$$(\bar{D}_n)_{kk}=\min\left\lbrace\max\left\lbrace\frac{n^{-\beta}}{c_{\beta}},\sqrt{\frac{1}{n+1}\left(a_k+\sum_{i=0}^{n-1}\left(\nabla_{h}l \left( Y_{i+1} , \theta_{i}^{T}X_{i+1} \right) (X_{i+1})_{k} + \sigma (\theta_{i})_k \right)^2\right)}\right\rbrace,\frac{n^{\lambda '}}{\lambda_0 '}\right\rbrace.$$
where $0<\beta<(\gamma-\lambda')/2$ and $a_k>0$. The usual Adagrad algorithm is done with $\gamma=1/2$, which yields for us 
$$(\theta_{n+1})_{k}=(\theta_n)_{k}+\frac{\nabla_{h}l \left( Y_{n+1} , \theta_{n}^{T}X_{n+1} \right) (X_{n+1})_{k} + \sigma (\theta_{n})_k }{\min\left\lbrace\max\left\lbrace\frac{n^{-\beta+1/2}}{c_{\beta}},\sqrt{a_k+\sum_{i=0}^{n-1}\left(\nabla_{h}l \left( Y_{i+1} , \theta_{i}^{T}X_{i+1} \right) (X_{i+1})_{k} + \sigma (\theta_{i})_k \right)^2}\right\rbrace,\frac{n^{\lambda '+1/2}}{\lambda_0 '}\right\rbrace}.$$

Like the linear regression, the general linear model needs  minimal randomness to ensure the expected rate of convergence of Adagrad. Indeed, in the extreme case of a deterministic sequence $(X_n,Y_n)_{n\geq 0}$, Adagrad algorithm may diverge in the unfortunate situation where $\nabla_h\ell \left( Y_{i+1} , \theta_{i}^{T}X_{i+1} \right) (X_{i+1})_{k}$ vanishes or remains very small. Such behavior can be averted by requiring at the minimizer $\theta_{\sigma}$ a minimal variance for $\nabla_h\ell \left( Y, \theta_{\sigma}^{T}X \right) (X)_{k}$ for all $1\leq k\leq d$.  
\begin{enumerate} 
\item[\textbf{(GLM3)}]There is a positive constant $\alpha_{\sigma}>0$ such that for all $1\leq k\leq d$
\[
 Var\left[ \nabla_{h}l \left( Y , X^{T}\theta_{\sigma} \right)(X)_k \right] > \alpha_{\sigma} .
\]
\end{enumerate}
Remark that 
\begin{equation}\label{eq:adagrad_glm_Var=square}
Var\left[ \nabla_{h}l \left( Y , X^{T}\theta_{\sigma} \right)(X)_k \right]= \mathbb{E}\left[ \left\vert \nabla_{h}l \left( Y , X^{T}\theta_{\sigma} \right)(X)_k + \sigma (\theta_{\sigma})_k \right\vert^{2}  \right] ,
\end{equation}
so that \textbf{(GLM3)} can be seen as a mirror assumption to \textbf{(GLM2)}. We should stress that the existence of such $\alpha_{\sigma}$ is almost automatic when a minimal randomness between $X$ and $Y$ is assumed. Indeed, having $\nabla_hl \left( Y, \theta_{\sigma}^{T}X \right) X_{k}$ deterministic would imply an analytic relation between $Y$ and $X$. The main computational issue is to estimate a concrete value of $\alpha_{\sigma}$. An example dealing with the logistic regression is given in Section \ref{sec::alpha::log}.

When \textbf{(GLM3)} is assumed, one can show  using Theorem \ref{theo::GLM_ada} that there exists almost surely $n_0\geq n$ such that for $n\geq n_0$,
$$(\theta_{n+1})_{k}=(\theta_n)_{k}+\frac{\nabla_{h}l \left( Y_{n+1} , \theta_{n}^{T}X_{n+1} \right) (X_{n+1})_{k} + \sigma (\theta_{n})_k }{\sqrt{a_k+\sum_{i=0}^{n-1}\left(\nabla_{h}l \left( Y_{i+1} , \theta_{i}^{T}X_{i+1} \right) (X_{i+1})_{k} + \sigma (\theta_{i})_k \right)^2}},$$
so that we recover the usual Adagrad algorithm for large $n$. We can then rewrite Theorem \ref{theo:adagrad} for $\gamma \leq 1/2$ as follows (remark that similar statements hold for $\gamma>1/2$).

\begin{theo}\label{theo::GLM_ada}
Suppose Assumptions \textbf{(GLM1)}, \textbf{(GLM2)} and  \textbf{(GLM3)} hold. Then, for $\gamma=1/2$, we have
\begin{align*}
&\mathbb{E} \left[ \left\| \theta_{n} - \theta_{\sigma} \right\|^{2} \right]  \leq K^{ada}_{1,\text{GLM}}\exp \left( -  c_{\gamma} \sigma \tilde{\lambda}_{0,\text{GLM}} n^{\frac{p}{2(1+p)}} (1-\varepsilon(n)\right) \\
&  {K}^{ada}_{2,\text{GLM}}\log (n+1)^{\frac{p-1}{p}}n^{-\frac{(p-1)}{p}\min\left\lbrace  \frac{2(1-\gamma)\gamma(\gamma-2\beta)p}{2-\gamma} , 1 \right\rbrace} + {K}^{ada}_{3,\text{GLM}}n^{ - \gamma},
\end{align*}
where $\varepsilon(n)=o(1)$, $K^{ada}_{1,\text{GLM}}$, $K^{ada}_{2,\text{GLM}}$ and $K^{ada}_{3,\text{GLM}}$   have explicit formulas depending on the parameters of the model.
\end{theo}
We do not specify the exact value of the constants here, since they can easily be obtained along the lines of previous results.

\section{Proofs}\label{sec::proofs}
Throughout our proofs, to alleviate notations, we will denote by the same way $\| . \|$ the euclidean norm of $\mathbb{R}^{d}$ and the spectral norm for square matrices.
In addition, we will regularly use the following technical result from \cite[Proposition A.5]{ GBWW2021}. 

\begin{prop} \label{prop:appendix:delta_recursive_upper_bound_nt}
Let $(\gamma_{t})_{t \geq 1}$, $(\eta_{t})_{t \geq 1}$, and $(\nu_{t})_{t \geq 1}$ be some positive and decreasing sequences and let $(\delta_{t})_{t \geq 0}$,  satisfying the following:
\begin{itemize}
\item The sequence $\delta_{t}$ follows the recursive relation:
\begin{align} \label{eq:prop:appendix:delta_recursive_single}
\delta_{t} \leq \left( 1 - 2 \omega \gamma_{t} + \eta_{t} \gamma_{t} \right) \delta_{t-1} + \nu_{t} \gamma_{t},
\end{align}
with $\delta_{0} \geq 0$ and $ \omega > 0$.
\item Let $\gamma_{t}$ and $\eta_{t}$ converge to $0$.
\item Let $t_{0} = \inf \left\{ t\ge 1\, : \eta_{t} \leq \omega \right\}$, and let us suppose that for all $t \geq t_{0} +1$, one has $\omega \gamma_{t} \leq 1$.
\end{itemize}
Then, for all $t \in \mathbb{N}$, we have the upper bound:
\begin{align*} 
\delta_{t}\leq  \exp \left( - \omega \sum_{j=t/2}^{t} \gamma_{j} \right) 
\exp \left( 2 \sum_{i=1}^{t} \eta_{i} \gamma_{i} \right) 
\left( \delta_{0} + 2 \max_{1 \leq i \leq t} \frac{\nu_{i}}{\eta_{i}} \right)
+ \frac{1}{\omega} \max_{t/2 \leq i \leq t} \nu_{i}.
\end{align*}
with the convention that $\sum_{t_{0}}^{t/2}=0$ if $t/2 < t_{0}$.
\end{prop}
Moreover, we denote by $C_{1},C_{1}',C_{2},C_{2}'$ constants such that for all $h \in \mathbb{R}^d$,
\begin{equation}\label{eq:def_C1}
\mathbb{E}\left[ \left\| \nabla_{h} g \left( X , h \right) \right\|^{2} \right] \leq C_{1} + C_{2} \left\| h - \theta \right\|^{2}  , \quad \quad \mathbb{E}\left[ \left\| \nabla_{h} g \left( X , h \right) \right\|^{4} \right] \leq C_{1}' + C_{2}'  \left\| h - \theta \right\|^{4}.
\end{equation}
\subsection{Proof of Theorem \ref{theo1}}

Remark that thanks to a Taylor's expansion of the gradient, denoting $V_{n} = G \left( \theta_{n} \right) - G(\theta )$ and $g_{n+1}'=  \nabla_{h} g \left( X_{n+1} , \theta_{n} \right)$,
\begin{align}
\notag V_{n+1} & \leq V_{n} - \gamma_{n+1}\nabla G \left( \theta_{n} \right)^{T}A_{n} g_{n+1}' + \frac{L_{\nabla G}}{2}\gamma_{n+1}^{2}\left\| A_{n} \right\|^{2} \left\| g_{n+1}' \right\|^{2} \\
\label{majVn} & \leq  V_{n} - \gamma_{n+1}\nabla G \left( \theta_{n} \right)^{T}A_{n} g_{n+1}' + \frac{L_{\nabla G}}{2}\gamma_{n+1}^{2}\beta_{n+1}^{2} \left\| g_{n+1}' \right\|^{2},
\end{align}
where we used Hypothesis \textbf{(H1b)} on the last line. Then, taking the conditional expectation, thanks to assumption \textbf{(A1)}, and since $\left\| \theta_{n} - \theta \right\|^{2} \leq \frac{2}{\mu } V_{n}$,
\begin{align*}
\mathbb{E}\left[  V_{n+1} |\mathcal{F}_{n} \right] & \leq \left( 1+ \frac{C_{2}L_{\nabla G}}{\mu} \beta_{n+1}^{2}\gamma_{n+1}^{2} \right)  V_{n} - \gamma_{n+1}\nabla G \left( \theta_{n} \right)^{T}A_{n} \nabla G \left( \theta_{n} \right) + \frac{C_{1}L_{\nabla G}}{2}\gamma_{n+1}^{2}\beta_{n+1}^{2} \\
\end{align*}
Furthermore, since $G$ is $\mu$ strongly convex, it comes 
\begin{align}
\notag \nabla G \left( \theta_{n} \right)^{T} A_{n} \nabla G \left( \theta_{n} \right) & \geq \lambda_{\min} \left( A_{n} \right) \left\| \nabla G \left( \theta_{n} \right) \right\|^{2} \\
\notag & \geq 2\lambda_{n}\mu V_{n} \mathbf{1}_{\lambda_{\min} \left( A_{n} \right) \geq \lambda_{n}} \\
 & = 2\lambda_{n}\mu V_{n} -  \mathbf{1}_{\lambda_{\min} \left( A_{n} \right) < \lambda_{n}}2\lambda_{n}\mu V_{n}, 
\label{majgradsgrad} 
\end{align}
with $\lambda_n=\lambda_0(n+1)^{\lambda}$ with $0\leq \lambda<1-\gamma$. Applying Cauchy-Schwarz  yields
\begin{align*}
\mathbb{E}\left[ \nabla G \left( \theta_{n} \right)^{T}A_{n} \nabla G \left( \theta_{n} \right) \right] \geq& 2\lambda_{n} \mu \mathbb{E}\left[  V_{n} \right] - 2\lambda_{n} \mu \sqrt{\mathbb{E}[V_n^2]} \sqrt{\mathbb{P}\left[  \lambda_{\min} \left( A_{n} \right)<\lambda_n \right]}\\
\geq& 2\lambda_{n} \mu \mathbb{E}\left[  V_{n} \right] - 2\lambda_{n} \mu V \sqrt{\mathbb{P}\left[ \lambda_{\min} \left( A_{n} \right)<\lambda_n \right]}, 
\end{align*}
with $V^{2} \geq \sup_{n\geq 0}\mathbb{E}[V_n^2]$ calculated later. Then, Assumption \textbf{(H1a)} gives $\mathbb{P}\left[ \lambda_{\min} \left( A_{n} \right)<\lambda_n\right]\leq v_{n+1}(n+1)^{-\delta-q\lambda}:=\bar{v}_n$, so that
\begin{align*}
\mathbb{E}\left[  V_{n+1} \right] & \leq \left( 1 -2 \mu \lambda_{0}(n+1)^{-\lambda}\gamma_{n+1} + \frac{C_{2}L_{\nabla G}}{\mu} \beta_{n+1}^{2}\gamma_{n+1}^{2} \right) \mathbb{E}\left[  V_{n} \right] \\
&\hspace{5cm}+ 2 \lambda_{0}(n+1)^{-\lambda}\mu V \gamma_{n+1} \sqrt{\bar{v}_n}+ \frac{C_{1}L_{\nabla G}}{2}\gamma_{n+1}^{2}\beta_{n+1}^{2}.\\
\end{align*}

In order to apply Proposition \ref{prop:appendix:delta_recursive_upper_bound_nt}, let us denote 
\begin{equation}\label{eq:constant_C_m}
C_{M} = \max \left\lbrace \frac{C_{2}L_{\nabla G}c_{\beta}^{2}c_{\gamma}}{\mu}, \left( \mu \lambda_{0} \right)^{\frac{2\gamma - 2\beta}{\gamma+\lambda}}c_{\gamma}^{\frac{\gamma -2 \beta-\lambda}{\gamma+\lambda}} \right\rbrace,
\end{equation}
the last upper bound being added so that the terms of \eqref{eq:bound_recursion_first_result} below satisfy the third condition of Proposition \ref{prop:appendix:delta_recursive_upper_bound_nt}. Set $\tilde{\gamma}_n=c_\gamma n^{-(\lambda+\gamma)}$, and remark that

\begin{align}
\mathbb{E}\left[  V_{n+1} \right] & \leq \left( 1 -2 \mu \lambda_{0}\tilde{\gamma}_{n+1} + C_{M}(n+1)^{2\beta+\lambda - \gamma}\tilde{\gamma}_{n+1} \right) \mathbb{E}\left[  V_{n} \right] + 2 \lambda_{0}\mu V \sqrt{v_n}\tilde{\gamma}_{n+1} \nonumber\\
&\hspace{8cm}+ \frac{C_{1}L_{\nabla G}}{2}(n+1)^\lambda\gamma_{n+1}\beta_{n+1}^{2}\tilde{\gamma}_{n+1}.\label{eq:bound_recursion_first_result}
\end{align}
Then, since $2\gamma-2\beta-1\not=1$, with the help of Proposition \ref{prop:appendix:delta_recursive_upper_bound_nt} and an integral test for convergence to get $\sum_{k=1}^{n}   k^{2\beta - 2\gamma}\leq 1+\frac{n^{(1+2\beta-2\gamma)^+}}{\vert 2 \gamma - 2\beta -1\vert} $ and $\sum_{t=\lfloor n/2\rfloor}^nt^{-\gamma}\geq\frac{1-2^{\gamma-1}}{1-\gamma}n^{1-\gamma}\geq n^{1-\gamma}$ for $\gamma \in (0,1)$ ,
\begin{align}
&\mathbb{E}\left[ V_{n} \right]   \leq \exp \left( - c_{\gamma} \mu \lambda_{0}n^{1-(\lambda+\gamma)} \right)\exp \left(  2 C_{M} c_{\gamma} \left(1+\frac{n^{(1+2\beta-2\gamma)^+}}{\vert 2 \gamma - 2\beta -1\vert}\right) \right)\cdot  \nonumber\\&\left( \mathbb{E}\left[ V_{0} \right] + 4\frac{ \lambda_{0}\mu V }{C_{M}} \max_{1 \leq k \leq n} k^{\gamma - 2 \beta-\lambda}\sqrt{\bar{v}_{k}}+ \frac{C_{1}L_{\nabla G}c_{\gamma}c_{\beta}^{2}}{C_{M}}  \right)
 + 2  V\sqrt{\bar{v}_{n/2}} + \frac{C_{1}L_{\nabla G}}{2^{1+\lambda}\mu \lambda_{0}}n^{\lambda}\beta_{n/2}^{2}\gamma_{n/2},\label{eq:first_result_general}
\end{align}
where we recall that $\bar{v}_{n}=v_{n+1}(n+1)^{-\delta-q\lambda}\geq\mathbb{P}\left[ \lambda_{\min} \left( A_{n} \right)<\lambda_n \right]$. Remark that 
\begin{align*}
k^{\gamma - 2 \beta-\lambda}\sqrt{\bar{v}_{k}}=\sqrt{v_{k+1}}(k+1)^{\gamma-2\beta -\lambda}(k+1)^{-(\delta+q\lambda)/2}=&\sqrt{v_{k+1}}(k+1)^{\gamma-2\beta-\delta/2-(q/2+1)\lambda},
\end{align*}
so that $\max_{0 \leq k \leq n} (k+1)^{\gamma - 2 \beta-\lambda}\sqrt{\bar{v}_{k}} =\max_{1 \leq k \leq n+1} k^{\gamma-2\beta-\delta/2-(q/2+1)\lambda}\sqrt{v_{k}}$.  Hence, we get
\begin{align*}
\mathbb{E}\left[ V_{n} \right]  & \leq \exp \left( -c_{\gamma}  \mu \lambda_{0} n^{1-(\lambda+\gamma)} \right)\exp \left( 2 C_{M} c_{\gamma} \left(1+\frac{n^{(1+2\beta-2\gamma)^+}}{\vert 2 \gamma - 2\beta -1\vert}\right) \right)\\
 &\hspace{3cm}\cdot\left( \mathbb{E}\left[ V_{0} \right] + 4\frac{ \lambda_{0}\mu V }{C_{M}}\max_{1 \leq k \leq n+1} k^{\gamma-2\beta-\delta/2-(q/2+1)\lambda}\sqrt{v_{k}} + \frac{C_{1}L_{\nabla G}c_{\gamma}c_{\beta}^{2}}{C_{M}}  \right) \\
& + 2^{1+(\delta+q\lambda)/2}  V\sqrt{v_{\lfloor n/2\rfloor}}n^{-(\delta+q\lambda)/2} + 2^{\gamma - 2\beta -\lambda-1}\frac{C_{1}L_{\nabla G}c_{\gamma}c_{\beta}^{2}}{\mu \lambda_{0}}n^{2\beta +\lambda- \gamma}
\end{align*}
where $V$ is defined in  Lemma \ref{lem::majvn2}. Hence, as long as $\gamma+\lambda+(1+2\beta-2\gamma)^+<1$ ,which is satisfied since $\lambda< \min\lbrace \gamma-2\beta,1-\gamma \rbrace$, we have 
\begin{align*}
\mathbb{E}\left[ V_{n} \right]   \leq \exp \left( -c_{\gamma} \mu \lambda_{0} n^{1-(\lambda+\gamma)} (1-\varepsilon'(n)\right)&\left( K_1^{(1)}+K_{1'}^{(1)}\max_{1 \leq k \leq n+1} k^{\gamma-2\beta-\delta/2-(q/2+1)\lambda}\sqrt{v_{k}}\right)\\
&\hspace{1,5cm}+K_2^{(1)}n^{-(\gamma-2\beta -\lambda)}+K_3^{(1)}\sqrt{v_{\lfloor n/2\rfloor}}n^{-(\delta+q\lambda)/2},
\end{align*}
with 
\begin{equation}\label{eq:espilon(n)}
\varepsilon'(n)=\frac{2 C_{M} n^{-1+\lambda+\gamma}}{ \mu \lambda_{0}} \left(1+\frac{n^{(1+2\beta-2\gamma)^+}}{\vert 2 \gamma - 2\beta -1\vert}\right),
\end{equation}
\begin{equation}\label{eq:3.1_first_constant}
K_1^{(1)}=\left( \mathbb{E}\left[ V_{0} \right]  + \frac{C_{1}L_{\nabla G}c_{\gamma}c_{\beta}^{2}}{C_{M}}  \right) ,\quad K_{1}^{(1')}=4\frac{ \lambda_{0}\mu V  }{C_{M}},
\end{equation}
where $C_M$ is given in \eqref{eq:constant_C_m} and $V$ in Lemma \ref{lem::majvn2} and
\begin{equation}\label{eq:3.1_second_constants}
K_2^{(1)}= 2^{\gamma - 2\beta -\lambda-1}\frac{C_{1}L_{\nabla G}c_{\gamma}c_{\beta}^{2}}{\mu \lambda_{0}},\, K_3^{(1)}=2^{1+(\delta+q\lambda)/2} V.
\end{equation}
\begin{lem}\label{lem::majvn2}
Suppose Assumption \textbf{(A1)} for  $p \geq 2$ and \textbf{(H1b)} hold. Then, for all $n \geq 0$, if $\gamma >1/2$ then

\[
\mathbb{E}\left[ V_{n}^{p} \right] \leq e^{ a_p c_{\gamma}^{2}c_{\beta}^{2} \frac{2\gamma -2 \beta}{2\gamma -2 \beta -1}} \max\left\lbrace 1 , \mathbb{E}\left[ V_{0}^{2} \right]  \right\rbrace :=V_n^p
\]
and if $\gamma\leq 1/2$ then
\begin{align*}
&\mathbb{E}\left[ V_{n}^{p} \right] \leq \exp\left(-p\mu\lambda_0 ' c_{\gamma}\left(1+\frac{1+\left(\frac{c_{\gamma}c_{\beta}^2 a_p}{ p\mu \lambda_0 ' }\right)^{\frac{1-\gamma-\lambda ' }{\gamma-2\beta-\lambda ' }}}{1-\gamma-\lambda ' }\right)+c_{\gamma}^2c_{\beta}^2a_p\left(1+\frac{1+\left(\frac{c_{\gamma}c_{\beta}^2 a_p}{p\mu \lambda_0 ' }\right)^{\frac{1-2\gamma+2\beta}{\gamma-2\beta-\lambda ' }}}{1-2\gamma+2\beta}\right)\right) =:V_p^p
\end{align*}
with $a_2$ given in \eqref{def::a2} and $a_{p}$ is given by \eqref{def::ap} for $p > 2$.
\end{lem}
The proof of this Lemma is given in Section \ref{sec::technic}.
\subsection{Proof of Theorem \ref{theo2}}
Remark that thanks to Assumption \textbf{(H1b)}, one has
\[
\mathbb{E}\left[ \left\| A_{n} \right\|^{2} \left\| g_{n+1}' \right\|^{2} |\mathcal{F}_{n} \right] \leq C_{1} \left\| A_{n} \right\|^{2} + \frac{C_{2}L_{\nabla G}}{\mu} \left\| A_{n} \right\|^{2}V_{n} \leq C_{1} \left\| A_{n} \right\|^{2} + \beta_{n+1}^{2} \frac{C_{2}L_{\nabla G}}{\mu} V_{n} .
\]
Moreover, with the help of Assumption \textbf{(H2a)}, 
\[
\mathbb{E}\left[ \left\| A_{n} \right\|^{2} \left\| g_{n+1}' \right\|^{2} \right] \leq  C_{1} C_{S}^{2} + \beta_{n+1}^{2} \frac{C_{2}L_{\nabla G}}{\mu} V_{n}
\]
leading as in the proof of Theorem \ref{theo1} to 
\begin{align*}
\mathbb{E}\left[  V_{n+1} \right]  \leq \left( 1 -2 \mu \lambda_{0}\gamma_{n+1} + \frac{C_{2}L_{\nabla G}}{\mu} \beta_{n+1}^{2}\gamma_{n+1}^{2} \right) \mathbb{E}\left[  V_{n} \right] &+ 2 \lambda_{0}\gamma_{n+1}\mu \mathbb{E}\left[ \mathbf{1}_{\lambda_{\min} \left( A_{n} \right) < \lambda_{n}} V_{n}\right]\\
&\hspace{1.5cm}+ \frac{C_{1}L_{\nabla G}C_{S}^{2}}{2}\gamma_{n+1}^{2}. \\
\end{align*}
Using Hölder inequality with $p$ yields then 
$$\mathbb{E}\left[ \mathbf{1}_{\lambda_{\min} \left( A_{n} \right) < \lambda_{n}} V_{n}\right]\leq \left( \mathbb{P}\left[\mathbf{1}_{\lambda_{\min} \left( A_{n} \right) < \lambda_{n}} \right]\right)^{\frac{p-1}{p}}\mathbb{E}\left[V_n^p\right]^{1/p}\leq \bar{v}_{n}^{\frac{p-1}{p}}V_p$$
with $\bar{v}_n=v_{n+1}(n+1)^{-\delta}$ and $V_p$ given in Lemma \ref{lem::majvn2}. Considering $C_{M}$ defined by 
\begin{equation}
\label{def::CM::theo2} C_{M} = \max \left\lbrace \frac{C_{2}L_{\nabla G}c_{\beta}^{2}c_{\gamma}}{\mu}, \left( \mu \lambda_{0} \right)^{\frac{2\gamma - 2\beta}{\gamma }}c_{\gamma}^{\frac{\gamma -2 \beta }{\gamma }} \right\rbrace,
\end{equation}
one has
\begin{align*}
\mathbb{E}\left[  V_{n+1} \right] & \leq \left( 1 -2 \mu \lambda_{0}\gamma_{n+1} + C_{M}(n+1)^{2\beta - \gamma}\gamma_{n+1} \right) \mathbb{E}\left[  V_{n} \right] + 2 \lambda_{0}\mu V_p \bar{v}_{n}^{\frac{p-1}{p}}\gamma_{n+1}\\
&\hspace{9cm}+ \frac{C_{1}L_{\nabla G}C_{S}^{2}}{2}\gamma_{n+1}^2. \\
\end{align*}
Then, applying Proposition \ref{prop:appendix:delta_recursive_upper_bound_nt} and with the help of integral tests for convergence, it comes
\begin{align}
\mathbb{E}\left[ V_{n} \right]  & \leq \exp \left( - c_{\gamma} \mu \lambda_{0}n^{1-\gamma} \right)\exp \left(  2 C_{M} c_{\gamma} \left(1+\frac{n^{(1+2\beta-2\gamma)^+}}{\vert 2 \gamma - 2\beta -1\vert}\right) \right)\cdot\nonumber\\
& \left( \mathbb{E}\left[ V_{0} \right] + 4\frac{ \lambda_{0}\mu V_p \max_{1 \leq k \leq n} k^{\gamma - 2 \beta}\bar{v}_{k}^{\frac{p-1}{p}}}{C_{M}}  + \frac{C_{1}L_{\nabla G}c_{\gamma}C_{S}^{2}}{C_{M}}  \right) 
 + 2  V_p\bar{v}_{n/2}^{\frac{p-1}{p}}\nonumber\\
 &\hspace{6cm} + 2^{\gamma  -1}\frac{C_{1}L_{\nabla G}c_{\gamma}C_{S}^{2}}{\mu \lambda_{0}}n^{ - \gamma} .\label{eq:second_result_general}
\end{align}
Concluding as in the proof of Theorem \ref{theo1}, we get 
\begin{align*}
\mathbb{E}\left[ V_{n} \right]  & \leq \exp \left( -  c_{\gamma}  \mu \lambda_{0} n^{1-\gamma} \left(1-\varepsilon(n)\right)\right)\cdot\left(K^{(2)}_1+K^{(2)}_{1'}\max_{1\leq j\leq n+1}v_k^{\frac{p-1}{p}}k^{\gamma - 2 \beta-\frac{p-1}{p}\delta}\right)\\
&\hspace{5cm}+ K^{(2)}_2v_{\lfloor n/2\rfloor}^{\frac{p-1}{p}}n^{-\frac{(p-1)}{p}\delta}+K^{(2)}_3n^{ - \gamma},
\end{align*}
with 
\begin{equation}
\label{eq:epsion(n)theo2} 
\varepsilon(n)=\frac{2 C_{M} n^{-1+\gamma}}{ \mu \lambda_{0}} \left(1+\frac{n^{(1+2\beta-2\gamma)^+}}{\vert 2 \gamma - 2\beta -1\vert}\right),
\end{equation}
where $C_{M}$ is defined by \eqref{def::CM::theo2} and
\begin{align}
K^{(2)}_1&=\left( \mathbb{E}\left[ V_{0} \right] +\frac{C_{1}L_{\nabla G}c_{\gamma}C_{S}^{2}}{C_{M}}  \right),\quad
K^{(2)}_{1'}=4\frac{ \lambda_{0}\mu V_p}{C_{M}},\label{eq:3.2_2_constant}\\
K^{(2)}_2&=2^{1+\delta\frac{p-1}{p}}V_p,\label{eq:3.2_3_constant}\\
K^{(2)}_3&=2^{\gamma  -1}\frac{C_{1}L_{\nabla G}c_{\gamma}C_{S}^{2}}{\mu \lambda_{0}}.\label{eq:3.2_4_constant}
\end{align}

\subsection{Proofs of Theorem \ref{theo3} and Corollary \ref{cor:simplified_theo3}}

\begin{proof}[Proof of Theorem \ref{theo3}]
Remark that one can rewrite 
\[
\theta_{n+1} - \theta = \theta_{n} - \theta - \gamma_{n+1} H^{-1} g_{n+1}' - \gamma_{n+1} \left( A_{n}- H^{-1} \right) g_{n+1}'
\]
leading, since $H$ is symmetric, to
\begin{align*}
 \left\| \theta_{n+1} - \theta \right\|^{2} & \leq \left\| \theta_{n} - \theta \right\|^{2} -2 \gamma_{n+1} \left\langle g_{n+1}' ,  H^{-1} \left( \theta_{n} - \theta \right) \right\rangle - 2 \gamma_{n+1} \left\langle \left(  A_{n} - H^{-1} \right) g_{n+1}' , \theta_{n} - \theta  \right\rangle \\
& + 2 \gamma_{n+1}^{2}  \left\| H^{-1} g_{n+1}' \right\|^{2} + 2\gamma_{n+1}^{2} \left\| A_{n}- H^{-1} \right\|^{2} \left\| g_{n+1}' \right\|^{2} 
\end{align*}
First, thanks to Assumption \textbf{(A3)} and by Cauchy-Schwarz inequality, 
\begin{align*}
(*) : = \left|  \mathbb{E}\left[ 2 \gamma_{n+1} \left\langle \left(  A_{n} - H^{-1} \right) g_{n+1}' , \theta_{n} - \theta \right\rangle |\mathcal{F}_{n} \right] \right| & = 2 \gamma_{n+1}  \left|   \left\langle \left(  A_{n} - H^{-1} \right) \nabla G \left( \theta_{n} \right) , \theta_{n} - \theta \right\rangle \right| \\
& \leq 2L_{\nabla G} \gamma_{n+1}\left\| A_{n} - H^{-1} \right\| \left\| \theta_{n} - \theta \right\|^{2}.
\end{align*}
Then, using Assumption \textbf{(A1')}, one has
\begin{align*}
(**)  :=  \mathbb{E}\left[  2 \gamma_{n+1}^{2}  \left\| H^{-1} g_{n+1}' \right\|^{2} | \mathcal{F}_{n} \right] & \leq 4 \gamma_{n+1}^{2} \text{Tr} \left( H^{-1} \Sigma H^{-1} \right) + 4 \gamma_{n+1}^{2} \left\| H^{-1} \right\|^{2} L_{\nabla g} \left\| \theta_{n}- \theta \right\|^{2}
\end{align*}
Finally, one has
\begin{align*}
(***) & = \mathbb{E}\left[ -2 \gamma_{n+1} \left\langle g_{n+1}' ,  H^{-1} \left( \theta_{n} - \theta \right) \right\rangle  |\mathcal{F}_{n}\right] \leq  - 2 \gamma_{n+1} \left\| \theta_{n} - \theta \right\|^{2} + 2\gamma_{n+1} \left\| H^{-1} \right\| \left\| \delta_{n} \right\| \left\| \theta_{n} - \theta \right\|
\end{align*}
with, using Assumption \textbf{(A4)}, $\left\| \delta_{n} \right\| := \left\| \nabla G \left( \theta_{n} \right) - H\left( \theta_{n} - \theta \right) \right\| \leq L_{\delta} \left\| \theta_{n} - \theta \right\|^{2}$. Hence,
\begin{align*}
(***) \leq - 2\gamma_{n+1} \left\| \theta_{n} - \theta \right\|^{2} + 2 \gamma_{n+1} \left\| H^{-1} \right\| L_{\delta} \left\| \theta_{n} - \theta \right\|^{3},
\end{align*}
which yields, using that $  \left\| \theta_{n} - \theta \right\|^3\leq \frac{1}{2a}\left\| \theta_{n} - \theta \right\|^2+\frac{a}{2}\left\| \theta_{n} - \theta \right\|^4$ with $a=\left\| H^{-1} \right\| L_{\delta}$,
\begin{align*}
(***) \leq - \gamma_{n+1} \left\| \theta_{n} - \theta \right\|^{2} +\gamma_{n+1} \left\| H^{-1} \right\|^{2} L_{\delta}^{2} \left\| \theta_{n} - \theta \right\|^{4}.
\end{align*}
Furthermore, 
\begin{align*}
(****) &  := \mathbb{E}\left[ 2 \gamma_{n+1}^{2} \left\| A_{n} - H^{-1} \right\|^{2} \left\| g_{n+1}' \right\|^{2}|\mathcal{F}_{n} \right]\\
 & \leq 2 \gamma_{n+1}^{2} \left\|  A_{n} - H^{-1}\right\|^{2} C_{1} + 2 \gamma_{n+1}^{2}C_{2} \left\| A_{n} - H^{-1} \right\|^{2} \left\| \theta_{n} - \theta \right\|^{2} \\
& \leq  2 \gamma_{n+1}^{2} \left\| A_{n} - H^{-1} \right\|^{2} C_{1} + C_{2} \gamma_{n+1} \left\| \theta_{n} - \theta \right\|^{4} + C_{2} \gamma_{n+1}^{3}  \left\| A_{n} - H^{-1} \right\|^{4}.
\end{align*}
As a conclusion, one has (after using Cauchy-Schwartz inequality on $(*)$),
\begin{align*}
\mathbb{E}\left[ \left\| \theta_{n+1} - \theta \right\|^{2} \right]  & \leq \left( 1- \gamma_{n+1} + 4 \left\| H^{-1} \right\|^{2} \gamma_{n+1}^{2}  L_{\nabla g} \right) \mathbb{E}\left[ \left\| \theta_{n} - \theta \right\|^{2} \right] + 4\gamma_{n+1}^{2} \text{Tr} \left( H^{-1}\Sigma H^{-1}\right) \\
&  + \gamma_{n+1} \left( \left\| H^{-1} \right\|^{2}L_{\delta}^{2} + C_{2} \right) \mathbb{E}\left[ \left\| \theta_{n} - \theta \right\|^{4} \right] + C_{2} \gamma_{n+1}^{3}  \mathbb{E}\left[ \left\|A_{n} - H^{-1} \right\|^{4} \right] \\
&  + 2C_{1}\gamma_{n+1}^2\mathbb{E}\left[ \left\| A_{n} - H^{-1} \right\|^{2} \right]+2\gamma_{n+1}L_{\nabla G}\sqrt{\mathbb{E}\left[ \left\| \theta_{n} - \theta \right\|^{4} \right]\mathbb{E}\left[ \left\| A_{n} - H^{-1} \right\|^{2}\right]},
\end{align*}
leading, using Proposition \ref{prop::ordre4} with the fact that $\mathbb{E}\left[\Vert \theta_n-\theta\Vert^4\right]\leq \frac{4}{\mu^2}\mathbb{E}\left[V_n^2\right]$ by \textbf{(A2)}, and \textbf{(H2b)} and \textbf{(H3)}, to
\begin{align*}
\mathbb{E}\left[ \left\| \theta_{n+1} - \theta \right\|^{2} \right]   \leq& \left( 1- \gamma_{n+1} + 4 \left\| H^{-1} \right\|^{2}\gamma_{n+1}^{2}  L_{\nabla g} \right) \mathbb{E}\left[ \left\| \theta_{n} - \theta \right\|^{2} \right] + 4\gamma_{n+1}^{2} \text{Tr} \left( H^{-1}\Sigma H^{-1}\right) \\
&  + \gamma_{n+1} \left( \left\| H^{-1} \right\|^{2}L_{\delta}^{2} + C_{2} \right) \frac{M_n}{\mu^2} + C_{2} \gamma_{n+1}^{3} 2^3 \left(C_S^4+\frac{1}{\mu^4}\right) \\
& + 2C_{1}\gamma_{n+1}^2v_{A,n} +2\gamma_{n+1}\frac{L_{\nabla G}}{\mu}\sqrt{M_n v_{A,n}}\\
\leq& \left( 1- \gamma_{n+1} + 4 \left\| H^{-1} \right\|^{2} \gamma_{n+1}^{2}  L_{\nabla g} \right) \mathbb{E}\left[ \left\| \theta_{n} - \theta \right\|^{2} \right]\\
& + \gamma_{n+1}\cdot\Bigg[4\gamma_{n+1} \text{Tr} \left( H^{-1}\Sigma H^{-1}\right) +\left( \frac{L_{\delta}^{2}}{\mu^2} + C_{2} \right) \frac{4M_n}{\mu^2} \\
&\hspace{2cm}+ C_{2} \gamma_{n+1}^{2} 2^3 \left(C_S^4+\frac{1}{\mu^4}\right)+ 2C_{1}\gamma_{n+1}v_{A,n} +4\frac{L_{\nabla G}}{\mu}\sqrt{M_n v_{A,n}}\Bigg].
\end{align*}
Finally, let us denote $C_{A} = c_\gamma\max \left\lbrace 4 \left\| H^{-1} \right\|^{2}  L_{\nabla g} , \frac{1}{4} \right\rbrace$.  
Then, with the help of Proposition \ref{prop:appendix:delta_recursive_upper_bound_nt}, one has
\begin{align*}
&\mathbb{E} \left[ \left\| \theta_{n} - \theta \right\|^{2} \right]  \leq  e^{ - \frac{1}{2}c_{\gamma}n^{1-\gamma} }e^{ 2 C_{A} c_{\gamma} \frac{2\gamma}{2\gamma -1} } \left( \mathbb{E}\left[ \left\| \theta_{0} - \theta \right\|^{2} \right] + \frac{8 \text{Tr} \left( H^{-1}\Sigma H^{-1} \right)}{C_{A}} + c_{\gamma}\frac{  16C_{2} \left( \mu^{-4} + C_{S}^{4} \right)}{C_{A}} + \frac{4C_{1} {v_{A,0}}}{C_{A}} \right) \\
&  + e^{ - \frac{1}{2}c_{\gamma}n^{1-\gamma} }e^{2 C_{A} c_{\gamma} \frac{2\gamma}{2\gamma -1} } \max_{1 \leq k \leq n} (k+1)^\gamma\cdot \left( 8\frac{L_{\delta}^{2}\mu^{-2} + C_{2}}{\mu^2C_{A}}M_{k-1}+ 8\frac{L_{\nabla G}}{C_A\mu}\sqrt{M_{k-1} v_{A,k-1}}     \right) \\
& + \frac{ 2^{3 + \gamma} c_{\gamma}\text{Tr} \left(H^{-1}\Sigma H^{-1} \right)}{n^{\gamma}} + \frac{8\left( \frac{ L_{\delta}^{2}}{\mu^{2}} + C_{2} \right)}{\mu^2} M_{n/2} + \frac{8L_{\nabla G}}{\mu} \sqrt{M_{n/2}v_{A,n/2}} \\
& + \frac{2^{4+2\gamma} C_{2}c_{\gamma}\left( \mu^{-4} + C_{S}^{4} \right) c_{\gamma}^{2} }{n^{2\gamma}} + \frac{2^{2+\gamma}C_{1}c_{\gamma} {v_{A,n/2}}}{n^{\gamma}}.
\end{align*}
Finally,
\begin{align*}
&\mathbb{E} \left[ \left\| \theta_{n} - \theta \right\|^{2} \right]  \leq  e^{ - \frac{1}{2}c_{\gamma}n^{1-\gamma} }\left(K^{(3)}_1+K^{(3)}_{1'}\max_{0\leq k \leq n} d_k(k+1)^{\gamma}\right) \\
& + n^{-\gamma}\left(2^{3 + \gamma} c_{\gamma}\text{Tr} \left(H^{-1}\Sigma H^{-1} \right)+\frac{K_2^{(3)}}{n^{\gamma}}+ K_{2'}^{(3)}v_{A,n/2}\right)+ d_{\lfloor n/2\rfloor}.
\end{align*}
with 
\begin{align}
K^{(3)}_1=&e^{ 2 C_{A} c_{\gamma}^{2} \frac{2\gamma}{2\gamma -1} } \left( \mathbb{E}\left[ \left\| \theta_{0} - \theta \right\|^{2} \right] + \frac{8 \text{Tr} \left( H^{-1}\Sigma H^{-1} \right)}{C_{A}} + c_{\gamma}\frac{  16C_{2} \left( \mu^{-4} + C_{S}^{4} \right)}{C_{A}} + \frac{4C_{1} {v_{A,0}}}{C_{A}} \right),\label{eq:theo3_constant_1}\\
K^{(3)}_{1'}=&\frac{1}{C_A}e^{ 2 C_{A} c_{\gamma}^{2} \frac{2\gamma}{2\gamma -1} ,},\quad 
d_n= 8L_{\nabla G} \sqrt{M_nv_{A,n}} + 8\frac{L_{\delta}^{2}\mu^{-2} + C_{2}}{\mu^2}M_n,\label{eq:theo3_constant_2}  
\end{align}
where we recall that $C_A= c_\gamma\max \left\lbrace 4  \left\| H^{-1} \right\|^{2} L_{\nabla g} , \frac{1}{4} \right\rbrace$, and 
\begin{equation}\label{eq:theo3_constant_3}
K_2^{(3)}=2^{4+2\gamma} C_{2}c_{\gamma}\left( \mu^{-4} + C_{S}^{4} \right) c_{\gamma}^{2} ,\quad K_{2'}^{(3)}=2^{2+\gamma}C_{1}c_{\gamma}.
\end{equation}
\end{proof}
\begin{proof}[Proof of Corollary \ref{cor:simplified_theo3}]
Remark that as long as $\delta\frac{p-2}{p}\geq 2\gamma$, by Proposition \ref{prop::ordre4} and the following discussion, 
\begin{align*}
\max_{0\leq k \leq n} d_k(k+1)^{\gamma}=\max_{0\leq k \leq n} \left((k+1)^{\gamma}8L_{\nabla G} \sqrt{M_kv_{A,k}} + 8\frac{L_{\delta}^{2}\mu^{-2} + C_{2}}{\mu^2}M_k\right)\\
\leq \frac{8L_{\nabla G}\sqrt{v_{A},0}}{c_\gamma} \sqrt{w_{\infty}(2\gamma)}+8\frac{L_{\delta}^{2}\mu^{-2} + C_{2}}{\mu^2}w_{\infty}(\gamma).
\end{align*}
Likewise,
$$M_{n/2}\leq \frac{2^{2\gamma}w_{\infty}(2\gamma)}{n^{2\gamma}}.$$
Hence, plugging these inequalities into Theorem \ref{theo3} yields
\begin{align*}
\mathbb{E} \left[ \left\| \theta_{n} - \theta \right\|^{2} \right]  \leq& n^{-\gamma}\left(2^{3 + \gamma} c_{\gamma}\text{Tr} \left(H^{-1}\Sigma H^{-1} \right)+\frac{K_2^{(3')}}{n^{\gamma}}+ K_{2'}^{(3')}v_{A,n/2}+ K_{2''}^{(3')}\sqrt{v_{A,n/2}}\right)\\
&\hspace{9cm} +K^{(3')}_{1}e^{ - \frac{1}{2}c_{\gamma}n^{1-\gamma} },
\end{align*}
with 
\begin{equation}\label{eq:cor_constant_expo}
K_1^{(3')}=K_{1}^{(3)}+K_{1}^{(3')}\left(\frac{8L_{\nabla G}\sqrt{v_{A},0}}{c_\gamma} \sqrt{w_{\infty}(2\gamma)}+8\frac{L_{\delta}^{2}\mu^{-2} + C_{2}}{\mu^2}w_{\infty}(\gamma)\right),
\end{equation}
\begin{equation}\label{eq:cor_constant_main}
K_{2}^{(3')}=K_{2}^{(3)}+2\frac{L_{\delta}^{2}\mu^{-2} + C_{2}}{\mu^2}2^{2\gamma}w_{\infty}(2\gamma),\quad K_{2'}^{(3')}=K_{2'}^{(3)},\, K_{2''}^{(3')}=2^{2+\gamma}L_{\nabla G} \sqrt{w_{\infty}(2\gamma)}.
\end{equation}
\end{proof}
\subsection{Proof of Theorem \ref{theo:adagrad}}
To prove this theorem, we will apply Theorem \ref{theo2}. We first need to check that $(A_n)_{n\geq 0}$ satisfies Assumptions \textbf{(H1a), (H1b)} and \textbf{(H2)}. Assumption \textbf{(H1b)} is given by construction (see \eqref{eq:modification_An_adagrad}) while \textbf{(H1a)} is given by the following lemma:
\begin{lem}\label{lem:H1_adagrad}
Assume \textbf{(A1)} is satisfied for some $p>2$. Then, for all $0<t<1$,
$$\mathbb{P}\left[ \lambda_{\min}\left( A_n \right)<c_{\beta}t\right]\leq v_nt^{2p},$$
with 
$$v_n=c_{\beta}^p\left(\left(\frac{1}{n}\sum_{i=1}^da_{k}\right)^p+C_1''+\frac{2^pC_2''V_p^p}{\mu^p}\right).$$
\end{lem}
The proof is given in Appendix \ref{sec::technic}.
Remark that $\mathbb{E}\left[V_n^p\right]<+\infty$ by Lemma \ref{lem::majvn2} with \textbf{(A1)}. Assume from now  that $ p >   2 $ and let $p' = \frac{2(1-\gamma)}{2-\gamma}p$ and $\lambda = (1-\gamma)(\gamma -2\beta)$.  Remark that $\lambda < 1- \gamma$, $\lambda < \gamma -2 \beta$ and $p' < p$. Hence, applying Proposition \ref{prop::ordrep'} with $\lambda_0=c_{\beta}$, $\delta=0$, $q=2p$, 
\begin{align*}
\mathbb{E}\left[ V_{n}^{p'} \right]   \leq \exp \left( -c_{\gamma} \mu \lambda_{0} n^{1-(\lambda+\gamma)} (1-\varepsilon'(n)\right)&\left( K_1^{(1')}+K_{1'}^{(1')}\max_{1 \leq k \leq n+1} k^{\gamma-2\beta-\lambda-2(p-p')\lambda}v_{0}^{\frac{p-p'}{p}}\right)\\
&\hspace{0,5cm}+K_2^{(1')}n^{-p'(\gamma-2\beta -\lambda)}+K_3^{(1')}v_{0}^{\frac{p-p'}{p}}(n+1)^{-2(p-p')\lambda},
\end{align*}
with $\epsilon'(n)$,  $K^{(1')}_1$, $K^{(1')}_{1'}$, $K^{(1')}_2$ and $K^{(1')}_3$ respectively given in \eqref{eq:espilon(n)'}, \eqref{eq:3.1p_first_constant} and \eqref{eq:3.1p_second_constants} with $\lambda_0=c_{\beta}$. By the choice of $\lambda,p'$ one has
\[
p'(\gamma -2\beta - \lambda) = p \frac{2(1-\gamma)}{2-\gamma} \gamma ( \gamma -2 \beta) = 2(p-p') \lambda ,
\]
 so that 
\begin{equation}\label{eq:bound_Vp'_adagrad}
\mathbb{E}\left[ V_{n}^{p'} \right]   \leq \tilde{K}_1\exp \left( -c_{\gamma} \mu c_{\beta} n^{1-((1-\gamma)(\gamma -2 \beta)+\gamma)} (1-\varepsilon'(n)\right)+ \tilde{K}_2(n+1)^{- \frac{2(1-\gamma)\gamma (\gamma -2 \beta)}{2-\gamma}p }:=c_n
\end{equation}
with 
$$\tilde{K}_1= K_1^{(1')}+K_{1'}^{(1')}v_{0}^{\frac{\gamma}{2- \gamma}},\quad \tilde{K}_2=K_2^{(1')}+K_3^{(1')}v_{0}^{\frac{\gamma	}{2- \gamma}}.$$ By strong convexity, one can so obtain a first   rate of convergence of the estimates. %the good rate of convergence as soon as $p > \frac{2-\gamma}{2(1-\gamma )(\gamma -2 \beta) }$.
The following lemma enables to ensure that \textbf{(H1a)} is satisfied, but with a possibly better rate than with Lemma \ref{lem:H1_adagrad}.
 
\begin{lem}\label{lem:tech:ada:lambda}
Assume \textbf{(A1)} is satisfied for some  {$p > 2$}. Then, 
$$\mathbb{P}[\lambda_{\min}(A_n)<\tilde{\lambda}_0]\leq \frac{v_0\log(n+1)}{(n+1)^{\frac{2(1-\gamma)\gamma (\gamma -2 \beta)}{2-\gamma}p \wedge 1}},$$
with $\tilde{\lambda}_0 = \left[ \frac{2(1-\gamma)}{2-\gamma}p\left(C_{\left( \frac{2(1-\gamma)}{2-\gamma}\right)}+1\right)\right]^{-\frac{2-\gamma}{2(1-\gamma)p}}$ and $v_0$ is given in \eqref{eq:Adagrad_v0} with $p' = \frac{2(1-\gamma)}{2-\gamma}p$.
\end{lem}
The proof is given in Appendix \ref{sec::technic}.
We can also deduce from \eqref{eq:bound_Vp'_adagrad} a bound on $\mathbb{E}\left[ \Vert A_n\Vert^4\right]$ in case only \textbf{(A6)} holds.
\begin{lem}\label{lem:H2_adagrad}
Assume Assumptions \textbf{(A1)-(A6)}  and \textbf{(A1')} hold  for some $p > 2$. Then, for $\beta< \min\left\lbrace \frac{ (1-\gamma)\gamma(\gamma-2\beta)p}{4(2-\gamma)}  , 1/4 \right\rbrace$, the sequence of random matrices $\left( A_{n} \right)$ defined by \eqref{def::an::ada} verifies
$$\mathbb{E}\left[\Vert A_n\Vert^4\right]\leq C_{S}^4,$$
with $C_{S}^4$ given in \eqref{eq:CS4_adagrad}. 
\end{lem}
The proof is given in Appendix \ref{sec::technic}.
If the stronger hypothesis \textbf{(A6')} holds, an improved and simpler bound on $\mathbb{E}\left[\Vert A_n\Vert^4\right]$ can be reached, as next lemma shows.
\begin{lem}\label{lem:tech:ada:last:prime}
Assume Assumptions \textbf{(A1)-(A6')}  and \textbf{(A1')} hold  for some $p>2$. Then, for $\beta<\min \lbrace \gamma/2\wedge 1/4 \rbrace $, the sequence of random matrices $\left( A_{n} \right)$ defined by \eqref{def::an::ada} verifies
$$\mathbb{E}\left[\Vert A_n\Vert^4\right]\leq C_{S}^4,$$
with $C_{S}^4$ given in \eqref{eq:CS4_adagrad_bis}. 
\end{lem}
The proof is given in Appendix \ref{sec::technic}.
Theorem \ref{theo:adagrad} is then a consequence of Theorem \ref{theo2} whose hypotheses are satisfied thanks to   Lemma \ref{lem:H1_adagrad}, \ref{lem:tech:ada:lambda} and \ref{lem:H2_adagrad} (or \ref{lem:tech:ada:last:prime}). We then have 
\begin{align*}
\mathbb{E}\left[ V_{n} \right]  & \leq \exp \left( -  c_{\gamma}  \mu \tilde{\lambda}_{0} n^{1-\gamma} \left(1-\varepsilon(n)\right)\right)\cdot\left(K^{(2)}_1+K^{(2)}_{1'}\max_{1\leq j\leq n+1}v_k^{\frac{p-1}{p}}k^{\gamma - 2 \beta-\frac{p-1}{p}\delta}\right)\\
& + K^{(2)}_2v_{\lfloor n/2\rfloor}^{\frac{p-1}{p}}n^{-\frac{(p-1)}{p}\min\left\lbrace  \frac{2(1-\gamma)\gamma(\gamma-2\beta)p}{2-\gamma} , 1 \right\rbrace}+K^{(2)}_3n^{ - \gamma}
\end{align*}
with $K^{(2)}_1,K^{(2)}_{1'},K^{(2)}_2$ and $K^{(2)}_{3}$ respectively given in \eqref{eq:3.2_2_constant}, \eqref{eq:3.2_3_constant} and \eqref{eq:3.2_4_constant} with $\delta= \min\left\lbrace  \frac{2(1-\gamma)\gamma(\gamma-2\beta)p}{2-\gamma} , 1 \right\rbrace$, $\lambda_0$ given in \eqref{eq:Adagrad_def_lambda0}, $v_n=v_0\log(n+1)$ with $v_0$ given in \eqref{eq:Adagrad_v0} and $C_{S}$ given in \eqref{eq:CS4_adagrad} or \eqref{eq:CS4_adagrad_bis} depending on whether \textbf{(A6)} or \textbf{(A6')} holds. By strong convexity 
\begin{align*}
\mathbb{E}\left[ \Vert \theta_n-\theta\Vert^2 \right]  & \leq \tilde{K}_{1}^{(4)}\exp \left( -  c_{\gamma}  \mu \tilde{\lambda}_{0} n^{1-\gamma} \left(1-\tilde{\varepsilon}(n)\right)\right)+ \tilde{K}^{(4)}_2 \left( v_{0}\log (n+1) \right)^{\frac{p-1}{p}}n^{-\frac{(p-1)}{p}\min \left\lbrace \frac{2(1-\gamma)\gamma(\gamma-2\beta)p}{2-\gamma} ,  1 \right\rbrace } \\
& +\tilde{K}^{(4)}_3n^{ - \gamma},
\end{align*}
with $\tilde{\lambda}_{0}$ defined in \eqref{eq:Adagrad_def_lambda0}
\begin{equation}\label{eq:espilon(n):adagrad}
\tilde{\varepsilon}(n)=\frac{2 C_{M} n^{-1+(1-\gamma)(2\gamma - \beta)+\gamma}}{ \mu \tilde{\lambda}_{0}} \left(1+\frac{n^{(1+2\beta-2\gamma)^+}}{\vert 2 \gamma - 2\beta -1\vert}\right),
\end{equation}
with 
\begin{align}\label{eq:Adagrad_constant}
 & \tilde{K}_{1}^{(4)}=\frac{2}{\mu}\left(K^{(2)}_1+K^{(2)}_{1'}v_0\right), \quad \tilde{K}_{2}^{(4)}=\frac{2K_2^{(4)}}{\mu},\quad \tilde{K}_{3}^{(4)}=\frac{2 K_3^{(4)}}{\mu}.
\end{align}
where  $v_{0}$ is given in \eqref{eq:Adagrad_v0}.
\subsection{Proofs of Theorem \ref{theo::LM} and Theorem \ref{theo::LM_ada}}
The proof relies on the verification of each assumption needed in Theorem \ref{theo3}.
\medskip

\noindent\textbf{Verifying Assumptions \textbf{(A1)}, \textbf{(A1')} to \textbf{(A6)}. }
First, remark that 
$$\left\| \nabla_{h} g \left( X , Y , h \right) \right\| \leq \left\| \left( X^{T}h - X^{T}\theta - \epsilon \right) X \right\| \leq \left| \epsilon \right| \left\| X \right\| + \left\| X \right\|^{2} \left\| h - \theta \right\|.$$ 
Then, if $X$ and $\epsilon$ respectively admit moments of order $4p$ and $2p$, since $\epsilon $ and $X$ are independent,
\[
\mathbb{E}\left[ \left\| \nabla_{h} g \left( X ,Y , h \right) \right\|^{2p} \right] \leq \sigma_{(2p)} + C_{(2p)} \left\| h - \theta \right\|^{2p}
\]
with $\sigma_{(t)}=2^{t-1}\mathbb{E}\left[\vert \epsilon\vert^t\right]\mathbb{E}\left[ \Vert X\Vert^t\right]$ and $C_{(t)}=2^{t-1}\mathbb{E}\left[\Vert X\Vert^{2t}\right]$. In a particular case, if $p\geq 2$, Assumption \textbf{(A1)} is verified. Furthermore, since for all $h$,  $\nabla^{2}G(h)= \mathbb{E}\left[ XX^{T} \right]$ is positive, \textbf{(A2)} to  \textbf{(A4)} hold with $\mu=\lambda_{\min}\left( \mathbb{E}\left[ XX^{T} \right] \right)=:\lambda_{\min}$, $L_{\nabla G} = \lambda_{\max} \left( \mathbb{E}\left[ XX^{T} \right] \right) =: \lambda_{\max}$ and \textbf{(A5)} holds with $L_{\delta}= 0$. Finally
Assumption \textbf{(A1')} is verified since
\begin{align*}
\mathbb{E}\left[ \left\|  \nabla_{h} g \left( X ,Y, h \right) - \nabla_{h} g \left( X , Y,\theta \right)   \right\|^{2} \right]=&\mathbb{E}\left[ \left\|   X^T(h-\theta)X \right\|^{2} \right]\\
 \leq& \underbrace{ \mathbb{E}\left[  \left\| X \right\|^{4} \right]}_{=: L_{\nabla g}} \left\| h- \theta \right\|^{2}.
\end{align*}

\medskip

We can now prove Theorem \ref{theo::LM}

\begin{proof}[Proof of Theorem \ref{theo::LM}]
\noindent\textbf{Verifying Assumption (H1) for Stochastic Newton algorithm. } Let us first check Assumption \textbf{(H1)} for $\widetilde{S}_n=\frac{1}{n+1}\left[S_0+\sum_{i=1}^n X_iX_i^T\right]$.
\begin{lem}\label{lem:(H1)_regression_first_step}
Suppose that $X$ admits $4p$-moments, with $p>2$. Then, for $\lambda_0=\frac{1}{2\mathbb{E}\Vert X\Vert^2}$, we have 
$$\mathbb{P}\left[ \lambda_{\min}\left(\widetilde{S}_{n}^{-1}\right)<\lambda_0\right]\leq \tilde{v}_n$$
with 
$$\tilde{v}_n=\frac{2^{p-1}}{\left(\mathbb{E}\left[ \Vert X\Vert^2 \right]\right)^p}\left(C_1(p)n^{1-p}\mathbb{E}\left[ \vert Z\vert^p\right] +C_2(p)n^{-p/2}\left(\mathbb{E}\left[ \vert Z\vert^2 \right]\right)^{p/2} + \left\| S_{0} \right\|^{p}n^{-p}\right),$$
where $Z=\Vert X\Vert^2-\mathbb{E}\left[ \Vert X\Vert^2 \right]$ and $C_1(p), C_2(p)$ are numerical constants given in Rosenthal inequality, see \cite{Pinelis}.
\end{lem}
The proof is given in Section \ref{sec::proof:lem::linear}. To deal with $\overline{S}_n=\frac{\Vert \widetilde{S}_n^{-1}\Vert}{\min(n^\beta,\Vert \widetilde{S}_n^{-1}\Vert)}\tilde{S}_n$, one first needs the following control on the behavior of $\lambda_{\min} (\widetilde{S}_n)$. Set $H=\mathbb{E}\left[XX^T\right]$.
\begin{prop}[See \cite{Koltchinskii}, Theorem 1.5 and Theorem 3.3]\label{Mendelson_result}
Suppose that $0< \lambda_{\min}I_{d}\leq H := \mathbb{E}\left[ XX^{T} \right] \leq \lambda_{\max}I_{d}$ and that there exists $L_{MK}>0$ such that $\mathbb{E}\left[  \langle X,t\rangle^{2} \right]\leq L_{MK} \mathbb{E}\left[ \left|  \langle X,t\rangle\right| \right]$ for all $t\in \mathbb{S}^{d-1}$. Then, for $n\geq c_1d$,
$$\mathbb{P}\left[ \lambda_{\min} \left(\frac{1}{n}\sum_{i=1}^n X_iX_i^T\right)\leq c_2 \right]\leq 2\exp\left(-c_3n\right),$$
with $c_1=\frac{\lambda_{\max}^2(16L_{MK})^4}{\lambda_{\min}^2}$, $c_2=\frac{\lambda_{\min}}{8\sqrt{2}L_{MK}^2}$ and $c_3=\frac{1}{128L_{MK}^4}$.
\end{prop} 
Remark that the constant $c_1,\, c_2$ and $c_3$ are fairly explicit in terms of $L_{MK}$ and $\lambda_{\min}$. For the latter result and Lemma \ref{lem:(H1)_regression_first_step} and Proposition \ref{Mendelson_result} we deduce Hypothesis (H1) for $\overline{S}_n$. We will need several times the threshold 
\begin{equation}\label{eq:threshold_linear}
n_{0}=\max\left\lbrace c_1 d,\left(\frac{1}{c_\beta c_2}\left(1+\frac{1}{c_1d}\right)\right)^{-1/\beta}\right\rbrace.
\end{equation}
\begin{lem}\label{lem:(H1)_regression}
Suppose that $X$ satisfies hypothesis of Proposition \ref{Mendelson_result} and admits $4p$-moments, with $p>2$. Then, for $\lambda_0=\frac{1}{2\mathbb{E}\left[ \Vert X\Vert^2 \right]}$, we have 
$$\mathbb{P}\left[ \lambda_{\min}\left(\overline{S}_n^{-1}\right)<\lambda_0\right]\leq v_{n+1}(n+1)^{-p/2}$$
with $\delta=p/2$, $v_{n+1}=(n+1)^{\delta}$ for $n\leq n_0$ and, for $n>n_0$,
$$v_{n}=2\exp(-c_3n)n^{p/2}+\frac{2^{p-1}\left(C_2(p)\mathbb{E}\left[ \vert Z\vert^2 \right]^{p/2} +C_1(p)n^{1-p/2}\mathbb{E}\left[ \vert Z\vert^p\right] + \left\| S_{0} \right\|^{p}n^{-p/2}\right)}{\mathbb{E}\left[ \Vert X\Vert^2 \right]^p},$$
where $c_1,\,c_2,\,c_3$ are given in Proposition \ref{Mendelson_result}, $C_1(p)$ and $C_2(p)$ are numerical constants depending on $p$ and $Z=\Vert X\Vert^2-\mathbb{E}\left[\Vert X\Vert^2\right]$. 
\end{lem}
The proof is given in Section \ref{sec::proof:lem::linear}. As a particular case, Assumption \textbf{(H1a)} is verified with a rate $\delta= p/2$ when $\gamma>1/2$.

\medskip

\noindent \textbf{Verifying Assumption (H2) for Stochastic Newton algorithm. } A straightforward deduction of the above lemma is the following.
\begin{lem}\label{lem::jesaismemepascequecest} Suppose that hypothesis of Proposition \ref{Mendelson_result} hold and that $X$ admits a moment of order $4p$ with $p> 2$. Then, for all $\kappa>0$, we have
$$\mathbb{E}\left[\Vert \bar{S}_{n}^{-1}\Vert^{\kappa}\right]\leq 2\beta_{n+1}^\kappa\exp(-c_3n)+c_2^{-\kappa}$$
 for $n\geq c_1d$ and
$$\mathbb{E}\left[\Vert \bar{S}_{n}^{-1}\Vert^{\kappa}\right]\leq \left[(c_1 d+1)\Vert S_0^{-1}\Vert\right]^\kappa$$
for $n\leq c_1d$, with $c_1, \, c_2,\, c_3$ given in Proposition \ref{Mendelson_result}. 
\end{lem}
The proof is given in Section \ref{sec::proof:lem::linear}.
Finally, the following proposition gives a precise bound for Assumption \textbf{(H2)}.
\begin{prop}\label{prop::H2} Suppose that hypothesis of Proposition \ref{Mendelson_result} hold and that $X$ admits a moment of order $4p$ with $p> 2$. Then
$$\mathbb{E}\left[\Vert \bar{S}_{n}^{-1}\Vert^{2}\right]\leq \max\left\lbrace 2c_\beta^2\left(\frac{2\beta}{ec_3}\right)^{2\beta}+c_2^{-2}, \left[(c_1 d+1)\left\|  S_0^{-1}\right\| \right]^2\right\rbrace \leq C_{S}^{2} $$
and 
$$\mathbb{E}\left[\Vert \bar{S}_{n}^{-1}\Vert^{4}\right]\leq  \max\left\lbrace 2c_\beta^4\left(\frac{4\beta}{ec_3}\right)^{4\beta}+c_2^{-4}, \left[(c_1 d+1)\left\|  S_0^{-1}\right\|\right]^4\right\rbrace \leq  C_{S}^{4} $$
for all $n\geq 0$, with $ C_{S}^{4} := \max\left\lbrace \left( 2c_\beta^2\left(\frac{4\beta}{ec_3}\right)^{2\beta}+c_2^{-2} \right)^{2}, \left[(c_1 d+1)\left\|  S_0^{-1}\right\|\right]^4\right\rbrace $
\end{prop}
The proof is given in Section \ref{sec::proof:lem::linear}. Remark that  $C_{S}=O(d)$.

\medskip

\noindent\textbf{A first convergence result. } 
Since in the case of the linear model, one as
$C_{1} = \sigma_{(2)},C_{1}' = \sigma_{(4)},C_{2}=C_{(2)},C_{2}' = C_{(4)},L_{\nabla G} = \lambda_{\max}, \mu = \lambda_{\min}, \lambda_{0} = \frac{1}{2 \mathbb{E}\left[ \left\| X \right\|^{2} \right]}$ $\delta = p/2$, Proposition \ref{prop::ordre4} can now be written as follows:
\begin{prop}\label{prop::lr}
Suppose that there is $p >2$ such that $X,\epsilon$ respectively admit moments of orders $4p$ and $2p$. Suppose also that there is a positive constant $L_{MK}$ such that for any $h \in \mathbb{S}^{d-1}$, $\sqrt{\mathbb{E}\left[ hXX^{T}h \right]} \leq L_{MK} \mathbb{E}\left[ \left| X^{T}h \right| \right]$.    Then, denoting $\lambda_{\min}$ and $\lambda_{\max}$ the smallest and largest eigenvalues of  $\mathbb{E}\left[ XX^{T} \right]$,
\begin{align*}
\mathbb{E}\left[ V_{n}^{2} \right]  \leq &\exp \left( - \frac{3c_\gamma\lambda_{\min}}{4 \mathbb{E}\left[ \left\| X \right\|^{4} \right]} n^{1-\gamma} \right)\left(K^{(2')}_{1,\text{lin}}+K^{(2')}_{1',\text{lin}}\max_{1\leq k\leq n+1}v_k^{\frac{p-2}{p}}k^{\gamma-\frac{p-2}{p}}\right) \\
 &\hspace{6cm}+K^{(2')}_{2,\text{lin}}n^{-2\gamma} + K^{(2')}_{3,\text{lin}}v_{\lfloor n/2\rfloor}^{(p-2)/p}n^{- (p-2)/2} 
:=c_{n,\text{lin}}.
\end{align*}
with $v_n$ given by Lemma \ref{lem:(H1)_regression} and
$$K_{1,\text{lin}}^{(2')}=e^{ 2 a_{M,\text{lin}}\frac{2 \gamma -2 \beta}{2\gamma  -2 \beta -1} }\left( \mathbb{E}\left[ V_{0}^{2} \right] + \frac{2a_{1,\text{lin}}c_{\gamma}^{2} }{ a_{M,\text{lin}}} \right), \quad K^{(2')}_{1',\text{lin}}= e^{ 2 a_{M,\text{lin}}\frac{2 \gamma -2 \beta}{2\gamma  -2 \beta -1} }\frac{4\lambda_{\min}V_{p,\text{lin}}^{2}}{a_{M,\text{lin}}\mathbb{E}\left[ \left\| X \right\|^{2} \right]},$$
$$K^{(2')}_{2,\text{lin}}=\frac{2^{1+2\gamma}a_{1,\text{lin}}c_{\gamma}^{2}\mathbb{E}\left[ \left\| X \right\|^{2} \right]}{  3\lambda_{\min}},\quad K^{(2')}_{3,\text{lin}}= \frac{2^{p/2+1}}{3} V_{p,\text{lin}}^{2},$$
where, recalling the notations $\sigma_{(t)} = 2^{ t-1} \mathbb{E}\left[ \left|\epsilon \right|^{ t} \right] \mathbb{E}\left[ \left\| X \right\|^{ t} \right]$ and $C_{(t)} = 2^{t-1}\mathbb{E}\left[ \left\| X \right\|^{2t} \right]$,
$$a_{M,\text{lin}} :=  \max \left\lbrace  \left(  \frac{2\lambda_{\max}C_{(2)}}{ \lambda_{\min}} +\frac{2\lambda_{\max}^{2}}{\lambda_{\min}^{2}} \left( 4C_{(2)} + C_{(4)}c_{\gamma}^{2}c_{\beta}^{2} \right)   \right) c_{\gamma}c_{\beta}^{2} , \left( \frac{3\lambda_{\min}}{4\mathbb{E}\left[ \left\| X \right\|^{2} \right]} \right)^{\frac{2\gamma -2\beta}{\gamma}}c_{\gamma}^{\frac{\gamma -2 \beta}{\gamma}} \right\rbrace ,$$ 
with $C_{S}$ given by Proposition \ref{prop::H2}, 
$a_{1,\text{lin}}:= C_{S}^{4}\lambda_{\max}^{2}\left( \frac{16\lambda_{\max}^{2}\sigma_{(2)}^{2}\mathbb{E}\left[ \left\| X \right\|^{2} \right]}{\lambda_{\min}^{3}} + \frac{\sigma_{(4)}c_{\gamma}}{2} +   \frac{2C_{(2)}^{2}\mathbb{E}\left[ \left\| X \right\|^{2} \right]}{  \lambda_{\min} } \right)$  and 
\[
\mathbb{E}\left[ V_{n,}^{p} \right] \leq e^{ a_{p,\text{lin}} c_{\gamma}^{2}c_{\beta}^{2} \frac{2\gamma -2 \beta}{2\gamma -2 \beta -1}} \max\left\lbrace 1 , \mathbb{E}\left[ V_{0}^{2} \right]  \right\rbrace :=V_{p ,\text{lin}}^p
\]
where
\begin{align}
\notag &  a_{p,\text{lin}} : = p\left( \frac{C_{(2)}}{\lambda_{\min}} + \frac{\sigma_{(2)}}{2} \right) + 2^{p-2}(p-1)p \lambda_{\max}^{2} \left( c_{\gamma}^{2}c_{\beta}^{2} \left(\sigma_{(4)} + \frac{4C_{(4)}}{\lambda_{\min}^{2}} \right) + \frac{2\sigma_{(2)}}{\lambda_{\min}} + \frac{4C_{(2)}}{\lambda_{\min}^{2}} \right) \\
\label{def::ap::lin} &  + 2^{p-2}(p-1)p\lambda_{\max}^{p} \left( c_{\gamma}^{2p-2}c_{\beta}^{2p-2} \left(\sigma_{(2p)} + \frac{2^{p}C_{(2p)}}{\lambda_{\min}^{2}} \right)  + c_{\gamma}^{p-2}c_{\beta}^{p-2} \left( \frac{1}{2}\sigma_{(2p)} + \frac{2p}{\lambda_{\min}^{2}} \left( \frac{1}{2} + \sqrt{C_{(2p)}} \right) \right) \right) .
\end{align}

  \end{prop}

\noindent\textbf{Verifying Assumption (H3) for Stochastic Newton algorithm.} Hypothesis \textbf{(H3)} is then a straightforward combination of the convergence of $\overline{S}_n$ towards $H$, together with Hypothesis \textbf{(H2)}.
\begin{lem}\label{lem:Hypothesis_H3_regression}
Suppose that $X$ admits moments of order $2p$ with $p>4$, and let suppose as well that the distribution of $X$ satisfies hypothesis of Proposition \ref{Mendelson_result}. Then, for $n\geq n_0$ (with $n_0$ defined in \eqref{eq:threshold_linear}),
\begin{equation}
\label{def::vhn}\mathbb{E}\left[ \left\| \overline{S}_n^{-1}-H^{-1} \right\|^2 \right]\leq \frac{4\left(\mathbb{E} \left[ \Vert X\Vert^{2p} \right]\right)^{2/p}}{\left(\lambda_{\min}\beta_n\right)^2}e^{-c_3(p-2)n/p}+ \frac{2 \mathbb{E}\left[ \| X\|^{4} \right] }{ n\left( \lambda_{\min}c_2 \right)^{2}}+\frac{2\left\Vert S_0-H\right\Vert_{F}^2}{n^2\left( \lambda_{\min}c_2 \right)^{2}} =: v_{H,n} .
\end{equation}
\end{lem}
For $n<n_0$, we simply bound 
\[
\mathbb{E}\left[ \left\| \overline{S}_n^{-1}-H^{-1} \right\|^2 \right]\leq \max\left\lbrace\frac{2}{\lambda_{\min}^2}+2C_S^2,v_{H,n_0}\right\rbrace:=v_{H,n}.
\]
By Lemma \ref{lem:(H1)_regression}, \textbf{(H1a)} is satisfied with $\delta=p/2$. Applying Theorem \ref{theo3} with the constants computed in the previous lemmas and proposition, we get
finally,
\begin{align*}
&\mathbb{E} \left[ \left\| \theta_{n} - \theta \right\|^{2} \right]  \leq  e^{ - \frac{1}{2}c_{\gamma}n^{1-\gamma} }\left(K^{(3)}_{1,\text{lin}}+K^{(3)}_{1',\text{lin}}\max_{0\leq k \leq n} d_k(k+1)^{\gamma}\right) \\
& + n^{-\gamma}\left(2^{3 + \gamma} c_{\gamma} \mathbb{E} \left[ \epsilon^{2} \right]\text{Tr} \left(H^{-1} \right)+\frac{K_{2,\text{lin}}^{(3)}}{n^{\gamma}}+ K_{2',\text{lin}}^{(3)}v_{H,n/2}\right)+ d_{\lfloor n/2\rfloor}.
\end{align*}
with $v_{H,n}$ defined by \eqref{def::vhn}, recalling that $\lambda_{\min}$ and $\lambda_{\max}$ are the smallest and largest eigenvalues of  $\mathbb{E}\left[ XX^{T} \right]$,  and since for the linear case one has $C_{A} = 4c_{\gamma} \frac{\mathbb{E}\left[ \left\| X \right\|^{4} \right]}{\lambda_{\min}^{2}}  \geq 4 c_{\gamma}$,
\begin{align}
\notag K^{(3)}_{1,\text{lin}}& =e^{ 8 \frac{\mathbb{E}\left[ \| X\|^{4} \right]}{\lambda_{\min}^{2}} c_{\gamma}^{3} \frac{2\gamma}{2\gamma -1} } \left( \mathbb{E}\left[ \left\| \theta_{0} - \theta \right\|^{2} \right] + \frac{2 \mathbb{E} \left[ \epsilon^{2} \right] \text{Tr} \left( H^{-1}  \right)}{c_{\gamma}} +    4C_{(2)} \left( \lambda_{\min}^{4} + C_{S}^{4} \right) + \frac{\sigma_{(2)} {v_{H,0}}}{c_{\gamma}} \right),\\
\notag K^{(3)}_{1',\text{lin}}& =\frac{1}{4c_{\gamma}}e^{ 8 \frac{\mathbb{E}\left[ \left\| X  \right\|^{4}  \right]}{\lambda_{\min}^{2}} c_{\gamma}^{3} \frac{2\gamma}{2\gamma -1} ,},\quad 
d_n= 8\lambda_{\max} \sqrt{c_{n,\text{lin}}v_{H,n}} + 8\frac{ C_{(2)}}{\lambda_{\min}^2}c_{n,\text{lin}}, \\
\label{def::const::lm} K_{2,\text{lin}}^{(3)}& =2^{4+2\gamma} C_{(2)}c_{\gamma}\left( \lambda_{\min}^{-4} + C_{S}^{4} \right) c_{\gamma}^{2} ,\quad K_{2',\text{lin}}^{(3)}=2^{2+\gamma}\sigma_{(2)}c_{\gamma},
\end{align}
and $c_{n,\text{lin}}$ and  $C_{S}^{4}$ are respectively defined in  Propositions \ref{prop::lr} and \ref{prop::H2}.
\end{proof}
\begin{proof}[Proof of Theorem \ref{theo::LM_ada}]
Let us first prove that Assumption \textbf{(A6')} is fulfilled. For all $h$,
\[
\mathbb{E}\left[ \nabla_{h}g \left( X , Y , h \right) \nabla_{h} g \left( X , Y , h \right)^{T} \right] = \mathbb{E}\left[ \left( Y  - X^{T}h \right)^{2} XX^{T} \right] = \mathbb{E}\left[ \epsilon^{2} \right] \mathbb{E}\left[ XX^{T} \right] + \mathbb{E}\left[ \left( X^{T}h  - X^{T}\theta \right)^{2}XX^{T} \right]
\] 
and \textbf{(A6')} is satisfied with $\alpha=\mathbb{E}\left[\epsilon^2\right]\lambda_{\min}$, we have by \eqref{eq:CS4_adagrad_bis},
$$\mathbb{E}\left[\Vert A_n\Vert^4\right]\leq \frac{4d\left(1+\sigma_{(4)} + C_{(4)}\frac{4V_{2,ada}^2}{\lambda_{\min}^2}\right)}{\mathbb{E}\left[\epsilon^2\right]^2\lambda_{\min}^2}:=C_{S,ada}^4,$$
with $V_2$ given by Lemma \ref{lem::majvn2} for $p=2$. Then, applying Theorem \ref{theo:adagrad}, 
\begin{align*}
\mathbb{E}\left[ \Vert \theta_n-\theta\Vert^2 \right]  & \leq  {K}_{1,lin}^{ada}\exp \left( -  c_{\gamma}  \lambda_{\min} \lambda_{0,lin}^{ada} n^{1-\gamma} \left(1- {\varepsilon}_{n,lin}^{ada}\right)\right) \\
& +  {K}^{ada}_{2,lin}\left( v_{0,lin}^{ada}\log (n+1)\right)^{\frac{p-1}{p}}n^{-\frac{(p-1)}{p}\min\left\lbrace  \frac{2(1-\gamma)\gamma(\gamma-2\beta)p}{2-\gamma} , 1 \right\rbrace }   + {K}^{ada}_{3,lin}n^{ - \gamma},
\end{align*}
with $\lambda_{0,lin}^{ada} = \left[\frac{4(1-\gamma)p}{2-\gamma}\left(C_{\left( \frac{4p(1-\gamma)}{2-\gamma} \right)}+1\right)\right]^{-\frac{2-\gamma}{	4p(1-\gamma)}}$, and  recalling that $\lambda_{\min}$ and $\lambda_{\max}$ are the smallest and largest eigenvalues of  $\mathbb{E}\left[ XX^{T} \right]$,  
\begin{align}
\varepsilon_{n,lin}^{ada} & =  \frac{2 C_{M,lin}^{ada} n^{-1+(1-\gamma)(2\gamma - \beta)+\gamma}}{ \lambda_{\min} {\lambda}_{0,lin}^{ada}} \left(1+\frac{n^{(1+2\beta-2\gamma)^+}}{\vert 2 \gamma - 2\beta -1\vert}\right),\label{def::epsilon::lm::ada}\\
 K^{ada}_{1,\text{lin}}& =\frac{2}{\lambda_{\min}}\left(\mathbb{E}\left[V_0\right]+\frac{ c_{\gamma}\lambda_{\max}\sigma_{(2)}C_{S,ada}^{2}}{ C_{M,lin}^{ada}} + \frac{4\lambda_{\min}\lambda_{0,lin}^{ada}V_{p,lin}^{ada}}{C_{M,lin}^{ada}}\right),\label{def::constant_1::lm::ada}\\
K^{ada}_{2,\text{lin}}& = \frac{1}{\lambda_{\min}} 2^{p/2+3/2} V_{p,lin}^{ada}  \label{def::constant_2::lm::ada}\\
K^{ada}_{3,\text{lin}}& =  \frac{2^{\gamma  }c_{\gamma}\lambda_{\max}\sigma_{(2)}C_{S,ada}^{2}}{\lambda_{\min}^{2} \lambda_{0,lin}^{ada}} .\label{def::constant_3::lm::ada}
\end{align}
where $v_0=dM(\beta)+\frac{d2^{\frac{2(1-\gamma)}{2-\gamma}p}\left(\sigma_{\left( \frac{4(1-\gamma)}{2-\gamma}p \right)}+2^{\frac{2(1-\gamma)}{2-\gamma}p}C_{\left( \frac{4(1-\gamma)}{2-\gamma}p \right)}\frac{ V_{p,ada }^{\frac{2(1-\gamma)}{2-\gamma}p}}{\lambda_{min}^{\frac{2(1-\gamma)}{2-\gamma}p}}\right)}{\sigma_{\left(  \frac{4(1-\gamma)}{2-\gamma}p \right)}+1}.$

$$C_{M,lin}^{ada}=\max\left\{\frac{C_{(2)}\lambda_{\max}c_{\beta}^{2}c_{\gamma}}{\lambda_{\min}},(\lambda_{\min}\lambda_{0,lin}^{ada})^{\frac{2\gamma-2\beta}{\gamma}}c_{\gamma}^{\frac{\gamma-2\beta}{\gamma}}\right\}$$
and 
\begin{align*}
 V_{p,ada}^{p} & = e^{-p\lambda_{\min}\lambda_{0 } ' c_{\gamma}\left(1+\frac{1+\left(\frac{c_{\gamma}c_{\beta}^2 a_{p,lin}^{ada}}{ p \lambda_{\min} \lambda_{0 }'}\right)^{\frac{1-\gamma-\lambda ' }{\gamma-2\beta-\lambda ' }}}{1-\gamma-\lambda ' }\right)+c_{\gamma}^2c_{\beta}^2a_p\left(1+\frac{1+\left(\frac{c_{\gamma}c_{\beta}^2 a_{p,lin}^{ada}}{p\lambda_{\min} \lambda_{0}'}\right)^{\frac{1-2\gamma+2\beta}{\gamma-2\beta-\lambda '}}}{1-2\gamma+2\beta}\right)} 
\end{align*}
where 
\begin{align}
& \notag a_{p,lin}^{ada}    = p\left( \frac{C_{(2)}}{\lambda_{\min}} + \frac{\sigma_{(2)}}{2} \right) + 2^{p-2}(p-1)p \lambda_{\max}^{2} \left( c_{\gamma}^{2}c_{\beta}^{2} \left( \sigma_{(4)} + \frac{4C_{(4)}}{\lambda_{\min}^{2}} \right) + \frac{2\sigma_{(2)}}{\mu} + \frac{4C_{(2)}}{\lambda_{\min}^{2} } \right) \\
\label{def::ap::ada::lin} &  + 2^{p-2}(p-1)p\lambda_{\max}^{p} \left( c_{\gamma}^{2p-2}c_{\beta}^{2p-2} \left( \sigma_{(2p)} + \frac{2^{p}C_{(2p)}}{\lambda_{\min}^{2}} \right)  + c_{\gamma}^{p-2}c_{\beta}^{p-2} \left( \frac{1}{2}\sigma_{(2p)} + \frac{2p}{\lambda_{\min}^{2}} \left( \frac{1}{2} + \sqrt{C_{(2p)}} \right) \right) \right) ,
\end{align}
and
\begin{equation}
\label{def::a2::lin::ada} a_{2,lin}^{ada} = \sigma_{(2)} + \frac{2C_{(2)}}{\lambda_{\min}} + \frac{4\lambda_{\max}^{2}}{\lambda_{\min}}\sigma_{(2)} + \frac{8\lambda_{\max}^{2}C_{(2)}}{\lambda_{\min}^{2}} + 2\lambda_{\max}^{2}\sigma_{(4)}c_{\gamma}^{2}c_{\beta}^{2} + \frac{8\lambda_{\max}^{2}C_{(4)}}{\lambda_{\min}^{2}}c_{\gamma}^{2}c_{\beta}^{2}
\end{equation}
\end{proof}
\subsection{Proof of Theorem \ref{theo::glm}}

The proof relies on the verification of each Assumption in Theorem \ref{theo3}.

\noindent\textbf{Verifying Assumptions \textbf{(A1)}, \textbf{(A1')} to \textbf{(A6)}. }
First, remark that taking for all $0 \leq a \leq 2p$, one has
\begin{align}
\notag \mathbb{E}  \left[ \left\| \nabla_{h}l \left( Y , X^{T}h \right)  X + \sigma h \right\|^{a} \right] & \leq  2^{a-1} \mathbb{E}\left[ \left\| \nabla_{h}l \left( Y , X^{T}\theta_{\sigma} \right)X + \sigma \theta_{\sigma} \right\|^{a}  \right] \\
\notag &  + 2^{a-1} \mathbb{E}\left[ \left\| \nabla_{h}l \left( Y , X^{T}h \right)X - \nabla_{h} \ell \left( Y , X^{T} \theta_{\sigma} \right)X + \sigma \left( h- \theta_{\sigma}  \right) \right\|^{a}  \right] \\
\label{def::GLMa} & \leq 2^{a-1 }  L_{\sigma}^{a}  + 2^{a-1} \underbrace{\mathbb{E}\left[ \left( L_{\nabla l}\left\| X \right\| + \sigma \right)^{a} \right]}_{=: C_{\text{GLM}}^{(a)}}  \left\| h - \theta_{\sigma} \right\|^{a}
\end{align}
and Assumption \textbf{(A1)} is so verified. In a same way,
\[
\mathbb{E} \left[ \left\|  \left( \nabla_{h} g \left( X , h \right) - \nabla_{h} g \left( X , \theta_{\sigma} \right) \right) \right\|^{2} \right] \leq \mathbb{E}\left[ \left( L_{\nabla l}\| X\| + \sigma \right)^{2} \right]\left\| h - \theta_{\sigma} \right\|^{2} \leq C_{\text{GLM}}^{(2)} \left\| h - \theta_{\sigma} \right\|^{2}
\]
and \textbf{(A1')} is so verified. Remark that \textbf{(A2)} and \textbf{(A4)} are verified by hypothesis with $\mu=\sigma$, while for \textbf{(A3)}, one has
\begin{equation}\label{def::CGLM}
\left\| \mathbb{E} \left[ \nabla_{h}^{2} \ell \left( Y , X^{T}h \right)XX^{T} + \sigma I_{d} \right] \right\|_{op} \leq  L_{\nabla l} \mathbb{E}\left[ \left\| X \right\|^{2} \right] + \sigma =:C_{\text{GLM}}.
\end{equation}
Observe that Assumption \textbf{(A5)} is given by \textbf{(GLM1)} while
for Assumption \textbf{(A6)}, \textbf{(GLM3)} together \eqref{eq:adagrad_glm_Var=square}, which yields 
$$\mathbb{E}\left[  \left(\nabla_{h}g(X,\theta_v) \right)_k^{2}  \right]=\mathbb{E}\left[ \left\vert \nabla_{h}l \left( Y , X^{T}\theta_{\sigma} \right)X_k + \sigma (\theta_{\sigma})_k \right\vert^{2}  \right]>\alpha_{\sigma}$$
for all $1\leq k\leq d$. 
\medskip

\noindent\textbf{Verifying Assumption \textbf{(H1)}.}
The following lemma ensures that Assumption \textbf{(H1)} is fulfilled.
\begin{lem}\label{lem::GLM}
Assume first \eqref{upperbound_glm} and that $X$ admits a moment of order $2p$ for some $p<0$. In the regularized case defined by \eqref{eq:regularized_equation}, denoting $\lambda_0=\frac{1}{2L_{\nabla l}\mathbb{E}\left[ \left\|  X\right\|^2\right] +2\lambda}$, we have 
$$\mathbb{P}\left[ \lambda_{\min}\left(\overline{S}_{n}^{-1} \right)<\lambda_0\right]\leq v_n$$
with 
$$v_n=\frac{2^{p-1}}{\left(L_{\nabla l}\mathbb{E}\left[ \left\|  X\right\|^2 \right]+\sigma\right)^p}\left(n^{-p} \left\| S_{0} \right\|^{p}+ C_1(p)n^{1-p}\mathbb{E}\left[  \left|  T\right|^p\right]+C_2(p)n^{-p/2}\left(\mathbb{E}\left[ \left| T\right|^2 \right] \right)^{p/2}\right),$$
where $T=L_{\nabla l}\left( \left\|  X \right\|^2-\mathbb{E}\left[ \left\|  X\right\|^2 \right] \right)+\sigma\left(\left\|  Z\right\|^2-1\right)$ and $Z$ being a standard $d$-dimensional random variable independent of $X$. In addition, $C_{1}(p)$ and $C_{2}(p)$ are given in \cite{Pinelis}.
\end{lem}
The proof is given in Appendix \ref{sec::proof::tec::lm}. Observe that if $p > 4 \gamma$, one has $v_{n} = o \left( \gamma_{n} \right)$.

\noindent\textbf{Verifying Assumption (H2). } The following proposition ensures that \textbf{(H2)} is fulfilled.

\begin{prop}\label{lem::GLM::H2}
Considering from the regularized problem given by \eqref{eq:regularized_equation}, one has for all $n \geq 0$,
\begin{align*}
\left\|  \bar{S}_{n}^{-1}\right\| \leq 2d\max\left\lbrace\frac{1}{\sigma},\left\|  S_0^{-1}\right\|\right\rbrace =: C_{S,\sigma}
\end{align*}

\end{prop}

\begin{rmq}
Remark that if \eqref{lowerbound_glm} holds for some constant $\alpha > 0$ and if $\mathbb{E}\left[ XX^{T} \right]$ is positive, under hypothesis of Proposition \ref{Mendelson_result}, for all $n \geq 0$ and for $\sigma =0$, one has
\begin{align*} & \mathbb{E}\left[\Vert \bar{S}_{n}^{-1}\Vert^{2}\right]\leq \frac{1}{\alpha^{2}} \max \left\lbrace 2 c_{\beta}^{2} \left( \frac{2\beta}{ec_{3}} \right)^{2\beta} + c_{2}^{-2} , \left( \left( c_{1}d+1 \right) \left\| S_{0}^{-1} \right\| \right)^{2} \right\rbrace \leq   C_{S,0}^2 ,  \\
& \mathbb{E}\left[\Vert \bar{S}_{n}^{-1}\Vert^{4}\right]\leq \frac{1}{\alpha^{4}} \max \left\lbrace 2 c_{\beta}^{4} \left( \frac{2\beta}{ec_{3}} \right)^{4\beta} + c_{2}^{-4} , \left( \left( c_{1}d+1 \right) \left\| S_{0}^{-1} \right\| \right)^{4} \right\rbrace \leq   C_{S,0}^4 
\end{align*} 
with $C_{S,0}^{4} = \frac{1}{\alpha^{4}} \max \left\lbrace \left( 2 c_{\beta}^{2} \left( \frac{2\beta}{ec_{3}} \right)^{2\beta} + c_{2}^{-2} \right)^{2} , \left( \left( c_{1}d+1 \right) \left\| S_{0}^{-1} \right\| \right)^{4} \right\rbrace$.
\end{rmq}

\noindent\textbf{A first result}

%\textcolor{red}{On a les trucs suivants: $c_{\beta} = C_{S,\sigma} = 2d \max \left\lbrace \frac{1}{\sigma}, \left\| S_{0}^{-1} \right\| \right\rbrace$, $\beta =0$, $C_{1} = 2L_{\sigma}^{2}, C_{1}' = 8L_{\sigma}^{4}$, $C_{2} = 2\left( L_{\nabla l} + \sigma \right)^{2} = 2 C_{GLM}^{(2)}$, $C_{2}' = 8\left( L_{\nabla l} + \sigma \right)^{4} = 8C_{GLM}^{(4)}$, $\mu = \sigma$, $\lambda_{0} = \frac{1}{2L_{\nabla l}\mathbb{E}\left[ \left\| X \right\|^{2} \right] + 2 \sigma} = \frac{1}{2C{GLM}}$, $L_{\nabla G} = L_{\nabla l} \mathbb{E}\left[ \| X \|^{2} \right] + \sigma = C_{GLM} $. }

Remark that one can rewrite Proposition \ref{prop::ordre4} as follows:
\begin{prop}\label{prop::glm}
Suppose there exist $p>2$ such that $X$ admits a $2p$-th order moment and that there is $L_{\sigma}$ verifying
\begin{equation}\label{majllambda}
\mathbb{E}\left[ \left| \nabla_{h}l \left( Y , X^{T}\theta_{\sigma} \right) \right|^{p} \left\|  X \right\|^{p} \right] +  \sigma \theta_{\sigma} \leq L_{\sigma}^{p} .
\end{equation}
Then, 
\begin{align*}
\mathbb{E}\left[ V_{n}^{2} \right]  \leq &\exp \left( - \frac{3 c_\gamma \sigma}{4 C_{\text{GLM}}}n^{1-\gamma} \right)\left(K^{(2')}_{1,\text{GLM}}+K^{(2')}_{1',\text{GLM}}\max_{1\leq k\leq n+1}v_k^{\frac{p-2}{p}}k^{\gamma-\frac{p-2}{p}\delta}\right) \\
 &\hspace{4cm}+K^{(2')}_{2,\text{GLM}}n^{-2\gamma} + K^{(2')}_{3,\text{GLM}}v_{\lfloor n/2\rfloor}^{(p-2)/p}n^{-\delta(p-2)/p} =: v_{n,\text{GLM}} ,
 \end{align*}
with $v_{n}$ defined in Lemma \ref{lem::GLM}, $C_{S,\sigma}$ defined in Lemma \ref{lem::GLM::H2},  $C_{\text{GLM}} $ and $C_{\text{GLM}}^{(a)}$ defined in equations \eqref{def::CGLM} and \eqref{def::GLMa},
\begin{align*}
a_{1,\text{GLM}} & = C_{S,\sigma}^{4}C_{\text{GLM}}^{2}\left( \frac{ 64 L_{\sigma}^{4}C_{\text{GLM}}^{5}}{\sigma^{3}} + 4 c_{\gamma}L_{\sigma}^{4}+   \frac{4 L_{\sigma}^{4}C_{\text{GLM}}}{  \sigma} \right) \\
a_{M,\text{GLM}} & =  \max \left\lbrace  \left(  \frac{4 C_{\text{GLM}}C_{\text{GLM}}^{(2)}}{ \sigma} +\frac{2C_{\text{GLM}}^{2} }{\sigma^{2}} \left(8 C_{\text{GLM}}^{(2)} + 8 C_{\text{GLM}}^{(4)}c_{\gamma}^{2}C_{S,\sigma}^{2} \right)   \right) c_{\gamma}C_{S,\sigma}^{2}  , \left( \frac{3\sigma}{4C_{\text{GLM}}} \right)^{2}c_{\gamma} \right\rbrace \\
K^{(2')}_{1,\text{GLM}}& = \exp \left( 2 a_{M,\text{GLM}}\frac{2 \gamma }{2\gamma   -1} \right) \left( \mathbb{E}\left[ V_{0}^{2} \right] + \frac{2a_{1,\text{GLM}}c_{\gamma}^{2}}{a_{M,\text{GLM}}}\right) \\
 K^{(2')}_{1',\text{GLM}}& =\exp \left( 2 a_{M,\text{GLM}}\frac{2 \gamma }{2\gamma   -1} \right) \cdot \frac{4 \sigma V_{p,\text{GLM}}^{2 }}{a_{M,\text{GLM}}C_{\text{GLM}}} \\
 K^{(2')}_{2,\text{GLM}}& = \frac{2^{2\gamma + 1}a_{1,\text{GLM}}C_{\text{GLM}}c_{\gamma}^{2}}{3 \sigma} \\
 K^{(2')}_{3,\text{GLM}}& =\frac{2^{2+(p-2)\delta/p}}{3}V_{p,\text{GLM}}^{2 },
\end{align*}
with $V_{p,GLM}^{p}  =e^{ a_{p,GLM}c_{\gamma}^{2}C_{S,\sigma}^{2} \frac{2\gamma}{2\gamma-1}  } \max\left\lbrace 1 , \mathbb{E}\left[ V_{0}^{p} \right]  \right\rbrace$ where
\scriptsize{
\begin{align}
\notag    a_{p,\text{GLM}} : &  = p\left( \frac{2C_{\text{GLM}}^{(2)}}{\sigma} +  L_{\sigma}^{2} \right) + 2^{p-2}(p-1)p C_{\text{GLM}}^{2} \left( c_{\gamma}^{2}C_{S,\sigma}^{2} \left( 8L_{\sigma}^{4} + \frac{32 C_{\text{GLM}}^{(4)}}{\sigma^{2}} \right) + \frac{4L_{\sigma}^{2}}{\sigma} + \frac{8C_{\text{GLM}}^{(2)}}{\sigma^{2}} \right) \\
\notag  &  + 2^{p-2}(p-1)pC_{\text{GLM}}^{p} \left( c_{\gamma}^{2p-2}C_{S,\sigma}^{2p-2} \left( 2^{2p-1} L_{\sigma}^{2p} + \frac{2^{3p-1}C_{\text{GLM}}^{(2p)}}{\sigma^{2}} \right)   + c_{\gamma}^{p-2}C_{S,\sigma}^{p-2} \left( 2^{2p-2}L_{\sigma}^{2p} + \frac{2p}{\sigma^{2}} \left( \frac{1}{2} + 2^{p-1/2}\sqrt{C_{\text{GLM}}^{(2p)}} \right) \right) \right) .
\end{align}}
 \end{prop}

\noindent\textbf{Verifying Assumption (H3). }
We prove here that \textbf{(H3)} holds for general linear models. We now denote
\[
H_{\sigma} =: \mathbb{E}\left[ \nabla_{h}^{2}\ell \left( Y , \theta_{\sigma}^{T}X \right)XX^{T} \right] + \sigma I_{d} .
\]
\begin{prop}\label{Hypothesis (H3)}
Suppose Assumptions \textbf{(GLM1)} and \textbf{(GLM2)} hold, then for all $n \geq 0$, 
$$\mathbb{E} \left[ \left\|  \bar{S}_{n}^{-1}-H_{\sigma}^{-1}\right\|^2 \right] \leq \frac{4C_{S,\sigma}^{2}}{\sigma^{2}n}\left( L_{\nabla l}^2\mathbb{E}\left[ \left\| X \right\|^{4} \right]+\frac{L_{\nabla^{2}L}^2}{\sigma   }\sum_{i=0}^{n-1}v_{i,\text{GLM}}+\frac{1}{n }\Vert S_0-H\left(\theta_{\sigma} \right)\Vert^{2}\right) + \frac{16d^{4}C_{S,\sigma}^{2}}{n^{2}} =: v_{\ell,n}$$ 
with $v_{i,\text{GLM}}$ defined in Proposition \ref{prop::glm}.

\end{prop}

We can now finish the proof of Theorem \ref{theo::glm}. In this aim, let us first remark that for all $h,h'$,
\begin{align*}
\mathbb{E}\left[ \left\| \nabla_{h} \ell \left( y , X^{T} h \right)X + \sigma h - \nabla_{h} \ell \left( y , X^{T} h' \right)X - \sigma h' \right\|^{2} \right] \leq  2 \left(  L_{\nabla l}^{2} \mathbb{E}\left[ \left\| X \right\|^{2} \right] + \sigma^{2} \right) \left\| h-h' \right\|^{2}.
\end{align*}
Then, with the help of Theorem \ref{theo3}, one has
\begin{align*}
&\mathbb{E} \left[ \left\| \theta_{n} - \theta_{\sigma} \right\|^{2} \right]  \leq  e^{ - \frac{1}{2}c_{\gamma}n^{1-\gamma} }\left(K^{(3)}_{1,\text{GLM}}+K^{(3)}_{1',\text{GLM}}\max_{0\leq k \leq n} (k+1)^{\gamma}d_{k,\text{GLM}}\right) \\
&\hspace{2cm} + n^{-\gamma}\left(2^{3 + \gamma} c_{\gamma}\text{Tr} \left(H_{\sigma}^{-1}\Sigma_{\sigma} H_{\sigma}^{-1} \right)+\frac{K_{2,\text{GLM}}^{(3)}}{n^{\gamma}}+ K_{2',\text{GLM}}^{(3)}v_{l,n/2}\right)+ d_{\lfloor n/2\rfloor,\text{GLM}},
\end{align*}
with $\Sigma_{\sigma} := \mathbb{E}\left[ \left( \nabla_{h} \ell \left( y , X^{T} \theta_{\sigma} \right)X + \sigma \theta_{\sigma} \right)\left(\nabla_{h} \ell\left( y , X^{T} \theta_{\sigma} \right)X + \sigma \theta_{\sigma} \right)^{T} \right]$ and since  $ c_\gamma  4 \frac{C_{\text{GLM}}^{(2)}  }{\sigma^{2}} \geq C_{A,\text{GLM}}=  \geq 4c_{\gamma}$,
\begin{align}
K^{(3)}_{1,\text{GLM}}=&e^{ 8 \frac{C_{\text{GLM}}^{(2)}}{\sigma^{2}} c_{\gamma}^{3} \frac{2\gamma}{2\gamma -1} } \left( \mathbb{E}\left[ \left\| \theta_{0} - \theta_{\sigma} \right\|^{2} \right] + \frac{2 \text{Tr} \left( H_{\sigma}^{-1}\Sigma_{\sigma} H_{\sigma}^{-1} \right)}{c_{\gamma}} +   8\sigma^{2} \left( \sigma^{-4} + C_{S,\sigma}^{4} \right)  + \frac{2L_{\sigma}^{2} {v_{l,0}}}{c_{\gamma}} \right),\label{eq:theo3_constant_1_GLM}\\
K^{(3)}_{1',\text{GLM}}=&\frac{1}{4c_{\gamma}}e^{ 8 \frac{C_{\text{GLM}}^{(2)}}{\sigma^{2}} c_{\gamma}^{3}\frac{2\gamma}{2\gamma -1} ,},\quad 
d_{n,\text{GLM}}= 8C_{\text{GLM}} \sqrt{v_{n,\text{GLM}}v_{l,n}} + 8\frac{L_{\nabla^{2}L}^{2}\sigma^{-2} + 2C_{\text{GLM}}^{(2)}}{\sigma^2}v_{n,\text{GLM}},\label{eq:theo3_constant_2_GLM} \\ 
\label{eq:theo3_constant_3_GLM}
K_{2,\text{GLM}}^{(3)}   = & 2^{5+2\gamma} C_{\text{GLM}}^{(2)}c_{\gamma}\left( \sigma^{-4} + C_{S,\sigma}^{4} \right) c_{\gamma}^{2} ,\quad K_{2',\text{GLM}}^{(3)}=2^{3+\gamma}L_{\sigma}^{2}c_{\gamma}.
\end{align}
\begin{proof}[Proof of Theorem \ref{theo::GLM_ada}]
The proof follows exactly the same pattern as the proof of Theorem \ref{theo::LM_ada}, using Assumption \textbf{(A6)} together with Lemma \ref{lem:H2_adagrad} to compute the constant $C_S$ such that \textbf{(H2)} is satisfied. 
\end{proof}
\begin{appendix}

\section{Proofs of technical proposition}
\subsection{Proof of Proposition \ref{prop::ordrep'}}

Let us recall that 
\[
V_{n+1} = V_{n}  \underbrace{ - \gamma_{n+1}\left( g_{n+1}' \right)^{T}A_{n} \int_{0}^{1}  \nabla G \left( \theta_{n}  + t \left( \theta_{n+1} - \theta_{n} \right) \right) dt }_{=: U_{n+1}}
\]

Remark that for $a\geq 2$ and $x,h\in \mathbb{R}$ such that $x\geq 0$ and $x+h\geq 0$, we have by Taylor's expansion
\begin{equation}\label{eq:simple_inequality}
(x+h)^{a}\leq x^a+ax^{a-1}h+2^{p-2}a(a-1) (x^{a-2}\vert h\vert^{2}+\vert h\vert^a).
\end{equation}
This yields for $a=p'$, $x=V_n$ and $h=U_{n+1}$ and after conditioning on $\mathcal{F}_n$
\begin{align}
 \mathbb{E}\left[  V_{n+1}^{p'} |\mathcal{F}_{n} \right]  & \leq  V_{n}^{p'} + p' V_{n}^{p'-1}\mathbb{E}\left[  U_{n+1} |\mathcal{F}_{n} \right]\nonumber\\
 &+ 2^{p'-2}p'(p'-1)\left(  \mathbb{E}\left[   \vert U_{n+1}\vert^{2} | \mathcal{F}_{n} \right]V_n^{p'-2}+   \mathbb{E}\left[   \vert U_{n+1}\vert^{p'} | \mathcal{F}_{n} \right]\right).\label{eq:prop3.1_master_equation}
 \end{align}
Since $G$ is convex and $\nabla G $ is Lipschitz,
 \begin{align*}
 \mathbb{E}\left[ U_{n+1} V_{n}^{p'-1} |\mathcal{F}_{n} \right] & \leq -\mathbb{E}\left[ \gamma_{n+1} \left( g_{n+1}' \right)^{T} A_{n} \int_{0}^{1}   \nabla G \left( \theta_{n} \right)dt   |\mathcal{F}_{n} \right] V_{n}^{p'-1} \\
 &+\mathbb{E}\left[ \gamma_{n+1} \left( g_{n+1}' \right)^{T} A_{n} \int_{0}^{1} \left( \nabla G \left( \theta_{n} \right) - \nabla G \left( \theta_{n} + t \left( \theta_{n+1} - \theta_{n} \right) \right) \right) dt   |\mathcal{F}_{n} \right] V_{n}^{p'-1} \\
 &  \leq -\gamma_{n+1}\nabla G \left( \theta_{n} \right)^{T} A_{n} \nabla G \left( \theta_{n} \right)V_n^{p'-1}+\frac{L_{\nabla G}}{2}\gamma_{n+1}^{2} \mathbb{E}\left[  \left\|  g_{n+1}' \right\|^{2} |\mathcal{F}_{n} \right] \left\| A_{n} \right\|^{2} V_{n}^{p'-1} .\\
 & \leq -\gamma_{n+1}\nabla G \left( \theta_{n} \right)^{T} A_{n} \nabla G \left( \theta_{n} \right)V_n^{p'-1}+\gamma_{n+1}^{2}\beta_{n+1}^{2} \frac{L_{\nabla G}}{2} \left(  C_{1}  V_{n}^{p'-1} + \frac{2C_{2}}{\mu} V_{n}^{p'} \right) .
 \end{align*} 
By strong convexity, we have
\begin{align*}
\notag \nabla G \left( \theta_{n} \right)^{T} A_{n} \nabla G \left( \theta_{n} \right) V_n^{p'-1}& \geq \lambda_{\min} \left( A_{n} \right) \left\| \nabla G \left( \theta_{n} \right) \right\|^{2} V_n^{p'-1}\\
\notag & \geq  2\lambda_{n}\mu V_{n}^{p'} \mathbf{1}_{\lambda_{\min} \left( A_{n} \right) \geq \lambda_{n}} \\
 & =  2\lambda_{n}\mu V_{n}^{p'} -  2\mathbf{1}_{\lambda_{\min} \left( A_{n} \right) < \lambda_{n}} \lambda_{n}\mu V_{n}^{p'}, 
\end{align*}
where $\lambda_n=\lambda_0(n+1)^{\lambda}$ with $0\leq \lambda< \min \lbrace\gamma-2\beta,1-\gamma\rbrace$. Applying Hölder inequality yields then
\begin{align*}
\mathbb{E}\left[ \nabla G \left( \theta_{n} \right)^{T}A_{n} \nabla G \left( \theta_{n} \right) \right] \geq& 2 \lambda_{n} \mu \mathbb{E}\left[  V_{n}^{p'} \right] - 2 \lambda_{n} \mu \mathbb{E}[V_n^p]^{p'/p} \left( \mathbb{P}\left[\lambda_{\min}(A_n)<\lambda_n \right] \right)^{\frac{p-p'}{p}}\\
\geq&  2\lambda_{n} \mu \mathbb{E}\left[  V_{n}^{p'} \right] - 2\lambda_{n} \mu V_p^{p'} \left( \mathbb{P}\left[ \lambda_{\min}(A_n)<\lambda_n \right] \right)^{\frac{p-p'}{p}}, 
\end{align*}
with $V_p^{p} \geq \sup_{n\geq 0}\mathbb{E}[V_n^p]$ given by Lemma \ref{lem::majvn2}. Then, Assumption \textbf{(H1a)} gives $\mathbb{P}\left[ \lambda_{\min} \left( A_{n} \right)<\lambda_n\right]\leq v_{n+1}(n+1)^{-\delta-q\lambda}:=\bar{v}_n$, so that finally
 \begin{align}
& \mathbb{E}\left[ U_{n+1} V_{n}^{p'-1} \right]  \nonumber\\
 \leq&  - 2 \gamma_{n+1}\lambda_{n} \mathbb{E}\left[\mu  V_{n}^{p'} \right] + 2 \lambda_{n} \gamma_{n+1}\mu V_p^{p'} \bar{v}_n^{\frac{p-p'}{p}}+\gamma_{n+1}^{2}\beta_{n+1}^{2} \frac{L_{\nabla G}}{2} \left(  C_{1} \mathbb{E}\left[ V_{n}^{p'-1}\right] + \frac{2C_{2}}{\mu}\mathbb{E}\left[ V_{n}^{p'}\right] \right).\label{eq:prop_3.1_first_term}
 \end{align} 
Furthermore, since $\nabla G$ is $L_{\nabla G}$-Lipschitz, one has
\begin{align}
\notag \left\| \int_{0}^{1}  \nabla G \left( \theta_{n} + t \left( \theta_{n+1} - \theta_{n} \right) \right) dt \right\|  & \leq L_{\nabla G} \int_{0}^{1} \left(  \left\| \theta_{n} - \theta \right\| + t \left\| \theta_{n+1} - \theta_{n} \right\| \right) dt \\
\label{maj::rest::taylor} & \leq L_{\nabla G}\left(  \left\| \theta_{n} - \theta \right\| + \frac{1}{2}\gamma_{n+1} \left\| A_{n} \right\| \left\|  g_{n+1}' \right\|\right) .
\end{align} 
Hence, using \textbf{(H1b)} and the strong convexity of $G$ yields
\begin{align*}
\mathbb{E}\left[ \left\vert U_{n+1} \right\vert^{p'} |\mathcal{F}_{n} \right] &\leq     L_{\nabla G }^{p'} \left\| A_{n} \right\|^{p'} \gamma_{n+1}^{p'}\mathbb{E}\left[  \left\| g_{n+1}' \right\|^{p'} \left(2^{p'-1} \left\| \theta_{n} - \theta \right\|^{p'} + 2^{-1} \gamma_{n+1}^{p'}\left\| A_{n} \right\|^{p'} \left\| g_{n+1}' \right\|^{p'} \right) |\mathcal{F}_{n} \right] \\
& \leq \frac{L_{\nabla G}^{p'}}{2}\gamma_{n+1}^{p'} \beta_{n+1}^{p'}  \bigg( 2^{p'}\left(C_1^{(p'/2)}\frac{2^{p'/2}V_{n}^{p'/2}}{\mu^{p'/2}}+C_2^{(p'/2)}\frac{2^{p'}V_n^{p'}}{\mu^{p'}}\right)  \\
&\hspace{6cm}+ \gamma_{n+1}^{p'}\beta_{n+1}^{p'}\left(C_{1}^{(p')}+ C_{2}^{(p')} \frac{2^{p'}V_n^{p'}}{\mu^{p'}}\right)\bigg) \\
\end{align*}
%%& \leq\frac{L_{\nabla G}^{p'}}{2}\gamma_{n+1}^{p'} \beta_{n+1}^{p'} \Bigg( \left( \frac{2^{3(p'-1)}L_{\nabla_G}^{p'}p'}{\lambda_0\mu^{p'+1}}\gamma_{n+1}^{p'-2} + c_{\gamma}^{p'-1} \right)C_{1}^{(p')} \gamma_{n+1}\\
%&\hspace{6cm}+ \left( 2^{p'}\sqrt{C_{2}^{(p')}} + C_{2}^{(p')}  c_{\gamma}^{p'} +\frac{\lambda_0 \mu^{p'+1}}{p'2^{p'-1}L_{\nabla_G}^{p'}\gamma_{n+1}^{p'-1}} \right) \left\| \theta_{n} - \theta \right\|^{2p'} \Bigg)\\
%& \leq\frac{\lambda_0\mu\gamma_n}{p'}V_n^{p'}+\frac{L_{\nabla G}^{p'}}{2}\gamma_{n+1}^{p'} \beta_{n+1}^{p'} \left( \left( 2^{p'-1}  \gamma_{n+1}^{p'-2}+ c_{\gamma}^{p'-1} \right)C_{1}^{(p')}\gamma_{n+1}+ \frac{\left( 2^{2p'}\sqrt{C_{2}^{(p')}} + C_{2}^{(p')}  c_{\gamma}^{p'} 2^{p'}\right)}{\mu^{p'}} V_n^{p'} \right)
Specializing the latter inequality with $p'=2$ yields then (recalling inequalities \eqref{eq:def_C1})
\begin{align*}
&\mathbb{E}\left[ \left\vert U_{n+1} \right\vert^{2} |\mathcal{F}_{n} \right]V_n^{p'-2} \\
& \leq\frac{L_{\nabla G}^{2}}{2}\gamma_{n+1}^{2} \beta_{n+1}^{2}  \bigg( 2^{p'}\left(C_1\frac{2V_{n}}{\mu}+C_2\frac{2^2V_n^{2}}{\mu^{2}}\right)  + \gamma_{n+1}^{2}\beta_{n+1}^{2}\left(C'_{1}+ C'_{2}\frac{4V_n^{2}}{\mu^{2}}\right)\bigg)V_n^{p'-2},
\end{align*}
so that 
\begin{align*}
&\mathbb{E}\left[ \left\vert U_{n+1} \right\vert^{p'} |\mathcal{F}_{n} \right]+\mathbb{E}\left[ \left\vert U_{n+1} \right\vert^{2} |\mathcal{F}_{n} \right]V_n^{p'-2}\\
\leq & \frac{L_{\nabla G}^{p'}C_{1}^{(p')}}{2} \gamma_{n+1}^{2p'} \beta_{n+1}^{2p'}+ \frac{2^{3p'/2-1}L_{\nabla G}^{p'}C_1^{(p'/2)}}{\mu^{p'/2}}\gamma_{n+1}^{p'} \beta_{n+1}^{p'} V_{n}^{p'/2}+\frac{2^{p'}L_{\nabla G}^{2}C_1}{\mu}\gamma_{n+1}^{2} \beta_{n+1}^{2}  V_{n}^{p'-1}
\\
&+\frac{L_{\nabla G}^{2}C'_{1}}{2}\gamma_{n+1}^{4} \beta_{n+1}^{4} V_n^{p'-2}+V_n^{p'}\Bigg(\frac{2^{2p'-1}L_{\nabla G}^{p'}C_2^{(p'/2)}}{\mu^{p'}}\gamma_{n+1}^{p'} \beta_{n+1}^{p'} + \frac{2^{p'-1}L_{\nabla G}^{p'}C_{2}^{(p')}}{\mu^{p'}}\gamma_{n+1}^{2p'}\beta_{n+1}^{2p'}\\
&\hspace{4cm}+\frac{2^{p'+1}C_2L_{\nabla G}^{2}}{\mu^{2}}\gamma_{n+1}^{2} \beta_{n+1}^{2}   +\frac{2C'_{2}L_{\nabla G}^{2}}{\mu^{2}}\gamma_{n+1}^{4} \beta_{n+1}^{4}  \Bigg).
\end{align*}
Using the latter inequality with \eqref{eq:prop_3.1_first_term} in \eqref{eq:prop3.1_master_equation} yields then
\begin{align*}
\mathbb{E}\left[V_{n+1}^{p'}\right]\leq &\mathbb{E}\left[V^{p'}_{n}\right] -  2p'\mu \gamma_{n+1}\lambda_{n} \mathbb{E}\left[ V_{n}^{p'} \right] + 2 p'\lambda_{n} \gamma_{n+1}\mu V_p^{p'} \bar{v}_n^{\frac{p-p'}{p}}+ \mathbb{E}\left[P\left(\gamma_{n+1}^2\beta_{n+1}^2,V_n\right)\right]
\end{align*}
with $P(x,y)=A_0x^{p'}+A_{p'/2}x^{p'/2}y^{p'/2}+A_{p'-1}xy^{p'-1}+A_{p'-2}x^2y^{p'-2}+A_{p'}xy^{p'},$
where 
$$A_0=2^{p'-3}p'(p'-1)L_{\nabla G}^{p'}C_{1}^{(p')},\, A_{p'/2}=\frac{2^{5p'/2-3}p'(p'-1)L_{\nabla G}^{p'}C_1^{(p'/2)}}{\mu^{p'/2}},$$
$$A_{p'-1}=p'\frac{L_{\nabla G}}{2}+\frac{2^{2p'-2}p'(p'-1)L_{\nabla G}^{2}C_1}{\mu},\,A_{p'-2}=2^{p'-3}p'(p'-1)L_{\nabla G}^{2}C'_{1},$$
and 
\begin{align*}
A_{p'}=\frac{p'L_{\nabla G}C_2}{\mu}+p'(p'-1)\Bigg(\frac{2^{3p'-3}L_{\nabla G}^{p'}C_2^{(p'/2)}}{\mu^{p'}}c_{\gamma}^{p'-2} c_{\beta}^{p'-2}& + \frac{2^{2p'-3}L_{\nabla G}^{p'}C_{2}^{(p')}}{\mu^{p'}}c_{\gamma}^{2p'-2}c_{\beta}^{2p'-2}\\
&+\frac{2^{2p'-1}C_2L_{\nabla G}^{2}}{\mu^{2}}+\frac{2^{p'-1}C'_{2}L_{\nabla G}^{2}}{\mu^{2}}c_{\gamma}^{2} c_{\beta}^{2}\Bigg).
\end{align*}
Applying now Young's inequality, which implies  $a^ib^{p'-i}\leq \frac{ia^{p'}}{p'}+\frac{(p'-i)b^{p'}}{p'}$ for $0<i<p'$ and $a,b\geq 0$, yields for any $t>0$ and $i\in\{1,2,p'/2\}$
$$A_{p-i}x^iy^{p'-i}=\left(\frac{A_{i}^{1/i}x}{(t\lambda_n\gamma_n)^{\frac{p'-i}{p'}}}\right)^i\left((t\lambda_n\gamma_n)^{\frac{i}{p'}}y\right)^{p'-i}\leq \frac{iA_i^{\frac{p'}{i}}x^{p'}}{(t\lambda_n\gamma_n)^{\frac{p'-i}{i}}}+\frac{(p'-i)t\lambda_n\gamma_ny^{p'}}{p'},$$
so that using the latter inequality with $t=\frac{p'^2 \mu}{3(p'-i)\mu}$ for $i\in\{1,2,p'/2\}$ and using that 
\begin{align*}
\frac{\gamma_{n+1}^{2p'}\beta_n^{2p'}}{(\gamma_{n+1}\lambda)^{\frac{p'}{i}-1}} & =(\gamma_{n+1}\lambda_n) c_{\gamma}^{\frac{2i-1}{i}p'}c_{\beta}^{2p'}\lambda_0^{-\frac{p'}{i}}(n+1)^{-\frac{(2i-1)p'}{i}\gamma+2p'\beta+\frac{p'}{i}\lambda} \\
& \leq (\gamma_{n+1}\lambda_n) c_{\gamma}^{\frac{2i-1}{i}p'}c_{\beta}^{2p'}\lambda_0^{-\frac{p'}{i}}(n+1)^{-p'(\gamma-2\beta-\lambda)}
\end{align*}
 gives
$$\mathbb{E}\left[P\left(\gamma_{n+1}^2\beta_{n+1}^2,V_n\right)\right]\leq L \left(\frac{p'\mu}{2} \lambda_n\gamma_{n+1}\right)(n+1)^{-p'(\gamma-2\beta-\lambda)}+\left(p'\mu(\lambda_n\gamma_{n+1})+A_{p'}(\gamma_{n+1}\beta_{n+1})^2\right)\mathbb{E}\left[V_n^{p'}\right]$$
with 
$$L=\frac{c_{\gamma}^{2p'-1}c_{\beta}^{2p'}}{\lambda_0}A_0+\frac{3 c_{\gamma}^{2(p'-1)c_{\beta}^{2p'}}\lambda_0^{-2}\sqrt{A_{p'}}}{4\mu}A_{p'/2}+\frac{2c_{\gamma}^{\frac{3}{2}p'}c_{\beta}^{2p'}\lambda_0^{-\frac{p'}{2}}}{\left(\frac{p'^2 \mu}{3(p'-2) }\right)^{\frac{p'-2}{2}}}A_2^{p'/2}+\frac{c_{\gamma}^{p'}c_{\beta}^{2p'}\lambda_0^{-p'}}{\left(\frac{p'^2 \mu}{3(p'-1) }\right)^{p'-1}}A_1.$$
Putting together the previous inequalities and taking the expectation yield then
\begin{align*}
\mathbb{E}\left[ V_{n+1}^{p'}  \right]  \leq &\left( 1-   p'\mu\gamma_{n+1}\lambda_{n}+\frac{A_{p'}c_{\gamma}c_{\beta}^2}{\lambda_0}(n+1)^{-\gamma+2\beta+\lambda} \gamma_{n+1}\lambda_n\right)\mathbb{E}\left[ V_{n}^{p'}  \right]  \\
&\hspace{3cm}+ \lambda_{n} \gamma_{n+1}\left(2p'\mu V_p^{p'} \bar{v}_n^{\frac{p-p'}{p}}+ \frac{Lp' \mu }{2}(n+1)^{-p'(\gamma-2\beta-\lambda)}\right).
\end{align*}

Then, recalling that $\bar{v}_n=v_{n+1}(n+1)^{-\delta-q\lambda}$ and using Proposition \ref{prop:appendix:delta_recursive_upper_bound_nt} yields
\begin{align*}
\mathbb{E}\left[ V_{n} ^{p'}\right]   \leq& \exp \left( -\frac{c_{\gamma} p'\mu \lambda_{0}}{2} n^{1-(\lambda+\gamma)} (1-\varepsilon(n)\right)\left( K_1^{(1')}+K_{1'}^{(1')}\max_{1 \leq k \leq n+1} k^{\gamma-2\beta-\lambda-\frac{p-p'}{p}(\delta+q\lambda)}v_{k}^{\frac{p-p'}{p}}\right)\\
&\hspace{4cm}+K_2^{(1')}n^{-p'(\gamma-2\beta -\lambda)}+K_3^{(1')}v_{\lfloor n/2\rfloor}^{\frac{p-p'}{p}}(n+1)^{-\frac{p-p'}{p}(\delta+q\lambda)},
\end{align*}
with 
\begin{equation}\label{eq:espilon(n)'}
\varepsilon(n)=\frac{4 C'_{M} n^{-1+\lambda+\gamma}}{ \mu  p'\lambda_{0}} \left(1+\frac{n^{(1+2\beta-2\gamma)^+}}{\vert 2 \gamma - 2\beta -1\vert}\right),
\end{equation}
and
\begin{equation}\label{eq:3.1p_first_constant}
K_1^{(1')}=\left( \mathbb{E}\left[ V_{0} \right]  + \frac{  p' \mu L}{C'_{M}}  \right) ,\quad K_{1'}^{(1')}=\frac{ 4p'\mu V_p^{p'}  }{C'_{M}},
\end{equation}
where 
\begin{equation}\label{eq:constant_C'_m}
C'_{M} = \max \left\lbrace \frac{A_{p'}c_{\gamma}c_{\beta}^2}{\lambda_0}, \left( \frac{\mu  p'\lambda_{0}}{8} \right)^{\frac{2\gamma - 2\beta}{\gamma+\lambda}}c_{\gamma}^{\frac{\gamma -2 \beta-\lambda}{\gamma+\lambda}} \right\rbrace,
\end{equation} and
\begin{equation}\label{eq:3.1p_second_constants}
K_2^{(1')}=  2^{p'(\gamma-2\beta-\lambda)}L ,\, K_3^{(1')}= 2^{2+\frac{p-p'}{p}(\delta+q\lambda)}V_{p}^{p'} .
\end{equation}
where $V_{p}$ is given in Lemma \ref{lem::majvn2}.

\subsection{Proof of Proposition \ref{prop::ordre4}}
Remark that with the help of a Taylor's expansion of $G$, one has
\begin{align*}
V_{n+1} & = V_{n} + \left( \theta_{n+1} - \theta _{n} \right)^{T}\int_{0}^{1} \nabla G \left( \theta_{n} + t \left( \theta_{n+1} - \theta_{n} \right) \right) dt \\
& =  V_{n} - \gamma_{n+1} \left( g_{n+1}' \right)^{T} A_{n}\int_{0}^{1}  \nabla G \left( \theta_{n} + t \left( \theta_{n+1} - \theta_{n} \right) \right) dt . 
\end{align*}
Then, using \eqref{maj::rest::taylor} one has
\begin{align*}
V_{n+1}^{2} & \leq  V_{n}^{2} -\overbrace{ 2 \gamma_{n+1} V_{n} \left( g_{n+1}' \right)^{T}A_{n} \int_{0}^{1} \nabla G \left( \theta_{n} + t \left( \theta_{n+1} - \theta_{n} \right) \right) dt}^{:= (\star ) }  \\
& + \underbrace{L_{\nabla G}^{2} \left\| A_{n} \right\|^{2}\left\| g_{n+1}' \right\|^{2}\gamma_{n+1}^{2}\left( 2\left\| \theta_{n} - \theta \right\|^{2} +\frac{1}{2} \gamma_{n+1}^{2} \left\| A_{n} \right\|^{2}\left\| g_{n+1}' \right\|^{2} \right) }_{:= (\star \star )}
\end{align*}
We now bound $(\star)$ and $(\star \star)$. First, thanks to Assumption \textbf{(H1)} and since $\left\| \theta_{n} - \theta \right\|^{2} \leq \frac{2}{\mu} V_{n}$, one has
\begin{align*}
\mathbb{E}\left[ (\star \star )| \mathcal{F}_{n} \right] & \leq \frac{4L_{\nabla G}^{2}C_{1}}{\mu} \gamma_{n+1}^{2} \left\| A_{n} \right\|^{2} V_{n} + \frac{8L_{\nabla G}^{2}C_{2}}{\mu^{2}}\left\| A_{n} \right\|^{2}\gamma_{n+1}^{2}V_{n}^{2} \\
&\hspace{3cm}+\frac{1}{2} L_{\nabla G  }^{2}C_{1}' \gamma_{n+1}^{4}\left\| A_{n} \right\|^{4} + \frac{2L_{\nabla G}^{2} C_{2}'}{\mu^{2}} \gamma_{n+1}^{4}\left\| A_{n} \right\|^{4} V_{n}^{2} \\
& \leq \frac{8L_{\nabla G}^{4}C_{1}^{2}}{\mu^{3}\lambda_{0}} \gamma_{n+1}^{3} \left\| A_{n} \right\|^{4} + \frac{1}{2}\mu \lambda_{0} \gamma_{n+1}V_{n}^{2} +  \frac{L_{\nabla G  }^{2}C_{1}'}{2} \gamma_{n+1}^{4}\left\| A_{n} \right\|^{4} \\
&\hspace{5cm}+   \frac{2L_{\nabla G}^{2}}{\mu^{2}} \left( 4C_{2} + C_{2}'c_{\gamma}^{2}c_{\beta}^{2} \right) \gamma_{n+1}^{2}\beta_{n+1}^{2} V_{n}^{2}
\end{align*}
% Ancienne version
%\begin{align*}
%& \leq \frac{L_{\nabla G}^{4}}{8\mu^{3}\lambda_{0}} \left\| A_{n} \right\|^{4} \gamma_{n+1}^{3} + \lambda_{0}\mu \gamma_{n+1}V_{n}^{2} + \frac{L_{\nabla G}^{2}}{4} C_{1} \gamma_{n+1}^{4}\left\| A_{n} \right\|^{4} + \frac{L_{\nabla G}^{2}}{2\mu} C_{2} \gamma_{n+1}^{4}\beta_{n+1}^{2}\left\| A_{n} \right\|^{2} V_{n}
%\end{align*}
Then, taking the expectation with Assumption \textbf{(H2b)},
\begin{align*}
\mathbb{E}\left[ (\star \star ) \right] 
& \leq \frac{8L_{\nabla G}^{4}C_{1}^{2}}{\mu^{3}\lambda_{0}} \gamma_{n+1}^{3} C_{S}^{4} + \frac{\mu \lambda_{0}}{2} \gamma_{n+1}\mathbb{E}\left[  V_{n}^{2} \right] +\frac{ L_{\nabla G  }^{2}C_{1}'}{2}\gamma_{n+1}^{4}C_{S}^{4} \\
&\hspace{5cm}+   \frac{2L_{\nabla G}^{2}}{\mu^{2}} \left( 4C_{2} + C_{2}'c_{\gamma}^{2}c_{\beta}^{2} \right) \gamma_{n+1}^{2}\beta_{n+1}^{2} \mathbb{E}\left[  V_{n}^{2} \right].
\end{align*}
Moreover, since $\nabla G$ is $L_{\nabla G}$-Lipschitz, one can check that
\begin{align*}
\left\| \int_{0}^{1} \nabla G \left( \theta_{n} + t \left( \theta_{n+1} - \theta_{n} \right) \right) - \nabla G \left( \theta_{n} \right)  dt \right\| \leq &L_{\nabla G}\int_{0}^{1}t dt \gamma_{n+1} \left\| A_{n} \right\| \left\| g_{n+1}' \right\| \\
\leq& \frac{L_{\nabla G}}{2} \gamma_{n+1} \left\| A_{n} \right\| \left\| g_{n+1}' \right\| .
\end{align*}
Then, one has
\begin{align*}
\mathbb{E}\left[ (\star) |\mathcal{F}_{n} \right] &   \geq 2 \gamma_{n+1} \nabla G \left( \theta_{n} \right)^{T} A_{n} \nabla G \left( \theta_{n} \right) V_{n} - L_{\nabla G}\gamma_{n+1}^{2} \left\| A_{n} \right\|^{2} \mathbb{E}\left[\left\| g_{n+1}' \right\|^{2}\vert \mathcal{F}_n\right] V_{n} \\
& \geq 2 \gamma_{n+1} \nabla G \left( \theta_{n} \right)^{T} A_{n} \nabla G \left( \theta_{n} \right) V_{n} - L_{\nabla G}\gamma_{n+1}^{2}  \left\| A_{n} \right\|^{2} C_{1} V_{n} - \frac{2L_{\nabla G}C_{2}}{\mu} \gamma_{n+1}^{2} \left\| A_{n} \right\|^{2} V_{n}^{2} \\
& \geq 2 \gamma_{n+1} \nabla G \left( \theta_{n} \right)^{T} A_{n} \nabla G \left( \theta_{n} \right) V_{n} -   \frac{C_{1}^{2}L_{\nabla G}^{2}}{2 \mu \lambda_{0}} \gamma_{n+1}^{3}  \left\| A_{n} \right\|^{4} - \frac{\mu \lambda_{0}\gamma_{n+1}}{2} V_{n}^2 - \frac{2L_{\nabla G}C_{2}}{\mu} \gamma_{n+1}^{2}\beta_{n+1}^{2} V_{n}^{2}.
\end{align*}
Furtermore, with the help of inequality \eqref{majgradsgrad} it comes 
\begin{align*}
\gamma_{n+1} \nabla G \left( \theta_{n} \right)^{T} A_{n} \nabla G \left( \theta_{n} \right) V_{n} \geq 2 \lambda_{0} \mu\gamma_{n+1}  V_{n}^{2} - 2\lambda_{0}\mu \gamma_{n+1} \mathbf{1}_{A_{n} <  \lambda_{0}} V_{n}^{2} .
\end{align*}
Then,  with the help of Holder's inequality, coupled with \textbf{(H1a)} for $t=1$, one has
\begin{align*}
\mathbb{E}\left[ (\star) \right] \geq \frac{7}{2}\lambda_{0}\mu \gamma_{n+1} V_{n}^{2} -4 \lambda_{0} \mu \gamma_{n+1}\bar{v}_{n}^{(p-2)/p}V_{p}^{2} - \frac{C_{1}^{2}L_{\nabla G}^{2}}{2 \mu \lambda_{0}} \gamma_{n+1}^{3}  C_{S}^{4} - \frac{2L_{\nabla G}C_{2}}{\mu} \gamma_{n+1}^{2}\beta_{n+1}^{2} \mathbb{E} \left[  V_{n}^{2} \right]
\end{align*}
with $V_{p}$ defined in Lemma \ref{lem::majvn2} and $\bar{v}_n:=v_n(n+1)^{-\delta}$ is the upper bound from \textbf{(H1a)} on $\mathbb{P}\left[ \lambda_{\min}\left(A_n\right)\leq\lambda_0\right]$.
Let 
\begin{equation}\label{eq:def_aM}
a_{M} :=  \max \left\lbrace  \left(  \frac{2L_{\nabla G}C_{2}}{\mu} +\frac{2L_{\nabla G}^{2}}{\mu^{2}} \left( 4C_{2} + C_{2}'c_{\gamma}^{2}c_{\beta}^{2} \right)   \right) c_{\gamma}c_{\beta}^{2} , \left( \frac{3\lambda_{0}\mu}{2} \right)^{\frac{2\gamma -2\beta}{\gamma}}c_{\gamma}^{\frac{\gamma -2 \beta}{\gamma}} \right\rbrace ,
\end{equation} one has
\begin{align}
\mathbb{E}\left[ V_{n+1}^{2} \right] & \leq \left( 1 -3 \lambda_{0} \mu \gamma_{n+1} + a_{M}n^{2\beta - \gamma}\gamma_{n+1} \right) \mathbb{E}\left[ V_{n} \right] + 4 \lambda_{0} \mu \gamma_{n+1}\bar{v}_{n}^{(p-2)/p}V_{p}^{2} \nonumber\\
& + \underbrace{C_{S}^{4}L_{\nabla G}^{2}\left( \frac{8L_{\nabla G}^{4}C_{1}^{2}}{\mu^{3}\lambda_{0}} + \frac{C_{1}'c_{\gamma}}{2}+   \frac{C_{1}^{2}}{2 \mu \lambda_{0}} \right)}_{=: a_{1}} \gamma_{n+1}^{3} \label{eq:def_a_1}
\end{align}
Applying Proposition \ref{prop:appendix:delta_recursive_upper_bound_nt}, it comes (with analogous calculus to the ones in the proof of Theorem \ref{theo1})
\begin{align*}
 &\mathbb{E}\left[ V_{n}^{2} \right]
 \leq \exp \left( - \frac{3}{2}c_{\gamma} \lambda_{0}\mu n^{1-\gamma} \right) \exp \left( 2 a_{M}\frac{2 \gamma -2 \beta}{2\gamma  -2 \beta -1} \right) \\
 &\cdot\left( \mathbb{E}\left[ V_{0}^{2} \right] + \frac{2a_{1}c_{\gamma}^{2}}{a_{M}} + \frac{8 \lambda_{0}\mu c_{\gamma}V_{p}^{2/p}}{a_{M}}\max_{1\leq k\leq n+1 }v_k^{\frac{p-2}{p}}k^{\gamma-\frac{(p-2)}{p}\delta} \right) + \frac{2^{2\gamma}a_{1}c_{\gamma}^{2}}{3\lambda_{0} \mu}n^{-2\gamma} + \frac{4}{3} V_{p}^{2 }\bar{v}_{\lfloor n/2\rfloor}^{(p-2)/p} .
\end{align*}
where $V_{p}$ is given by Lemma \ref{lem::majvn2} and $\bar{v}_{\lfloor n/2\rfloor}\leq v_{n/2}2^{\delta}(n+1)^{-\delta}$. Setting 
\begin{equation}\label{eq:prop31_constant1}
K^{(2')}_1= \exp \left( 2 a_{M}\frac{2 \gamma -2 \beta}{2\gamma  -2 \beta -1} \right) \left( \mathbb{E}\left[ V_{0}^{2} \right] + \frac{2a_{1}c_{\gamma}^{2}}{a_{M}}\right),
\end{equation}
\begin{equation}\label{eq:prop31_constant1'}
 K^{(2')}_{1'}=\exp \left( 2 a_{M}\frac{2 \gamma -2 \beta}{2\gamma  -2 \beta -1} \right) \cdot \frac{8 \lambda_{0}\mu V_{p}^{2 }}{a_{M}} ,
\end{equation}
with $a_M$ given in \eqref{eq:def_aM}, $a_1$ given in \eqref{eq:def_a_1} and $V_p$ given in Lemma \ref{lem::majvn2}, and
\begin{equation}\label{eq:prop31_constant2}
K^{(2')}_2= \frac{2^{2\gamma}a_{1}c_{\gamma}^{2}}{3\lambda_{0} \mu},\quad K^{(2')}_3=\frac{2^{2+(p-2)\delta/p}}{3}V_{p}^{2 },
\end{equation}
we finally get 
\begin{align*}
\mathbb{E}\left[ V_{n}^{2} \right]  \leq &\exp \left( - \frac{3}{2}c_\gamma\lambda_{0}\mu n^{1-\gamma} \right)\left(K^{(2')}_1+K^{(2')}_{1'}\max_{1\leq k\leq n+1}v_k^{\frac{p-2}{p}}k^{\gamma-\delta\frac{p-2}{p}}\right) \\
 &\hspace{6cm}+K^{(2')}_2n^{-2\gamma} + K^{(2')}_3v_{\lfloor n/2\rfloor}^{(p-2)/p}n^{-\delta(p-2)/p} .
 \end{align*}
%\begin{lem}\label{lem::majvn3}
%Suppose Assumption \textbf{(A1)} for some $p\geq 2$ and \textbf{(H1b)} hold. Then, for all $n \geq 0$, if $\gamma >1/2$ then

%\[
%\mathbb{E}\left[ V_{n}^{p} \right] \leq e^{ a_p c_{\gamma}^{2}c_{\beta}^{2} \frac{2\gamma -2 \beta}{2\gamma -2 \beta -1}} \max\left\lbrace 1 , \mathbb{E}\left[ V_{0}^{p} \right]  \right\rbrace :=V_p^p
%\]
%and if $\gamma\leq 1/2$ then
%\begin{align*}
%&\mathbb{E}\left[ V_{n}^{p} \right] \\
%&\leq\underbrace{\exp\left(-p\mu\lambda_0c_{\gamma}\left(1+\frac{1+\left(\frac{c_{\gamma}c_{\beta}^2 a_p}{p\mu \lambda_0}\right)^{\frac{1-\gamma-\lambda}{\gamma-2\beta-\lambda}}}{1-\gamma-\lambda}\right)+c_{\gamma}^2c_{\beta}^2a_p\left(1+\frac{1+\left(\frac{c_{\gamma}c_{\beta}^2 a_p}{p\mu \lambda_0}\right)^{\frac{1-2\gamma+2\beta}{\gamma-2\beta-\lambda}}}{1-2\gamma+2\beta}\right)\right)\max\left\lbrace 1 , \mathbb{E}\left[ V_{0}^{p} \right]  \right\rbrace}_{ :=V_p^p},
%\end{align*}
%with $a_p$ given in \eqref{eq:definition_a_p} in both cases.
%\end{lem}

\section{Proofs of tehcnical lemmas} \label{sec::technic}

\subsection{Proof of Lemma \ref{lem::majvn2} (Nouvelle version)}
Observe that since the proofs are analogous, we only make the proof for $p > 2$, and for the case where $p=2$, if there are some differences in the proof, it will be indicated with the help of remarks.

With the help of a Taylor  expansion of the functional $G$, one has
\begin{align*}
V_{n+1} &  =  V_{n} - \gamma_{n+1} \left( g_{n+1}' \right)^{T} A_{n}\int_{0}^{1}  \nabla G \left( \theta_{n} + t \left( \theta_{n+1} - \theta_{n} \right) \right) dt .
\end{align*}
Then, applying the inequality 
\begin{align*}
(a+h)^p & \leq a^p+pa^{p-1}h+\frac{p(p-1)h^2}{2}\max(1,2^{p-3})(a^{p-2}+\vert h\vert^{p-2}) \\
& \leq  a^p+pa^{p-1}h+ p(p-1)2^{p-3}h^2(a^{p-2}+\vert h\vert^{p-2})
\end{align*}
for $a,a+h \geq 0$ to $a=V_n$ and $h=- \gamma_{n+1} \left( g_{n+1}' \right)^{T} A_{n}\int_{0}^{1} (1-t) \nabla G \left( \theta_{n} + t \left( \theta_{n+1} - \theta_{n} \right) \right) dt$, one has
\begin{align*}
V_{n+1}^{p} & \leq V_{n}^{p} -    p \gamma_{n+1}  \left( g_{n+1}' \right)^{T} A_{n}\int_{0}^{1}   \nabla G \left( \theta_{n} + t \left( \theta_{n+1} - \theta_{n} \right) \right) dt  V_{n}^{p-1}   \\
& +   2^{p-3} p(p-1)\left\Vert \gamma_{n+1} \left( g_{n+1}' \right)^{T} A_{n}\int_{0}^{1}   \nabla G \left( \theta_{n} + t \left( \theta_{n+1} - \theta_{n} \right) \right) dt\right\Vert^2 V_{n}^{p-2}  \\
& +   2^{p-3}p(p-1)\left\Vert \gamma_{n+1} \left( g_{n+1}' \right)^{T} A_{n}\int_{0}^{1} \nabla G \left( \theta_{n} + t \left( \theta_{n+1} - \theta_{n} \right) \right) dt\right\Vert^p  
\end{align*}
\begin{rmq}
Observe that in the case where $p=2$, one has 
\[
(a+h)^{2} = a^{2} +2ah + h^{2} = a^{p} +2a^{p-1}h + p(p-1)2^{p-3}h^{2} |h|^{p-2}
\]
the last term on the right hand-side of previous inequality can be considered equal to $0$.
\end{rmq}

Recalling that since $\nabla G$ is $L_{\nabla G}$-Lipschitz, one has
\begin{align*}
\left\| \int_{0}^{1} (1-t) \nabla G \left( \theta_{n} + t \left( \theta_{n+1} - \theta_{n} \right) \right) dt \right\|  & \leq  L_{\nabla G} \left(  \left\| \theta_{n} - \theta \right\| + \gamma_{n+1} \left\| A_{n} \right\| \left\|  g_{n+1}' \right\|\right),
\end{align*}
which implies 
\begin{align*}
V_{n+1}^{p} & \leq V_{n}^{p} -  \left. p \gamma_{n+1}  \left( g_{n+1}' \right)^{T} A_{n}\int_{0}^{1}   \nabla G \left( \theta_{n} + t \left( \theta_{n+1} - \theta_{n} \right) \right) dt  V_{n}^{p-1} \right\rbrace =: (*) \\
& + \left. 2^{p-2}p(p-1)L_{\nabla G}^{2} \gamma_{n+1}^{2} \left\| g_{n+1}' \right\|^{2} \left\| A_{n} \right\|^{2} \left( \left\| \theta_{n} - \theta \right\|^{2} + \gamma_{n+1}^{2} \left\| A_{n} \right\|^{2} \left\| g_{n+1}' \right\|^{2} \right) V_{n}^{p-2}   \right\rbrace =:(**)\\
& + \left. 2^{p-2}p(p-1)L_{\nabla G}^{p} \gamma_{n+1}^{p} \left\| g_{n+1}' \right\|^{p} \left\| A_{n} \right\|^{p} \left( \left\| \theta_{n} - \theta \right\|^{p} + \gamma_{n+1}^{p} \left\| A_{n} \right\|^{p} \left\| g_{n+1}' \right\|^{p} \right) \right\rbrace =:(***)
\end{align*}
Furthermore, one has
\begin{align*}
(*) & =   - p \gamma_{n+1}  \left( g_{n+1}' \right)^{T}   A_{n}\int_{0}^{1}  \nabla G \left( \theta_{n} + t \left( \theta_{n+1} - \theta_{n} \right) \right) dt  V_{n}^{p-1} \\
&    =  - p \gamma_{n+1}  \left( g_{n+1}'  \right)^{T} A_{n} \nabla G \left( \theta_{n} \right)   V_{n}^{p-1}   \\
 & - p \gamma_{n+1}  \left( g_{n+1}'  \right)^{T} A_{n}\int_{0}^{1}   \left( \nabla G \left( \theta_{n} + t \left( \theta_{n+1} - \theta_{n} \right) \right) - \nabla G \left( \theta_{n} \right) \right) dt  V_{n}^{p-1} \\
\end{align*}
Since $A_{n}$ is positive and since $\nabla G$ is $L_{\nabla G}$-lipschitz, taking the conditional expectation, it comes, since for all $a,b \geq 0$, $ab \leq \frac{1}{p}a^{p} + \frac{p-1}{p}b^{p/(p-1)}$ and with the help of Assumption \textbf{(H1a)},
\begin{align*}
\mathbb{E}\left[ (*) |\mathcal{F}_{n} \right] &  \leq - p \gamma_{n+1}\nabla G \left( \theta_{n} \right)^{T} A_{n} \nabla G \left( \theta_{n} \right) V_{n}^{p-1} +  \frac{p}{2}\gamma_{n+1}^{2}\left\| A_{n} \right\|^{2}\mathbb{E}\left[ \left\| g_{n+1}' \right\|^{2} |\mathcal{F}_{n} \right] V_{n}^{p-1} \\
&  \leq - p \gamma_{n+1}  \lambda_{\min} \left( A_{n} \right) \left\| \nabla G \left( \theta_{n} \right) \right\|^{2} V_{n}^{p-1} +   \frac{p}{2}\beta_{n+1}^{2} \gamma_{n+1}^{2} \left( C_{1} + C_{2} \left\| \theta_{n} - \theta \right\|^{2} \right) V_{n}^{2} \\
& \leq -p \mu \gamma_{n+1}  \lambda_{\min} \left( A_{n} \right) V_{n}^{p}  + \frac{pC_{2}}{\mu}\beta_{n+1}^{2}\gamma_{n+1}^{2} V_{n}^{p} + \frac{pC_{1}}{2}\beta_{n+1}^{2} \gamma_{n+1}^{2} V_{n}^{p-1} \\
&  \leq  - p \mu \gamma_{n+1} \lambda_{n+1}' \mathbf{1}_{\gamma \leq 1/2} V_{n}^{p}+  \left( \frac{pC_{2}}{\mu} + \frac{C_{1}(p-1)}{2}  \right) \beta_{n+1}^{2}\gamma_{n+1}^{2} V_{n}^{p} + \frac{C_{1}}{2}\beta_{n+1}^{2} \gamma_{n+1}^{2},
\end{align*}
with $\lambda_{n}' = \lambda_{0}' n^{-\lambda '}$. We also used Assumptions \textbf{(A1)}  on the first inequality and the fact that $\left\| \theta_{n} - \theta \right\|^{2} \leq \frac{2}{\mu} V_{n} \leq \frac{2}{\mu^{2}} \left\| \nabla G \left( \theta_{n} \right) \right\|^{2}$ on the third inequality. For the same reasons, one has
\begin{align*}
\mathbb{E}\left[ (**) |\mathcal{F}_{n} \right] & \leq 2^{p-2}p(p-1)L_{\nabla G}^{2} \left( \gamma_{n+1}^{4}\beta_{n+1}^{4} \left( C_{1}' + \frac{4C_{2}'}{\mu^{2}} V_{n}^{2} \right) + \gamma_{n+1}^{2}\beta_{n+1}^{2} \left( \frac{2C_{1}}{\mu}V_{n} + \frac{4C_{2}}{\mu^{2}}V_{n}^{2} \right) \right)    V_{n}^{p-2}
  \\
&  \leq   2^{p-2}(p-1)L_{\nabla G}^{2}  \gamma_{n+1}^{4}\beta_{n+1}^{4} \left(  2C_{1}' + \left( (p-2)C_{1}' + \frac{4pC_{2}'}{\mu^{2}} \right) V_{n}^{p} \right) \\
& + 2^{p-2}(p-1)L_{\nabla G}^{2}  \gamma_{n+1}^{2}\beta_{n+1}^{2} \left(  \frac{2C_{1}}{\mu} +   \left( \frac{2(p-1)C_{1}}{\mu} + \frac{4pC_{2}}{\mu^{2}} \right) V_{n}^{p}   \right)
\end{align*}
In a same way, thanks to Assumptions \textbf{(A1'')} and \textbf{(H1)}, one has
\begin{align*}
\mathbb{E}\left[ (***) |\mathcal{F}_{n} \right] & \leq 2^{p-2}p(p-1)L_{\nabla G}^{p} \gamma_{n+1}^{ 2p} \beta_{n+1}^{ 2p}   \left( C_{1}^{(p)} + \frac{2^{p}C_{2}^{(p)}}{\mu^{p}} V_{n}^{p}  \right)  \\
& + 2^{p-2}p(p-1)L_{\nabla G}^{p} \gamma_{n+1}^{ p} \beta_{n+1}^{ p} \left( \frac{1}{2} C_{1}^{(p)} + \frac{2^{p}}{\mu^{p}} \left( \frac{1}{2}+ \sqrt{C_{2}^{(p)}} \right) V_{n}^{p} \right)
\end{align*}
Taking the expectation on $\mathbb{E}\left[ (*) |\mathcal{F}_{n} \right]+\mathbb{E}\left[ (**) |\mathcal{F}_{n} \right]+\mathbb{E}\left[ (***) |\mathcal{F}_{n} \right]$, applying the latter inequalities, it comes
\begin{align*}
\mathbb{E}\left[ V_{n+1}^{p} \right] \leq \max\left\lbrace \mathbb{E}\left[ V_{n}^{p} \right],1\right\rbrace \left(  1-p\mu\lambda_{n+1}' \gamma_{n+1}\mathbf{1}_{\gamma \leq 1/2} + a_p\gamma_{n+1}^{2}\beta_{n+1}^{2} \right)
\end{align*}
with
\begin{align}
\notag a_{p} :& = p\left( \frac{C_{2}}{\mu} + \frac{C_{1}}{2} \right) + 2^{p-2}(p-1)p L_{\nabla G}^{2} \left( c_{\gamma}^{2}c_{\beta}^{2} \left( C_{1}' + \frac{4C_{2}'}{\mu^{2}} \right) + \frac{2C_{1}}{\mu} + \frac{4C_{2}}{\mu^{2}} \right) \\
\label{def::ap} &  + 2^{p-2}(p-1)pL_{\nabla G}^{p} \left( c_{\gamma}^{2p-2}c_{\beta}^{2p-2} \left( C_{1}^{(p)} + \frac{2^{p}C_{2}^{(p)}}{\mu^{2}} \right)  + c_{\gamma}^{p-2}c_{\beta}^{p-2} \left( \frac{1}{2}C_{1}^{(p)} + \frac{2p}{\mu^{2}} \left( \frac{1}{2} + \sqrt{C_{2}^{(p)}} \right) \right) \right) .
\end{align}
\begin{rmq}
Observe that in the case where $p = 2$,   one has
\begin{equation}
\label{def::a2} a_{2} = C_{1} + \frac{2C_{2}}{\mu} + \frac{4L_{\nabla G}^{2}}{\mu}C_{1} + \frac{8L_{\nabla G}^{2}C_{2}}{\mu^{2}} + 2L_{\nabla G}^{2}C_{1}'c_{\gamma}^{2}c_{\beta}^{2} + \frac{8L_{\nabla G}^{2}C_{2}'}{\mu^{2}}c_{\gamma}^{2}c_{\beta}^{2}
\end{equation}
\end{rmq}
If $\gamma>1/2$, by summation,
\[
\mathbb{E}\left[ V_{n}^{p} \right] \leq e^{ a_p c_{\gamma}^{2}c_{\beta}^{2} \frac{2\gamma -2 \beta}{2\gamma -2 \beta -1}} \max\left\lbrace 1 , \mathbb{E}\left[ V_{0}^{p} \right]  \right\rbrace =: V_p^p.
\]
If $\gamma\leq 1/2$, let $n_0$ be the smallest integer such that 
$\gamma_{n+1}^{2}\beta_{n+1}^{2}a_p> p\mu\lambda_n '\gamma_{n+1}$. Recording that $\lambda_n '=\lambda_0 '(n+1)^{-\lambda '}$, we have
$n_0=\left\lfloor \left(\frac{c_{\gamma}c_{\beta}^2 a_p}{p\mu \lambda_0 '}\right)^{\frac{1}{\gamma-2\beta-\lambda '}}\right\rfloor $. Then,
\begin{align*}
\mathbb{E}\left[ V_{n}^{p} \right] \leq& \exp\left(\sum_{n=0}^{n_0}-p\mu\lambda_n '\gamma_{n+1}+a_p\gamma_{n+1}^2\beta_{n+1}^2\right)\max\left\lbrace 1 , \mathbb{E}\left[ V_{0}^{p} \right]  \right\rbrace\\
\leq& \exp\left(-p\mu\lambda_0 ' c_{\gamma}\left(1+\frac{1+\left(\frac{c_{\gamma}c_{\beta}^2 a_p}{ p\mu \lambda_0 ' }\right)^{\frac{1-\gamma-\lambda ' }{\gamma-2\beta-\lambda ' }}}{1-\gamma-\lambda ' }\right)+c_{\gamma}^2c_{\beta}^2a_p\left(1+\frac{1+\left(\frac{c_{\gamma}c_{\beta}^2 a_p}{p\mu \lambda_0 ' }\right)^{\frac{1-2\gamma+2\beta}{\gamma-2\beta-\lambda ' }}}{1-2\gamma+2\beta}\right)\right) =:V_p^p.
\end{align*}

\subsection{Proof of Lemma \ref{lem:H1_adagrad}}
Recall that $\left(A_n\right)_{kk'}=\max \left\lbrace \min\left\lbrace c_{\beta} n^{\beta}, \left(\overline{A_n}\right)_{kk'}\right\rbrace , \lambda_{0}'n^{-\lambda '} \mathbf{1}_{\gamma \leq 1/2} \right\rbrace$ with $\left(\overline{A_{n}}\right)_{kk'}=\frac{\delta_{kk'}}{\sqrt{\frac{1}{n+1}\left(a_{k}+\sum_{i=0}^{n-1}\left(\nabla_hg(X_{i+1},\theta_i)_k\right)^2\right)}}$. Since $\lambda_{\min}\left( A_n \right) \geq \lambda_{\min} \left( \overline{A_n} \right)$ on the event $\left\lbrace \lambda_{\min} \left( \overline{A_n} \right)<c_{\beta}\right\rbrace$,  we have for $0<t<1$
\begin{align*}
\mathbb{P}\left[ \lambda_{\min} \left( A_n \right) <tc_{\beta} \right]\leq& \mathbb{P}\left[ \lambda_{\min} \left( \overline{A_n}\right)<tc_{\beta}\right]\\
\leq& \mathbb{P}\left[\max_{1\leq k\leq d}\frac{1}{n+1}\left(a_{k}+\sum_{i=0}^{n-1}\left(\nabla_hg\left( X_{i+1},\theta_i\right))_k\right)^2\right)>\frac{1}{c_{\beta}^2t^2}\right].
\end{align*}
Then, Markov inequality for $p>2$ and Jensen inequality yields 
\begin{align*}
\mathbb{P}&\left[\max_{1\leq k\leq d}\sqrt{\frac{1}{n+1}\left(a_{k}+\sum_{i=0}^{n-1}\left(\nabla_hg\left(X_{i+1},\theta_i\right)_k\right)^2\right)}>\frac{1}{c_{\beta}t}\right]\\
&\hspace{3cm}\leq c_{\beta}^{2p}t^{2p}\mathbb{E}\left[\left(\max_{1\leq k\leq d}\frac{1}{n+1}\left(a_{k}+\sum_{i=0}^{n-1}\left(\nabla_hg\left(X_{i+1},\theta_i\right)_k\right)^2\right)\right)^p\right]\\
&\hspace{3cm}\leq c_{\beta}^{2p}t^{2p}\mathbb{E}\left[\left(\frac{1}{n+1}\left(\sum_{i=1}^da_{k}+\sum_{i=0}^{n-1}\left\Vert\nabla_hg\left(X_{i+1},\theta_i\right)\right\Vert^2\right)\right)^p\right]\\
&\hspace{3cm}\leq c_{\beta}^{2p}t^{2p}\frac{1}{n+1}\left(\left(\sum_{i=1}^da_{k}\right)^p+\sum_{i=0}^{n-1}\mathbb{E}\left[\left\Vert\nabla_hg\left(X_{i+1},\theta_i\right)\right\Vert^{2p}\right]\right).
\end{align*}
Then, using Assumption  \textbf{(A1)} and then \textbf{(A2)} we get 
\begin{align*}
\mathbb{P}&\left[\max_{1\leq k\leq d}\sqrt{\frac{1}{n+1}\left(a_{k}+\sum_{i=0}^{n-1}\left(\nabla_hg(X_{i+1},\theta_i)_k\right)^2\right)}>\frac{1}{c_{\beta}t}\right]\\
&\hspace{3cm}\leq c_{\beta}^{2p}t^{2p}\frac{1}{n+1}\left(\left(\sum_{i=1}^da_{k}\right)^p+nC_1''+C_2''\sum_{i=0}^{n-1}\mathbb{E}\left[\left\Vert \theta_i-\theta\right\Vert^{2p}\right]\right)\\
&\hspace{3cm}\leq c_{\beta}^{2p}t^{2p}\frac{1}{n+1}\left(\left(\sum_{i=1}^da_{k}\right)^p+nC_1''+\frac{2^pC_2''}{\mu^p}\sum_{i=0}^{n-1}\mathbb{E}\left[V_n^{p}\right]\right).
\end{align*}
By the bound $\mathbb{E}\left[V_n^{p}\right]\leq V_p^p$ from Lemma \ref{lem::majvn2}, we finally get
$$\mathbb{P}\left[ \max_{1\leq k\leq d}\sqrt{\frac{1}{n+1}\left(a_{k}+\sum_{i=0}^{n-1}\left(\nabla_hg\left(X_{i+1},\theta_i\right)_k\right)^2\right)}>\frac{1}{c_{\beta}t}\right]\leq v_nt^{2p}$$
with 
\begin{equation}\label{eq:def_v0_adagrad}
v_n= c_{\beta}^{2p}\left(\left(\frac{1}{n}\sum_{i=1}^da_{k}\right)^p+C_1''+\frac{2^pC_2''V_p^p}{\mu^p}\right).
\end{equation}

\subsection{Proof of Lemma \ref{lem:tech:ada:lambda}}
Set $E_k=\mathbb{E}\left[ \nabla_hg\left(X_{},\theta\right)_k^2\right]$ and $\partial_k^2g(h)=\mathbb{E}\left[\nabla_hg(X,h)_k^2\right]$. Then, by Jensen's inequality for $p'\geq 2$,
\begin{align*}
 \left\vert\left(\overline{A_n}\right)_{kk}\right\vert^{-2p'}\leq 2^{p' -1}\left\vert \frac{1}{n+1}\sum_{i=0}^{n-1}\nabla_hg\left(X_{i+1},\theta_i\right)_k^2 -\partial_k^2g\left(\theta_i\right) \right\vert^{p'}+2^{p'-1}\left\vert\frac{a_k}{n+1}+\frac{1}{n+1}\sum_{i=0}^{n-1}\partial_k^2 g(\theta_i)\right\vert^{p'}.
\end{align*}
Hence, for any $x>0$,
\begin{align}
\mathbb{P}\left[\left\vert\left(\overline{A_n}\right)_{kk}\right\vert< \frac{1}{x}\right]=\mathbb{P}\left[\left\vert\left(\overline{A_n}\right)_{kk}\right\vert^{-2p'}> x^{2p'}\right]\leq & \mathbb{P}\left[\left\vert \frac{1}{n+1}\sum_{i=0}^{n-1}\nabla_hg\left(X_{i+1},\theta_i\right)_k^2 -\partial_k^2g\left(\theta_i\right) \right\vert^{p'}>\frac{x^{2p'}}{2^{p'}}\right]\nonumber\\
+&\mathbb{P}\left[\left\vert\frac{a_k}{n+1}+\frac{1}{n+1}\sum_{i=0}^{n-1}\partial_k^2 g(\theta_i)\right\vert^{p'}>\frac{x^{2^{p'}}}{2^{p'}}\right].\label{eq:Adagrad_H1a_first_inequality}
\end{align}
Set $M_0=0$ and for $n\geq 1$,
$$M_n=\sum_{i=0}^{n-1}\nabla_hg(X_{i+1},\theta_i)_k^2 -\partial_k^2g(\theta_i).$$
Then, $(M_n)_{n\geq 0}$ is a martingale, and thus by Burkholder's inequality, see \cite[Theorem 2.10]{hall2014martingale} there exists an explicit constant $C_{p'}$ such that
\begin{align*}
\mathbb{E}\left[ \vert M_n\vert^{p'} \right]\leq C_{p'}\mathbb{E}\left[\left\vert\sum_{i=1}^{n}\left(M_i-M_{i-1}\right)^2\right\vert^{p'/2}\right]\leq& C_{p'}n^{p'/2-1}
\sum_{i=1}^{n}\mathbb{E}\left[\left\vert M_i-M_{i-1}\right\vert^{p'}\right]\\
\leq& C_{p'}n^{p'/2-1}
\sum_{i=0}^{n-1}\mathbb{E}\left[\left\vert\nabla_hg\left(X_{i+1},\theta_i\right)_k^2 -\partial_k^2g\left(\theta_i\right)\right\vert^{p'}\right],
\end{align*}
where we used Jensen's inequality on the second inequality. By Assumption \textbf{(A1)}, the  strong convexity of $G$ and Lemma \ref{lem::majvn2}, 
\begin{align*}
\mathbb{E}\left[\left(\nabla_hg\left(X_{i+1},\theta_i\right)_k^2 -\partial_k^2g\left(\theta_i\right)\right)^{p'}\right]\leq 2^{p'}\mathbb{E}\left[\left(\nabla_hg\left(X_{i+1},\theta_i\right)_k\right)^{2p}\right]
\leq& 2^{p'}\mathbb{E}\left[\Vert \nabla_hg\left(X_{i+1},\theta_i\right)\Vert^{2p'}\right]\\
\leq& 2^{p'}C_1^{(p')}+2^{p'}C_2^{(p')}\mathbb{E}\left[\Vert \theta_i-\theta\Vert^{2p'}\right]\\
\leq&2^{p'} C_1^{(p')}+2^{2p'}C_2^{(p')}\frac{ V_p^{p'}}{\mu^p}.
\end{align*}
Hence,
\begin{equation}\label{eq:adagrad_bound_martingale}
\mathbb{E}\left[\left\vert \frac{1}{n+1}\sum_{i=0}^{n-1}\nabla_hg\left(X_{i+1},\theta_i\right)_k^2 -\partial_k^2g\left(\theta_i\right) \right\vert^{p'}\right]=\mathbb{E}\left[\left\vert \frac{1}{n+1}M_n\right\vert^{p'}\right]\leq 2^{p'}\frac{C_1^{(p')}+2^{p'}C_2^{(p')}\frac{ V_{p }^{p'}}{\mu^{p'}}}{(n+1)^{p'/2}},
\end{equation}
which yields for $x>0$
\begin{equation}\label{eq:adagrad_H1a_Mn}
\mathbb{P} \left[ \left\vert \frac{1}{n+1}\sum_{i=0}^{n-1}\nabla_hg\left(X_{i+1},\theta_i\right)_k^2 -\partial_k^2g\left(\theta_i\right) \right\vert^{p'}>\frac{x^{2p'}}{2^{p'}}  \right]\leq \frac{2^{2p'}}{x^{2p'}} \frac{C_1^{(p')}+2^{p'}C_2^{(p')}\frac{ V_{p}^{p'}}{\mu^{p'}}}{(n+1)^{p'/2}}.
\end{equation}
Next, by Jensen inequality,
\begin{align*}
\left\vert\frac{a_k}{n+1}+\frac{1}{n+1}\sum_{i=0}^{n-1}\left( \partial_k^2 g\left(\theta_i\right)\right)\right\vert^{p'}\leq& \frac{1}{n+1}\left(\left\vert a_k\right\vert^{p'}+\sum_{i=0}^{n-1}\left\vert\partial_k^2 g(\theta_i)\right\vert^{p'}\right).
\end{align*}
Using Assumption \textbf{(A1)} and then strong convexity yields
\begin{align*}
\left\vert\partial_k^2 g\left(\theta_i\right)\right\vert^{p'}\leq&  C_1^{(p')}+2^{p'}C_2^{(p')}\frac{V_p^{p'}}{\mu^{p'}},
\end{align*}
so that 
$$\left\vert\frac{a_k}{n+1}+\frac{1}{n+1}\sum_{i=0}^{n-1}\left( \partial_k^2 g\left(\theta_i\right)\right)\right\vert^{p'}\leq C_1^{(p')}+\frac{\left\vert a_k\right\vert^{p'}}{n+1}+\frac{2^{p'}C_2^{(p')}}{\mu^{p'}}\left(\frac{1}{n+1}\sum_{i=0}^{n-1}V_i^{p'}\right).$$
Hence, for $\frac{x^{2p'}}{2p'}>C_1^{(p')}$,
\begin{align*}
\mathbb{P}\left[\left\vert\frac{a_k}{n+1}+\frac{1}{n+1}\sum_{i=0}^{n-1}\partial_k^2 g(\theta_i)\right\vert^{p'}>\frac{x^{2p'}}{2^{p'}}\right]\leq& \mathbb{P}\left[\frac{1}{n+1}\left(\left\vert a_k\right\vert^{p'}+\frac{2^{p'}C_2^{(p')}}{\mu^{p'}}\sum_{i=0}^{n-1}V_i^{p'}\right)>\frac{x^{2p'}}{2^{p'}}-C_{1}^{(p')}\right]\\
\leq&\frac{1}{n+1}\frac{\mathbb{E}\left[\left\vert a_k\right\vert^{p'}+\frac{2^{p'}C_2^{(p')}}{\mu^{p'}}\sum_{i=0}^{n-1}V_i^{p'}\right]}{\frac{x^{2p'}}{2^{p'}}-C_{1}^{(p')}}.
\end{align*}
By \eqref{eq:bound_Vp'_adagrad} and the fact that $ \frac{1}{n+1} \sum_{i=0}^{n-1} (i+1)^{- \frac{2(1-\gamma)\gamma(\gamma-2\beta)p}{2-\gamma}} \leq \frac{1}{n+1}+  \frac{1}{\left| 1- \frac{2(1-\gamma)\gamma(\gamma-2\beta)p}{2-\gamma} \mathbf{1}_{\frac{2(1-\gamma)\gamma(\gamma-2\beta)p}{2-\gamma} \neq 1} \right|} \frac{\log(n+1)}{(n+1)^{\frac{2(1-\gamma)\gamma(\gamma-2\beta)p}{2-\gamma}\wedge 1}} $, and denoting $\tilde{1} = 1+  \frac{1}{\left| 1- \frac{2(1-\gamma)\gamma(\gamma-2\beta)p}{2-\gamma} \mathbf{1}_{\frac{2(1-\gamma)\gamma(\gamma-2\beta)p}{2-\gamma} \neq 1} \right|} $, it comes 
\begin{align*}
 \frac{1}{n+1} & \mathbb{E}\left[\left\vert a_k\right\vert^{p'}+\frac{2^{p'}C_2^{(p')}}{\mu^{p'}}\sum_{i=0}^{n-1}V_i^{p'}\right] = \frac{\left\vert a_k\right\vert^{p'}+\frac{2^{p'}C_2^{(p')}}{\mu^{p'}}\sum_{i=0}^{n-1}\mathbb{E}\left[V_i^{p'}\right]}{n+1}\\
\leq&\frac{2^{p'}C_2^{(p')}}{\mu^{p'}}\tilde{K}_2 \tilde{1}\frac{\log(n+1)}{(n+1)^{\frac{2(1-\gamma)\gamma(\gamma-2\beta)p}{2-\gamma}\wedge 1}}+\frac{2^{p'}C_2^{(p')}}{\mu^{p'}(n+1)}\left[1+\left\vert a_k\right\vert^{p'}+\tilde{K}_1\sum_{i=0}^{\infty}\exp \left( -c_{\gamma} \mu \lambda_{0} i^{1-(\lambda+\gamma)} (1-\varepsilon'(i)\right)\right]\\
\leq& M(\beta)\frac{\log(n+1)}{(n+1)^{\frac{2(1-\gamma)\gamma(\gamma-2\beta)p}{2-\gamma}\wedge 1}}
\end{align*}
with for $n\geq 2$
\begin{align*}
M(\beta)=\frac{2^{p'}C_2^{(p')}}{\mu^{p'}}\left[\tilde{K}_2 \tilde{1}+1+\left\vert a_k\right\vert^{p'}+\tilde{K}_1\sum_{n=0}^{+\infty}\exp \left( -c_{\gamma} \mu \lambda_{0} n^{1-(\lambda+\gamma)} (1-\varepsilon'(n)\right)\right]
\end{align*}
Choosing 
\begin{equation}\label{eq:Adagrad_def_lambda0}
\lambda_0=\left[2^{p'}(C_1^{(p')}+1)\right]^{-\frac{1}{2p'}}
\end{equation}
yields then 
$$\mathbb{P}\left[\left\vert\frac{a_k}{n+1}+\frac{1}{n+1}\sum_{i=0}^{n-1}\partial_k^2 g(\theta_i)\right\vert^{p'}>\frac{\lambda_{0}^{-2p'}}{2^{p'}}\right]\leq \frac{M(\beta)\log(n+1)}{(n+1)^{\frac{2(1-\gamma)\gamma(\gamma-2\beta)p}{2-\gamma}\wedge 1}}.$$
Putting the latter inequality with \eqref{eq:Adagrad_H1a_first_inequality} and \eqref{eq:adagrad_H1a_Mn} gives then 
\begin{align*}
\mathbb{P}\left[\lambda_{\min}\left(\overline{A_n}\right)< \lambda_0\right] & \leq \sum_{k=1}^d\mathbb{P}\left[\left\vert\left(\overline{A_n}\right)_{kk}\right\vert< \lambda_0\right] \\
& \leq \frac{dM(\beta)\log(n+1)}{(n+1)^{\frac{2(1-\gamma)\gamma(\gamma-2\beta)p}{2-\gamma}\wedge 1}}+\frac{d2^{p'}\left(C_1^{(p')}+2^{p'}C_2^{(p')}\frac{ V_{p }^{p'}}{\mu^{p'}}\right)}{(C_1^{(p')}+1)n^{p'/2}}\\
& \leq  \frac{v_0\log(n+1)}{(n+1)^{\frac{2(1-\gamma)\gamma(\gamma-2\beta)p}{2-\gamma}\wedge 1}}
\end{align*}
with 
\begin{equation}\label{eq:Adagrad_v0}
v_0=dM(\beta)+\frac{d2^{p'}\left(C_1^{(p')}+2^{p'}C_2^{(p')}\frac{ V_{p }^{p'}}{\mu^{p'}}\right)}{C_1^{(p')}+1}.
\end{equation}
Since $\mathbb{P}\left[\lambda_{\min}\left(A_n\right)< \lambda_0\right]\leq\mathbb{P}\left[\lambda_{\min}\left(\overline{A_n}\right)< \lambda_0\right]$, the result is deduced.

\subsection{Proof of Lemma \ref{lem:H2_adagrad}}
Set $E_k=\mathbb{E}\left[ \nabla_h\left(X_{},\theta\right)_k^2\right]$ and $\partial_k^2g(h)=\mathbb{E}\left[\nabla_h(X,h)_k^2\right]$. Then
\begin{align*}
\mathbb{E}\left[ \left\vert\left(\overline{A_n}\right)_{kk}^{-2}-E_k\right\vert^{p'}\right]\leq& 2^{p'-1}\mathbb{E}\left[\left\vert \frac{1}{n+1}\sum_{i=0}^{n-1}\nabla_h g\left(X_{i+1},\theta_i\right)_k^2 -\partial_k^2g\left(\theta_i\right) \right\vert^{p'}\right]\\
&+2^{p'-1}\mathbb{E}\left[\left\vert\frac{a_k-E_k}{n+1}+\frac{1}{n+1}\sum_{i=0}^{n-1}(\partial_k^2 g(\theta_i)-E_k)\right\vert^{p'}\right].
\end{align*}
By \eqref{eq:adagrad_bound_martingale},
$$\mathbb{E}\left[\left\vert \frac{1}{n+1}\sum_{i=0}^{n-1}\nabla_hg\left(X_{i+1},\theta_i\right)_k^2 -\partial_k^2g\left(\theta_i\right) \right\vert^{p'}\right]\leq 2^{p'}\frac{C_1^{(p')}+2^{p'}C_2^{(p')}\frac{ V_{p }^{p' }}{\mu^{p'}}}{(n+1)^{p'/2}}.$$
Next, by Jensen inequality,
\begin{align*}
\mathbb{E}\left[\left\vert\frac{a_k-E_k}{n+1}+\frac{1}{n+1}\sum_{i=0}^{n-1}\left( \partial_k^2 g\left(\theta_i\right)-E_k\right)\right\vert^{p'}\right]\leq& \frac{1}{n+1}\left(\left(a_k-E_k\right)^{p'}+\sum_{i=0}^{n-1}\mathbb{E}\left[\left\vert\partial_k^2 g(\theta_i)-E_k\right\vert^{p'}\right]\right).
\end{align*}
Using Cauchy-Schwarz inequality, Assumption \textbf{(A1')} and then Assumption \textbf{(A1)} yields
\begin{align*}
\mathbb{E} & \left[\left\vert\partial_k^2 g\left(\theta_i\right)-E_k\right\vert^{p'}\right]= \mathbb{E}\left[\left\vert\mathbb{E} \left[\nabla_h g\left(\theta_i,X\right)_k^2-\nabla_h g\left(\theta,X\right)_k^2 | \theta_i\right]\right\vert^{p'}\right]\\
\leq& \mathbb{E}\left[\left\vert\mathbb{E} \left[\left(\nabla_h g\left(\theta_i,X\right)_k-\nabla_h g\left(\theta,X\right)_k\right)  \left(\nabla_h g\left(\theta_i,X\right)_k+\nabla_h g\left(\theta,X\right)_k\right)| \theta_i \right]\right\vert^{p'}\right]\\
\leq& \mathbb{E}\bigg[\mathbb{E} \left[\left(\nabla_h g\left(\theta_i,X\right)_k-\nabla_h g\left(\theta,X\right)_k\right)^{2} | \theta_i\right]^{p'/2}  \mathbb{E}\left[\left(\nabla_h g\left(\theta_i,X\right)_k+\nabla_h g\left(\theta,X\right)_k\right)^2 | \theta_i \right]^{p'/2}\bigg]\\
\leq& 2^{p'/2 -1} L_{\nabla g}^{p'/2}\mathbb{E}\left[\Vert \theta_i-\theta\Vert^{p'}(2C_1^{p'/2}+C_2^{p'/2}\Vert \theta_i-\theta\Vert^{p'})\right]\\
\leq & \frac{ 2^{p' }L_{\nabla g}^{p'/2}C_{1}^{p'/2}}{\mu^{p'/2}}\mathbb{E}\left[V_i^{p'/2}\right]+\frac{2^{3p'/2-1} C_2^{p'/2}L_{\nabla g}^{p'/2}}{\mu^{p'}}\mathbb{E}\left[V_i^{p'}\right]\\
\leq& \frac{ 2^{p' }L_{\nabla g}^{p'/2}C_{1}^{p'/2}}{\mu^{p'/2}}\sqrt{c_i}+\frac{2^{3p'/2-1} C_2^{p'/2}L_{\nabla g}^{p'/2}}{\mu^{p'}}c_i,
\end{align*}
where $c_i$ is given in \eqref{eq:bound_Vp'_adagrad}. Putting all the latter bounds together yields, using that $E_k\leq C_1$,
\begin{align*}
\mathbb{E}&\left[\left\vert\frac{a_k-E_k}{n+1}+\frac{1}{n+1}\sum_{i=0}^{n-1}\left( \partial_k^2 g\left(\theta_i\right)-E_k\right)\right\vert^{p'}\right]\\
\leq& \frac{1}{n+1}\left[2^{p'-1}(a_k^{p'}+C_1^{p'})+\sum_{i=0}^{n-1} \left(\frac{ 2^{p' }L_{\nabla g}^{p'/2}C_{1}^{p'/2}}{\mu^{p'/2}}\sqrt{c_i}+\frac{2^{3p'/2-1} C_2^{p'/2}L_{\nabla g}^{p'/2}}{\mu^{p'}}c_i\right)\right].
\end{align*}
Hence, noting that $V_{p}<\infty$ by Assumption \textbf{(A1')} and Lemma \ref{lem::majvn2},
\begin{align*}
&\mathbb{E}\left[ \vert(\overline{A_n})_{kk}^{-2}-E_k\vert^{p'} \right]\\\leq&\underbrace{ 2^{ p'-1}\frac{C_1^{(p')}+2^{p'}C_2^{(p')}\frac{V_p^{p'}}{\mu^{p'}}}{n}+\frac{2^{p'-1}}{n+1}\left[2^{p'-1}(a_k^{p'}+C_1^{p'})+ \sum_{i=0}^{n-1}\left(\frac{ 2^{p' }L_{\nabla g}^{p'/2}C_{1}^{p'/2}}{\mu^{p'/2}}\sqrt{c_i}+\frac{2^{3p'/2 -1} C_2^{p'/2}L_{\nabla g}^{p'/2}}{\mu^{p'}}c_i\right)\right]}_{:=\bar{c}_n},
\end{align*}
with, by \eqref{eq:bound_Vp'_adagrad}, $\bar{c}_n=O\left( \log(n)n^{-\left[\frac{ (1-\gamma)\gamma(\gamma-2\beta)p}{2-\gamma}\wedge 1\right]}\right)$.
Since by \textbf{(A6)} we have $E_k\geq \alpha$, we deduce by Markov's inequality that 
$$\mathbb{P}\left[ \left(\overline{A_n}\right)_{kk}^{-1}\leq \sqrt{\alpha/2}\right]=\mathbb{P}\left[ \left(\overline{A_n}\right)_{kk}^{-2}\leq \alpha/2\right]\leq \frac{2^{p'}}{\alpha^{p'}}\mathbb{E}\left[ \left|\left(\overline{A_n}\right)_{kk}^{-2}-E_k\right|^{p'}\right]\leq \frac{2^{p'}\bar{c}_n}{\alpha^{p'}}.$$
Hence, we have 
\begin{align*}
\mathbb{E}\left[\left(A_n\right)_{kk}^4\right]=&\mathbb{E}\left[\mathbf{1}_{\left(\overline{A_n}\right)_{kk}\geq \sqrt{\frac{2}{\alpha}}}\left(A_n\right)_{kk}^4\right]+\mathbb{E}\left[\mathbf{1}_{\left(\overline{A_n}\right)_{kk}< \sqrt{\frac{2}{\alpha}}}\left(A_n\right)_{kk}^4\right]\\
\leq&\mathbb{E}\left[\mathbf{1}_{\left(\overline{A_n}\right)_{kk}\geq \sqrt{\frac{2}{\alpha}}}c_{\beta}^4n^{4\beta}\right]+\mathbb{E}\left[\mathbf{1}_{\left(\overline{A_n}\right)_{kk}< \sqrt{\frac{2}{\alpha}}}\left(\overline{A_n}\right)_{kk}^4\right]\\
\leq& c_{\beta}^4n^{4\beta}\mathbb{P}\left[\left(\overline{A_n}\right)_{kk}^{-1}\leq \sqrt{\alpha/2}\right]+\frac{4}{\alpha^2}\leq \frac{2^{p'}c_{\beta}^4n^{4\beta}\bar{c}_n}{\alpha^{p'}}+\frac{4}{\alpha^2}.
\end{align*}
Since $\bar{c}_n=O\left( \log (n) n^{-\left[\frac{ (1-\gamma)\gamma(\gamma-2\beta)p}{2-\gamma}\wedge 1\right]}\right)$, for $\beta<\frac{ (1-\gamma)\gamma(\gamma-2\beta)p}{4(2-\gamma)} \wedge \frac{1}{4}$ we have  $\left[\frac{ (1-\gamma)\gamma(\gamma-2\beta)p}{2-\gamma}\wedge 1\right]-4\beta>0$ and thus
$$w(\beta)=\sup_{n\geq 1}\bar{c}_nn^{4\beta}<+\infty,$$
and finally 
$$\mathbb{E}\left[\Vert A_n\Vert^4\right]\leq \sum_{k=1}^d\mathbb{E}\left[(A_n)_{kk}^4\right]\leq C_{S}^4$$
with 
\begin{equation}\label{eq:CS4_adagrad}
C_{S}^4=d\left[\frac{2^{p'}c_{\beta}^4w(\beta)}{\alpha^{p'}}+\frac{4}{\alpha^2}\right].
\end{equation}

\subsection{Proof of Lemma \ref{lem:tech:ada:last:prime}}
First, we have by \textbf{(A6')}
$$\mathbb{E}\left[\left((\overline{A}_n)_{kk}\right)^{-2}\right]=\mathbb{E}\left[ \frac{1}{n+1}\sum_{i=0}^{n-1}\nabla_h g\left(X_{i+1},\theta_i\right)_k^2\right]=\frac{1}{n+1}\sum_{i=0}^{n-1}\mathbb{E}\left[\nabla_h g\left(X_{i+1},\theta_i\right)_k^2\right]\geq \alpha.$$
Then, as in the proof of the previous lemma,
$$\mathbb{E}\left[\left\vert \frac{1}{n+1}\sum_{i=0}^{n-1}\nabla_h g\left(X_{i+1},\theta_i\right)_k^2 -\frac{1}{n+1}\sum_{i=0}^{n-1}\mathbb{E}\left[ \nabla_h g\left(X_{i+1},\theta_i\right)_k^2\right]\right\vert^2\right]\leq \frac{C_1'+C_2'\frac{4V_2^2}{\mu^2}}{n}.$$
Hence, by Markov inequality,
\begin{align*}
\mathbb{P}\left[\left((\overline{A}_n)_{kk}\right)^{-2}\leq \alpha/2\right]\leq& \mathbb{P}\left[\left\vert \frac{1}{n+1}\sum_{i=0}^{n-1}\nabla_h g\left(X_{i+1},\theta_i\right)_k^2 -\frac{1}{n+1}\sum_{i=0}^{n-1}\mathbb{E}\left[ \nabla_h g\left(X_{i+1},\theta_i\right)_k^2\right]\right\vert^2>\frac{\alpha^2}{4}\right]\\
\leq & \frac{4\left(C_1'+C_2'\frac{4V_2^2}{\mu^2}\right)}{n\alpha^2}.
\end{align*}
We deduce as in the previous lemma that 
\begin{align*}
\mathbb{E}\left[\left(A_n\right)_{kk}^4\right]
\leq& c_{\beta}^4n^{4\beta}\mathbb{P}\left[\left(\overline{A_n}\right)_{kk}^{-1}\leq \sqrt{\alpha/2}\right]+\frac{4}{\alpha^2}\leq \frac{4\left(C_1'+C_2'\frac{4V_2^2}{\mu^2}\right)}{n^{1-4\beta}\alpha^2}+\frac{4}{\alpha^2}.
\end{align*}
When $\beta<1/4$, we finally get
$$\mathbb{E}\left[\Vert A_n\Vert^4\right]\leq \sum_{k=1}^d\mathbb{E}\left[(A_n)_{kk}^4\right]\leq C_{S}^4$$
with 
\begin{equation}\label{eq:CS4_adagrad_bis}
C_{S}^4=\frac{4d\left(1+C_1'+C_2'\frac{4V_2^2}{\mu^2}\right)}{\alpha^2}.
\end{equation}

\section{Proof of technical Lemma and Propositions for linear regression}\label{sec::proof:lem::linear}
\subsection{Proof of Lemma \ref{lem:(H1)_regression_first_step}}
Remark that 
$$\left\| \widetilde{S}_n \right\| \leq \frac{1}{n+1}\left( \left\| S_{0} \right\| +  \sum_{i=1}^n \left\|  X_iX_i^T\right\| \right)\leq \frac{1}{n} \left( \left\| S_{0} \right\| + \sum_{i=1}^n\Vert X_i\Vert^2 \right).$$
Hence, for $\lambda>0$,
$$\mathbb{P}\left[ \lambda_{\min}\left( \widetilde{S}_n^{-1}\right)< \lambda\right]=\mathbb{P}\left[ \Vert \widetilde{S}_n\Vert > 1/\lambda\right]\leq \mathbb{P}\left[ \frac{1}{n}\left( \left\| S_{0} \right\| + \sum_{i=1}^n\Vert X_i\Vert^2 \right)>\lambda^{-1} \right].$$
Taking $\lambda_0=\left(2\mathbb{E}\left[ \Vert X\Vert^2 \right]\right)^{-1}$ yields then
$$\mathbb{P}\left[ \lambda_{\min} \left( \widetilde{S}_n^{-1}\right)< \lambda_0\right]\leq \mathbb{P}\left[ \frac{1}{n}\left( \left\| S_{0} \right\| + \sum_{i=1}^n\left( \Vert X_i\Vert^2-\mathbb{E}\left[ \Vert X\Vert^2\right]\right)\right)>\mathbb{E}\left[ \Vert X\Vert^2 \right]\right].$$
Taking the $p$-power, applying Markov inequality and then Rosenthal inequality yields that
\begin{align*}
\mathbb{P} & \left[ \frac{1}{n}\left( \left\| S_{0} \right\| + \sum_{i=1}^n\left( \Vert X_i\Vert^2-\mathbb{E}\left[ \Vert X\Vert^2\right]\right)\right)>\mathbb{E}\left[ \Vert X\Vert^2\right]\right] \\
& \leq \mathbb{P}\left[\left(\frac{1}{n}\left( \left\| S_{0} \right\| + \left| \sum_{i=1}^n\left(\Vert X_i\Vert^2-\mathbb{E}\left[ \Vert X\Vert^2\right]\right)\right|\right)\right)^p>\left(\mathbb{E}\left[ \Vert X\Vert^2\right]\right)^p\right]\\
& \leq \frac{1}{\left(\mathbb{E}\left[ \Vert X\Vert^2\right]\right)^p} \mathbb{E}\left[\frac{1}{n^{p}} \left( \left\| S_{0} \right\| + \left| \sum_{i=1}^n\left(\Vert X_i\Vert^2-\mathbb{E}\left[ \Vert X\Vert^2\right] \right) \right| \right)^p\right]\\
& \leq \frac{2^{p-1}}{\left(\mathbb{E}\left[\Vert X\Vert^2 \right]\right)^p}\left(C_1(p)n^{1-p}\mathbb{E}\left[ \vert Z\vert^p \right]+C_2(p)n^{-p/2}\left(\mathbb{E}\left[ \vert Z\vert^2 \right]\right)^{p/2} + \left\| S_{0} \right\|^{p}n^{-p}\right),
\end{align*}
with $Z=\Vert X\Vert^2-\mathbb{E}\left[ \Vert X\Vert^2\right]$.

\subsection{Proof of Lemma \ref{lem:(H1)_regression}}
By definition of $\overline{S}_n$, $\overline{S}_n=\widetilde{S}_n$ on the event $T_n=\{\lambda_{\min}\left(\widetilde{S}_n\right)\geq \frac{1}{c_\beta n^{\beta}}\}$. Hence, for the same $\lambda_0$ as in Lemma \ref{lem:(H1)_regression_first_step},
\begin{align}
\mathbb{P}\left[ \lambda_{\min}\left(\bar{S}_n^{-1}\right)<\lambda_0 \right]&=\mathbb{P}\left[ T_n\cap\left\{\lambda_{\min}\left(\widetilde{S}_n^{-1}\right)<\lambda_0\right\}\right]+\mathbb{P}\left[ T_n^c  \right]\nonumber\\
&\leq \mathbb{P}\left[ \lambda_{\min}\left(\widetilde{S}_n^{-1}\right)<\lambda_0 \right]+\mathbb{P}\left[ T_n^c \right].\label{eq:lem_H1_first_inequality}
\end{align}
By Lemma \ref{lem:(H1)_regression_first_step}, 
\begin{equation}\label{eq:lem_H1_second_inequality}
\mathbb{P}\left[ \lambda_{\min}\left(\widetilde{S}_n^{-1}\right)<\lambda_0 \right]\leq \tilde{v}_n,
\end{equation}
with $\tilde{v}_n$ given in Lemma \ref{lem:(H1)_regression_first_step}. Then, for $n\geq n_0 $, where $n_0$ is defined in \eqref{eq:threshold_linear}, we have 
$n\geq \left(\frac{1}{c_\beta c_2}\left(\frac{n+1}{n}\right)\right)^{-1/\beta}$, and thus $\frac{n}{n+1}c_2\geq \frac{1}{c_\beta n^{\beta}}$. In particular, on the event $\left\{\lambda_{\min}\left(\frac{1}{n}\sum_{i=1}^nX_iX_i^T\right)>c_2 \right\}$, we have 
\begin{align*}
\lambda_{\min}\left(\widetilde{S}_n\right)=&\lambda_{\min}\left(\frac{1}{n+1}\left(S_0+\sum_{i=1}^nX_iX_i^T\right)\right)\\
\geq& \frac{n}{n+1}\lambda_{\min}\left(\frac{1}{n}\sum_{i=1}^nX_iX_i^T\right)>\frac{n}{n+1}c_2\geq \frac{1}{c_{\beta}n^{\beta}}.
\end{align*}
Hence, for $n\geq n_0$, $\left\{\lambda_{\min}\left(\frac{1}{n}\sum_{i=1}^nX_iX_i^T\right)>c_2\right\}\subset T_n$ and thus by Proposition \ref{Mendelson_result} and the fact that $n\geq c_1d$,
\begin{equation}\label{eq:lem_H1_third_inequality}
\mathbb{P}\left[ T_n^c\right]\leq \mathbb{P}\left[\lambda_{\min}\left(\frac{1}{n}\sum_{i=1}^nX_iX_i^T\right)<c_2 \right]\leq \exp(-c_3n).
\end{equation}
Using \eqref{eq:lem_H1_second_inequality} and \eqref{eq:lem_H1_third_inequality} in \eqref{eq:lem_H1_first_inequality} yields then
$$\mathbb{P}\left[\lambda_{\min}\left(\overline{S}_n^{-1}\right)<\lambda_0\right]\leq \tilde{v}_n+2\exp(-c_3n)$$
for $n\geq n_0$.
The statement of the lemma is then a rewriting of the latter inequality.

\subsection{Proof of Lemma \ref{lem::jesaismemepascequecest}}
Since we have 
$$\Vert \bar{S}_{n}^{-1}\Vert=\min\left\lbrace \Vert \tilde{S}_{n}^{-1}\Vert,\beta_{n+1}\right\rbrace=\min\left\lbrace\frac{1}{\lambda_{\min}\left\lbrace \tilde{S}_{n}\right)},\beta_{n+1}\right\rbrace,$$
for $c_1,c_2,c_3$ given in Proposition \ref{Mendelson_result}, $n\geq c_1d$ and $\kappa>0$,
$$\mathbb{E}\left[\Vert \bar{S}_{n}^{-1}\Vert^{\kappa}\right]\leq \beta_{n+1}^\kappa\mathbb{P}\left[ \lambda_{\min} \left(\tilde{S}_{n}\right)\leq c_2\right]+c_2^{-\kappa}\leq 2\beta_{n+1}^\kappa\exp\left(-c_3n\right)+c_2^{-\kappa}.$$
Since $\tilde{S}_n=\frac{1}{n+1}\left(S_0+\sum_{i=1}^n X_iX_i^T \right)$ and $\sum_{i=1}^n X_iX_i^T\geq 0$, we have $\tilde{S}_n\geq \frac{1}{n+1}S_0$ and thus $\Vert \bar{S}_n^{-1}\Vert\leq \Vert \tilde{S}_n^{-1}\Vert\leq (n+1)\Vert S_0^{-1}\Vert$ for $n\geq 1$. Hence, for $n\leq c_1d$,  $\Vert \tilde{S}_{n}^{-1}\Vert\leq (c_1 d+1)\Vert S_0^{-1}\Vert$ and we finally get the result.

\subsection{Proof of Proposition \ref{prop::H2}}
Recall that $\beta_{n}=c_{\beta}n^\beta$. Since,for $\kappa>0$, the map $g:t\mapsto (c_{\beta}t^\beta)^\kappa\exp(-c_3t)$ is bounded from above by $c_{\beta}^\kappa\left(\frac{\beta\kappa}{ec_3}\right)^{\beta\kappa}$, we get
$$\sup_{n\geq c_1 d}\mathbb{E}\left[\Vert \bar{S}_{n}^{-1}\Vert^{\kappa}\right]\leq 2  c_{\beta}^\kappa\left(\frac{\beta\kappa}{ec_3}\right)^{\beta\kappa}+c_2^{-\kappa}.$$
Taking into account the case $n\leq c_1d$ yields then
$$\sup_{n\geq 1}\mathbb{E}\left[\Vert \bar{S}_{n}^{-1}\Vert^{2}\right]\leq \max\left\lbrace2c_\beta^2\left(\frac{2\beta}{ec_3}\right)^{2\beta}+c_2^{-2}, \left[(c_1 d+1)\Vert S_0^{-1}\Vert\right]^2\right\rbrace ,$$ 
and
$$\sup_{n\geq 1}\mathbb{E}\left[\Vert \bar{S}_{n}^{-1}\Vert^{4}\right]\leq \max\left\lbrace 2c_\beta^4\left(\frac{4\beta}{ec_3}\right)^{4\beta}+c_2^{-4}, \left[\left(c_1 d+1\right)\Vert S_0^{-1}\Vert\right]^4\right\rbrace .$$

\subsection{Proof of Lemma \ref{lem:Hypothesis_H3_regression}}
First notice that 
$$\left\|  \overline{S}_n^{-1}-H^{-1} \right\|= \left\| \overline{S}_n^{-1}(H-\overline{S}_n)H^{-1}\right\|\leq \left\|  \overline{S}_n^{-1}\right\| \left\|  H-\overline{S}_n \right\| \left\|  H^{-1} \right\|.$$
Under hypothesis of  Proposition \ref{Mendelson_result}, 
$$\mathbb{P}\left[ \lambda_{\min} \left( \tilde{S}_{n} \right)\leq c_2 \right]\leq \exp\left(-c_3n\right)$$
for $n\geq c_1 d$. Since $\left\|  \bar{S}_n^{-1}\right\|\leq \left\|  \tilde{S}_n^{-1}\right\|$, $\lambda_{\min} \left( \bar{S}_n \right) \geq \lambda_{\min} \left( \tilde{S}_n \right)$ and thus we also have
$$\mathbb{P}\left[ \lambda_{\min} \left( \bar{S}_{n} \right)\leq c_2 \right]\leq \exp\left(-c_3n\right)$$
for $n\geq c_1 d$. Hence,  for $n\geq n_0$,
\begin{align*}
\mathbb{E}\left[ \left\|  \overline{S}_n^{-1}-H^{-1}\right\|^2\right]=&\mathbb{E}\left[\mathbf{1}_{\lambda_{\min}\left( \overline{S}_{n} \right)\leq c_2} \left\|  \overline{S}_n^{-1}-H^{-1}\right\|^2\right] +\mathbb{E}\left[\mathbf{1}_{\lambda_{\min} \left( \overline{S}_{n} \right)>c_2}\left\|  \overline{S}_n^{-1}-H^{-1}\right\|^2\right]\\
\leq &\frac{1}{\left(\lambda_{\min}\beta_n\right)^2}\mathbb{E}\left[\mathbf{1}_{\lambda_{\min} \left( \overline{S}_{n}\right)\leq c_2}\left\|  \overline{S}_n-H\right\|^2\right] +\frac{1}{ \left( \lambda_{\min}c_2 \right)^{2}}\mathbb{E}\left[ \left\|  \tilde{S}_n-H \right\|^2\right],
\end{align*}
where we used on the last equality that  for $n\geq n_0$, $\bar{S}_n=\tilde{S}_n$ on the event $\{\lambda_{\min}(\bar{S}_n>c_2)\}$, as in the proof of Lemma \ref{lem:(H1)_regression}. 
The first summand can be bounded using Hölder inequality with $\frac{1}{q}+\frac{1}{q'}=1$ and $q'=p/2$ as
\begin{align*}
\mathbb{E}\left[\mathbf{1}_{\lambda_{\min } \left( \overline{S}_{n} \right)\leq c_2}\left\| \overline{S}_n-H\right\|^2\right]\leq &\mathbb{P}\left[ \lambda_{\min} \left( \overline{S}_{n} \right)\leq c_2\right]^{1/q}\mathbb{E}\left[ \left\|  \overline{S}_n-H\right\|^{2q'}\right]^{1/q'}\\
\leq& \exp(-c_3(p-2)n/p)\mathbb{E}\left[ \left\|  \overline{S}_n-H \right\|^{p}\right]^{2/p}.
\end{align*}
Using the upper bound on $H$ and the convexity inequality $(a+b)^{p}\leq 2^{p-1}(a^{p}+b^{p})$ yields the rough bound
\begin{align*}
\mathbb{E}\left[ \left\| \overline{S}_n-H\right\|^{p}\right]^{2/p}\leq \mathbb{E}\left[ \left( \left\|  \overline{S}_n\right\|+ \left\|  H\right\| \right)^{p}\right]^{2/p}\leq& 2^{2-2/p}\left(\mathbb{E}\left[ \left\|  \overline{S}_n\right\|^{p} \right] +\lambda_{\max}^{p}\right)^{2/p}\\
\leq&4\max\left\lbrace \lambda_{\max}^2, \mathbb{E} \left[\left\| \overline{S}_{n}\right\|^{p} \right]^\frac{2}{p}\right\rbrace 
\end{align*}
Since $X$ admits moments of order $2p$, we get
$$ \mathbb{E}\left[  \left\| \overline{S}_{n} \right\|^{p} \right] \leq  \mathbb{E}\left[ \left( \frac{1}{n}\sum_{i=1}^n\left\|  X_i \right\|^2\right)^{p} \right] \leq \left(\mathbb{E} \left[  \Vert X\Vert^{2p} \right]\right)^{1/2}.$$
We hence get 
\begin{align*}
\mathbb{E}\left[\mathbf{1}_{\lambda_{\min} \left( \overline{S}_{n} \right)\leq c_2} \left\|  \overline{S}_n-H\right\|^2\right] & \leq 4\exp\left(-c_3(p-2)n/p\right)\max\left\lbrace \lambda_{\max}^2,\left(\mathbb{E} \left[ \Vert X\Vert^{2p} \right]\right)^{2/p}\right\rbrace \\
 & = 4\exp\left(-c_3(p-2)n/p\right)  \left(\mathbb{E} \left[ \Vert X\Vert^{2p} \right]\right)^{2/p}
\end{align*}
For the second summand, using the relation between Frobenius norm and operator norm yields
%\begin{align*}\mathbb{E}\left[ \left\|  \overline{S}_n-H \right\|^2\right]\leq \mathbb{E}\left[ \left\| \overline{S}_n-H\right\|_F^2\right]=&\mathbb{E}\left[\sum_{i,j=1}^d\left| \left(\overline{S}_{n}\right)_{ij}-H_{ij}\right|^2\right].
%\end{align*}
\begin{align*}\mathbb{E}\left[ \left\|  \overline{S}_n-H \right\|^2\right]& \leq \mathbb{E}\left[ \left\| \overline{S}_n-H\right\|_F^2\right] \\
& \leq  \frac{2}{(n+1)^{2}} \left\| S_{0} - H \right\|_{F}^{2} + \frac{2}{(n+1)^{2}} \mathbb{E}\left[ \left\| \sum_{k=1}^{n} \left(  X_{k}X_{k}^{T} - \mathbb{E}\left[ XX^{T} \right] \right) \right\|_{F}^{2} \right] \\
& =  \frac{2}{(n+1)^{2}} \left\| S_{0} - H \right\|_{F}^{2} + \frac{2}{n+1} \mathbb{E}\left[ \left\| X X^{T} - \mathbb{E}\left[ XX^{T} \right] \right\|_{F}^{2} \right] \\
& \leq  \frac{2}{(n+1)^{2}} \left\| S_{0} - H \right\|_{F}^{2} + \frac{2}{n+1} \mathbb{E}\left[ \left\| X \right\|^{4} \right] .
\end{align*}
Putting all the above bounds together yields the bound of the statement.

\section{Proof of technical Lemma and Propositions for generalized linear model}\label{sec::proof::tec::lm}

\subsection{Proof of Lemma \ref{lem::GLM}}
With the help of inequality \eqref{upperbound_glm}, it comes
$$\Vert \bar{S}_n\Vert\leq \frac{1}{n+1}\left\| S_{0} \right\| +  \frac{L_{\nabla l}}{n+1}\sum_{i=1}^n\left\|  X_i \right\|^2+\frac{\sigma d}{n+1}\sum_{i=1}^n\Vert Z_i\Vert^2.$$
with $Z_{i} = e_{i[d]+1}$. Then, a similar proof as the one of Lemma \ref{lem:(H1)_regression} yields that for $\lambda_0=\left(2L_{\nabla l}\mathbb{E}\left[ \left\|  X\right\|^2 \right] +2\sigma\right)^{-1}$,
\begin{align*}\mathbb{P} & \left[ \lambda_{\min} \left( \overline{S}_n^{-1} \right) < \lambda_0 \right] \\
& \leq \mathbb{P}\left[ \frac{\left\| S_{0} \right\|}{n}  + \frac{L_{\nabla l}}{n}\sum_{i=1}^n \left( \left\|  X_i\right\|^2-\mathbb{E}\left[ \left\|  X\right\|^2 \right] \right)+\frac{\sigma}{n}\sum_{i=1}^n\left( \left\|  Z_i \right\|^2-1\right)>L_{\nabla l}\mathbb{E}\left[ \Vert X\Vert^2 \right] +\sigma\right].
\end{align*}
Then, by Markov inequality for $p\geq 1$, we then get
\begin{align*}
\mathbb{P} & \left[ \lambda_{\min}\left(\overline{S}_n^{-1} \right)< \lambda_0 \right] \leq\frac{\mathbb{E}\left[ \left(  \frac{1}{n}\left\| S_{0} \right\| + \frac{1}{n}\sum_{i=1}^nL_{\nabla l}\left(\Vert X_i\Vert^2-\mathbb{E}\left[ \left\|  X\right\|^2 \right]\right)+\sigma\left(\Vert Z_i\Vert^2-1\right)\right)^{p}\right]}{\left(L_{\nabla l}\mathbb{E}\left[ \Vert X\Vert_2^2 \right]+\sigma\right)^p}\\
\leq&\frac{2^{p-1}}{\left(   L_{\nabla l}\mathbb{E}\left[ \left\|  X\right\|^2\right]+\sigma\right)^p}\left( n^{-p} \left\| S_{0} \right\|^{p} + C_1(p)n^{1-p}\mathbb{E}\left[ \vert T\vert^p\right]+C_2(p)n^{-p/2}\left(\mathbb{E}\left[ \Vert T\Vert^2 \right] \right)^{p/2}\right),
\end{align*}
with $T=L_{\nabla l}\left(\Vert X\Vert^2-\mathbb{E}\left[ \Vert X\Vert^2 \right]\right)+\sigma\left(\Vert Z\Vert^2-1\right)$.

\subsection{Proof of Proposition \ref{lem::GLM::H2}}
One directly has for all $n \geq 2d$
$$\lambda_{\min}\left( \overline{S}_n \right)\geq \frac{\lfloor n/d\rfloor\sigma}{(n+1)}\geq \frac{n+1 -d}{d(n+1)}\sigma\geq \frac{1}{2d} \sigma ,$$
and $\overline{S}_n\geq \frac{1}{2d}S_0$ for $n\leq 2d-1$, so that 
\[
\sup_{n\geq 1} \left\|  \bar{S}_{n}^{-1}\right\| \leq 2d \max\left\lbrace\frac{1}{\sigma},\left\|  S_0^{-1}\right\|\right\rbrace.
\]

  \subsection{Proof of Proposition \ref{Hypothesis (H3)}}
Let us denote
$$H \left(\theta_{\sigma} \right) =\mathbb{E}\left[\nabla_{h}^{2}\ell\left(Y,\theta_{\sigma}^{T}X \right)XX^T\right] \quad \quad \text{and} \quad \quad \overline{H}_{n}=\frac{1}{n+1} \left( S_{0} + \sum_{i=0}^{n-1}\nabla_{h}^{2}\ell\left( Y_{i+1},\langle \theta_i,X_{i+1}\rangle \right)X_{i+1}X_{i+1}^T \right).$$
One can decompose $\overline{H}_{n}-H\left( \theta_{\sigma} \right)$ as 
\begin{align*}
\overline{H}_{n}-H\left(\theta_{\sigma} \right)=&\frac{1}{n+1}\sum_{i=0}^{n-1}\nabla_{h}^{2}\ell\left(Y_{i+1},\langle \theta_i,X_{i+1}\rangle \right)X_{i+1}X_{i+1}^T + \frac{1}{n+1}S_{0}-H\left(\theta_{\sigma} \right)\\
= &\frac{1}{n+1}\sum_{i=0}^{n-1}\nabla_{h}^{2}\ell\left( Y_{i+1},\left\langle \theta_i,X_{i+1}\right\rangle\right)X_{i+1}X_{i+1}^T-H\left(\theta_i\right)\\
+&\frac{1}{n+1}\sum_{i=0}^{n-1}\left(H\left(\theta_i\right)- H\left(\theta_{\sigma} \right)\right) +\frac{1}{n+1} \left( S_{0} -  H\left(\theta_{\sigma} \right)  \right).
\end{align*}
Let us now give a rate of convergence of each term on the right-hand side of previous equality.
Set $M_n:=\sum_{i=0}^{n-1}\left(\nabla_{h}^{2} \ell \left(Y_{i+1},  \theta_i^{T}X_{i+1}\right)X_{i+1}X_{i+1}^T- H (\theta_i)\right)$. Since 
$\mathbb{E}\left[ \nabla_{h}^{2}\ell \left(Y_{i+1}, \theta_i^{T}X_{i+1} \right)X_{i+1}X_{i+1}^T\vert \mathcal{F}_{i} \right] = H\left(\theta_i\right),$ where $\left( \mathcal{F}_{i} \right)$ is the $\sigma$-algebra generated by the sample, i.e $\mathcal{F}_{i}:= \sigma \left( \left( X_{1} , Y_{1} \right) , \ldots , \left( X_{i} , Y_{i} \right) \right)$. Then, 
$(M_n)_{n\geq 1}$ is a martingale and thus
\begin{align*}
\frac{1}{(n+1)^{2}} \mathbb{E}\left[ \left\| M_{n} \right\|^{2} \right] \leq \frac{1}{(n+1)^{2}} \sum_{i=0}^{n-1} \mathbb{E}\left[ \left\| \left(\nabla_{h}^{2} \ell \left(Y_{i+1},  \theta_i^{T}X_{i+1}\right)X_{i+1}X_{i+1}^T- H (\theta_i)\right) \right\|^{2} \right] \leq \frac{L_{\nabla l}^{2} \mathbb{E}\left[ \left\| X \right\|^{4} \right]}{n}
\end{align*}
It then remains to handle $\frac{1}{n+1}\sum_{i=0}^{n-1}\left(H(\theta_i)-H\left(\theta_{\sigma}\right)\right)$. With the help of Assumption \textbf{(GLM1)}, one has
\begin{align*}
\mathbb{E}\left[ \left\Vert \frac{1}{n+1}\sum_{i=0}^{n-1}\left(H  (\theta_i)-H\left(\theta_{\sigma}\right)\right)\right\Vert^2\right]\leq& \frac{1}{n} \sum_{i=0}^{n-1}\mathbb{E}\left[\left\Vert H(\theta_i)- H\left(\theta_{\sigma} \right)\right\Vert^2 \right]\\
\leq& \frac{L_{\nabla^{2}L}^2}{n}\sum_{i=0}^{n-1}\mathbb{E}\left[ \left\Vert \theta_i-\theta_{\sigma} \right\Vert^2\right]\leq \frac{L_{\nabla^{2} L}^2}{\sigma n}\sum_{i=0}^{n-1}v_{i,\text{GLM}},
\end{align*}
with $v_{i,\text{GLM}}$ defined in Proposition \ref{prop::glm}.  %Hence, we finally get
%\mathbb{E}\left[ \left\Vert S_n-H\left(\theta_{\sigma}\right)\right\Vert^2\right] \leq \frac{3}{n}\left( C_{(4)}d^2L_{\nabla l}^2\mathbb{E}\left[ \left\| X \right\|^{4} \right]+\frac{L_{\nabla^{2}L}^2}{\sigma  }\sum_{i=0}^{n-1}v_{i,\text{GLM}}+\frac{1}{n }\Vert S_0-H\left(\theta_{\sigma}\right)\Vert^{2}\right).$$
Then, since
\[
\left\|  \frac{d\sigma}{n}\sum_{i=1}^ne_{i[d]+1}e_{i[d]+1}^{T} - \sigma I_d \right\|^2 = \left\|  \frac{d\sigma}{n}\sum_{i=d\lfloor\frac{n}{d}\rfloor}^n e_{i[d]+1}e_{i[d]+1}^{T} + \left(\frac{d\sigma}{n}\left\lfloor\frac{n}{d}\right\rfloor- \sigma\right) I_p \right\|^2
\]
and
\[
\frac{d^{2}\sigma^{2}}{n^{2}}\left\|  \sum_{k=d\lfloor\frac{n}{d}\rfloor}^ne_{i[d]+1}e_{i[d]+1}^{T} \right\|^2 \leq \frac{d^{2}\sigma^{2}}{n^{2}}\left( n-d\left\lfloor\frac{n}{d}\right\rfloor \right)\sum_{k=d\lfloor\frac{n}{d}\rfloor}^n \left\|  e_{i[d]+1}e_{i[d]+1}^{T} \right\|^2 \leq \frac{d^4\sigma^2}{n^2},
\]
it comes 
\begin{align*}
\left\| \overline{S}_{n} - H_{\sigma} \right\|^{2} \leq \frac{4}{n}\left(  L_{\nabla l}^2\mathbb{E}\left[ \left\| X \right\|^{4} \right]+\frac{L_{\nabla^{2}L}^2}{\sigma  }\sum_{i=0}^{n-1}v_{i,\text{GLM}}+\frac{1}{n }\Vert S_0-H\left(\theta_{\sigma}\right)\Vert^{2}\right) + \frac{16d^{4}\sigma^{2}}{n^{2}}
\end{align*}
Now, notice as in Lemma \ref{lem:Hypothesis_H3_regression} that
$$\left\| \overline{S}_n^{-1}-H_{\sigma}^{-1}\right\|=\left\|  \overline{S}_n^{-1}(H_{\sigma}-\overline{S}_n)H_{\sigma}^{-1}\right\|\leq \left\|  \overline{S}_n^{-1}\right\|\left\| H_{\sigma}-\overline{S}_n\right\| \left\|  H_{\sigma}^{-1}\right\|,$$
which yields, thanks to Proposition \ref{lem::GLM::H2}
\begin{align*}
\mathbb{E}\left[  \left\| \overline{S}_n^{-1}-H_{\sigma}^{-1}\right\|^{2} \right] \leq \frac{C_{S,\sigma}^{2}}{\sigma^{2}} \mathbb{E}\left[ \left\| \overline{S}_{n} - H_{\sigma} \right\|^{2} \right] ,
\end{align*}
i.e one has
\begin{equation}
\mathbb{E}\left[ \left\| \overline{S}_{n}^{-1} - H_{\sigma}^{-1} \right\|^{2} \right] \leq \frac{4C_{S,\sigma}^{2}}{\sigma^{2}n}\left(  L_{\nabla l}^2\mathbb{E}\left[ \left\| X \right\|^{4} \right]+\frac{L_{\nabla^{2}L}^2}{\sigma   }\sum_{i=0}^{n-1}v_{i,\text{GLM}}+\frac{1}{n }\Vert S_0-H\left(\theta_{\sigma}\right)\Vert^{2}\right) + \frac{16d^{4}C_{S,\sigma}^{2}}{n^{2}}.
\end{equation}

\section{How to verify \textbf{(GLM3)} for the logistic regression}\label{sec::alpha::log}

Remark that $\theta_{\sigma}$ is the unique solution to $\mathbb{E}\left[\nabla_{h}\ell\left(Y,X^T\theta_{\sigma}\right)X+\sigma\theta_{\sigma}\right]=0$, so that 
$$ \mathbb{E}\left[ \left\vert \nabla_{h}l \left( Y , X^{T}\theta_{\sigma} \right)X_k + \sigma (\theta_{\sigma})_k \right\vert^{2}  \right]=Var\left[  \nabla_{h}l \left( Y , X^{T}\theta_{\sigma} \right)X_k\right].$$
For the logistic regression, we have $Y\in\{-1,1\}$ and $\nabla_{h}l \left( Y , X^{T}\theta_{\sigma} \right)=\frac{-Y}{1+\exp(-Y\theta_{\sigma}^TX)}$, and thus we need to get a lower bound on the variance of $\frac{-YX_k}{1+\exp(-Y\theta_{\sigma}^TX)}$ for all $1\leq k\leq d$. To guarantee Assumptions \textbf{(GLM3)}, we impose a minimal randomness on $(X,Y)$ given by the existence for all $1\leq k'\leq d$ of $x^{k'}$ $\sigma(Y,X_i,i\not=k')$ measurable bounded by $M$ and an event $A\in \sigma(Y,X_i,i\not=k')$ with $\mathbb{P}\left[ A\cap\{\vert X_i\vert \leq M, 1\leq i\leq d, i\not=k'\} \right]>\eta$ and $c,\epsilon>0$ such that on $A$ we have
$$\mathbb{P}\left[ X_{k'}>x^{k'}+c\vert Y,X_i, i\not=k' \right]> \epsilon\quad \text{ and }\quad\mathbb{P}\left[ X_{k'}<x^{k'}-c\vert Y,X_i, i\not=k' \right]>\epsilon.$$
In particular, since $u\mapsto\frac{u}{1+\exp(-\alpha u)}$ is monotonic for all $\alpha\in\mathbb{R}$ and $\mathbb{C}^1$, there is $y_k$ $\sigma(Y,X_i,i\not=k)$ measurable and $c(M\Vert\theta_{\sigma}\Vert)$ explicitly depending on $M\Vert\theta_{\sigma}\Vert$ such that on $B:=A\cap \{\vert X_i\vert \leq M, 1\leq i\leq d, i\not=k\}$,
$$\mathbb{P}\left[\frac{-YX_k}{1+\exp(-Y\theta_{\sigma}^TX)}>y_k+c(M\Vert\theta_{\sigma}\Vert)\vert Y,X_i, i\not=k\right]> \epsilon,$$
and
$$\mathbb{P}\left[\frac{-YX_k}{1+\exp(-Y\theta_{\sigma}^TX)}<y_k-c(M\Vert\theta_{\sigma}\Vert)\vert Y,X_i, i\not=k\right]>\epsilon.$$
We deduce that on the event $B$ we have
$$Var\left[\frac{-YX_k}{1+\exp(-Y\theta_{\sigma}^TX)}\vert Y,X_i, i\not=k\right]\geq 2\epsilon c(M\Vert\theta_{\sigma}\Vert)^2.$$
Hence,
$$Var\left[\frac{-YX_k}{1+\exp(-Y\theta_{\sigma}^TX)}\right]\geq \mathbb{E}\left[\mathbf{1}_{B}Var\left[\frac{-YX_k}{1+\exp(-Y\theta_{\sigma}^TX)}\vert Y,X_i, i\not=k\right]\right]\geq 2\eta\epsilon c(M\Vert\theta_{\sigma}\Vert)^2,$$
and we can choose 
$$\alpha_{\sigma}=2\eta\epsilon c(M\Vert\theta_{\sigma}\Vert)^2.$$

\section{Counter-example for the quadratic convergence of the stochastic Newton algorithm without regularization}\label{sec::counter}
We show here that even in the simplest case $d=1$, stochastic Newton algorithm may not converge in quadratic mean. Suppose that we define here the naive Newton adaptive matrix $A_n$
$$A_n=\left[\frac{1}{n+1}\left(Id+\sum_{i=0}^{n-1}\nabla_h^2g(X_{i+1},\theta_{i})\right)\right]^{-1}.$$
Recall that is known \citep{BGB2020} that $\theta_n$ converges almost-surely to the minimizer $\theta_0$ at speed $n^{-\gamma}$ for $\gamma\in (1/2,1)$.
\subsection*{Counter-example with $\nabla g$ almost everywhere defined}
Set $g((x,y),\theta)=(x\theta)^2+y\lfloor \theta\rfloor \theta$ and let $(X,Y)$ be a random vector with independent coordinates such that $X\simeq Ber(1/2)$ and $\mathbb{P}[Y=1]=\mathbb{P}[Y=-1]=1/2$. Then, $G(\theta)=\mathbb{E}\left[X^2 \right]\theta^2+\mathbb{E}[Y]\lfloor \theta\rfloor \theta=\theta^2/2$ and we have Lebesgue almost surely $\nabla_{h} g((x,y),h)=2x^2h+y\lfloor h\rfloor$ and $\nabla^2_{h} g((x,y),h)=2x^2$.

Let $n\geq 1$. Then, $\mathbb{P}\left[ X_1=0,\ldots,X_n=0,Y_1=-1,\ldots,Y_n=-1 \right]=2^{-2n}$ and on the event $\{X_1=0,\ldots,X_n=0,Y_1=-1,\ldots,Y_n=-1\}$, as long as $\theta_k\not\in \mathbb{N}$ for all $k\geq 0$ (which will be temporarily assumed), 
$$A_{k}^{-1}=\frac{1}{k}\left(1+\sum_{i=0}^{k-1}2X_{i+1}^2\right)=\frac{1}{k}.$$
Hence, $A_k=k$ and $(\theta_k)_{1\leq k\leq n}$ is defined recursively by 
$$\theta_k=\theta_{k-1}-\gamma_kA_k\lfloor \theta_{k-1}\rfloor Y_k=\theta_{k-1}+k\gamma_k\lfloor \theta_{k-1}\rfloor.$$
If $\gamma_k=k^{-\alpha}$ for some $\alpha<1$, we then have $k\gamma_k=k^{1-\alpha}$, and thus for $\theta_0>1$
$$\theta_k\geq (1+k^{1-\alpha}/2)\theta_{k-1}.$$
We deduce that $\theta_n\geq \prod_{k=1}^n(1+k^{1-\alpha}/2)\geq (n !)^{1-\alpha}2^{-n}$. In particular,
$$\mathbb{E}\left[ \Vert \theta_n-\theta_0\Vert^2 \right] \geq 2^{-3n}(n !)^{1-\alpha}\xrightarrow[]{n\rightarrow \infty}\infty$$
when $\theta_k\not\in\mathbb{N}$ for all $k\geq 0$. Since for each $k\geq 1$, $\theta_k\not \in\mathbb{N}$ for almost every $\theta_0\in (1,2]$, the latter hypothesis holds for Lebesgue almost every choice of $\theta_0\in]1,2]$.

\subsection*{Counter-example with $\nabla g$ continuous}
Let $f$ be such that $f''(\theta)=\mathbf{1}_{\mathbb{Z}+]-1/3,1/3[}$, and set $g((x,y),\theta)=(x\theta)^2+yf(\theta)$. Let $(X,Y)$ be a random vector with independent coordinates satisfying $X\simeq Ber(1/2)$ and $Y\sim \mathcal{U}([-2,2])$. Then, $G(\theta)=\mathbb{E}\left[ X^2 \right]\theta^2+ \mathbb{E}[Y] f(\theta)=\theta^2/2$ and $\nabla_{h} g ((x,y),\theta)=2x^2\theta+yf'(\theta)$. Then, $A_0=1$,
$$A_{k}^{-1}=\frac{k}{k+1}(A_{k-1}^{-1}+2X_k^2+f''(\theta_{k-1})Y_k)$$
for $k\geq 1$ and
$$\theta_n=\theta_{n-1}-A_{n-1}\gamma_n\nabla_{h} g((X_n,Y_n),\theta_{n-1})=\theta_{n-1}-A_{n-1}\gamma_n(2X_n^2\theta_{n-1}+Y_nf'(\theta_{n-1})).$$
Set $\theta_0=3/2$ and $\gamma_{k}=k^{-\gamma}$ for $k\geq 1$, and consider $(X_i,Y_i)_{0\leq i\leq n}$ satisfying the following conditions:
\begin{itemize}
\item  $X_i=0$ for all $1\leq i\leq n$, which yields  $\mathbb{P}\left[ X_1=0,\ldots,X_n=0 \right]=2^{-n}$ and for all $k\geq 1$,
$$\theta_k=\theta_{k-1}-A_{k-1}\gamma_kY_kf'(\theta_{k-1}).$$
\item $\theta_{k-1}$ being known, $Y_k\in \frac{1}{\gamma_kA_{k-1}f'(\theta_{k-1})}\left(\left(\mathbb{Z}+]1/3,2/3[\right)-\theta_{k-1}\right)\cap [-2,-1]:=T_k$ (remark that $T_k$ will be shown to be non-empty).
\end{itemize}
\begin{lem}
The following facts hold for $k\geq 1$.
\begin{enumerate}
\item $\theta_k\geq k+1$,
\item $A_{k}=k+1$,
\item with $\ell$ denoting the Lebesgue measure, 
$$\ell\left(\frac{1}{\gamma_kA_{k-1}f'(\theta_{k-1})}\left(\left(\mathbb{Z}+]1/3,2/3[\right)-\theta_{k-1}\right)\cap [-2,-1])\right)\geq 1/6.$$
\end{enumerate}
\end{lem}
\begin{proof}
We will prove those three facts by induction on $k\geq 1$. For $k=1$, we have $A_{0}=\gamma_1=1$ and $f'(3/2)=1$ so that $\theta_1=3/2-Y_1$. Since 
\begin{align*}
T_1=\frac{1}{A_{0}\gamma_1f'(\theta_{0})}\left(\left(\mathbb{Z}+]1/3,2/3[\right)-\theta_{k-1}\right)\cap [-2,-1]=&(-3/2+\mathbb{Z}+]1/3,2/3[)\cap [-2,-1]\\
=&]-7/6,-1]\cup [-2,-11/6[,
\end{align*}
$\ell(T_1)\geq 1/3$. On the other hand, for $Y_1\in T_1$, $\theta_1\geq 3/2+1\geq 2$.

Let us show the induction. Set $k\geq 2$ and suppose the result is true for $l\leq k-1$. Then $\theta_l\in\mathbb{Z}+[1/3,2/3]$ for all $l\leq k$, which implies that $A_{k-1}=k-1$. Hence, 
$$\theta_k=\theta_{k-1}+k^{1-\gamma}Y_kf'(\theta_{k-1}).$$
By induction, $\theta_{k-1}\geq k$, and since $f'(\theta)\geq \theta/2$ for $\theta\geq 0$,
$$T_k=(\left(b+ a\mathbb{Z}+a]1/3,2/3[\right)\cap [-2,-1]$$
with $a=1/(A_{k-1}\gamma_kf'(\theta_{k-1}))\leq \frac{2}{k^{2-\gamma}}\leq 1$ and $ b=-\theta_{k-1}/a$. We deduce by pigeonhole principle that $\ell(T_k)\geq 1/3-\frac{1}{2}\cdot\frac{a}{3}\geq 1/6$. Finally, for $Y_k\in A_k$ we have $Y_k\leq -1$ so that
$$\theta_k\geq k+\frac{1}{2}k^{2-\gamma}\geq k+1.$$
\end{proof}
By the previous result, 
$$\mathbb{P}\left[ X_1=\dots=X_n=0, Y_1\in T_1, \dots,Y_n\in T_n \right]\geq 2^{-n}\cdot 6^{-n}=12^{-n}.$$
Moreover, from what we showed previously, on this event we have for $1\leq k\leq n$
$$\theta_k=\theta_{k-1}-Y_kA_{k-1}\gamma
_kf'(\theta_{k-1})\geq \theta_{k-1}+k^{2-\gamma}/2\geq \theta_{k-1}k^{2-\gamma}/2.$$
We deduce that $\theta_n\geq \theta_{k-1}(k-1)^{1-\gamma}/3\geq (n!)^{2-\gamma}/2^n$. In particular,
$$\mathbb{E}\left[\Vert \theta_n-\theta_0\Vert^2\right]\geq (n!)^{2-\gamma}/24^n\xrightarrow[]{n\rightarrow \infty}\infty.$$

Remark that the latter result can be easily adapted to get a counter-example with $g$ as smooth as desired.
\end{appendix}

\bibliographystyle{apalike}
\bibliography{biblio_redaction_2}

\end{document}